\documentclass[graybox]{svmult}

\usepackage{mathptmx}       
\usepackage{helvet}         
\usepackage{courier}        
\usepackage{makeidx}         
\usepackage{graphicx}        
\usepackage{multicol}        
\usepackage[bottom]{footmisc}

\makeindex             

\usepackage{amsmath,amsfonts,amssymb}
\usepackage{booktabs}
\usepackage{enumerate}
\usepackage{subfigure}
\usepackage{ragged2e}
\usepackage{bm}

\DeclareMathOperator{\conv}{conv}
\DeclareMathOperator{\proj}{proj}

\newcommand{\Z}{{\mathbb  Z}}
\newcommand{\Q}{{\mathbb  Q}}
\newcommand{\C}{{\mathbb  C}}
\newcommand{\R}{{\mathbb  R}}
\let\ve=\mathbf
\newcommand\OS{Opt.}

\newtheorem{algorithm}{Algorithm}

\usepackage{ifpdf}

\newcommand\psref[1]{\textbf{(\S\,\ref{#1})}}  

\begin{document}

\title*{Nonlinear Integer Programming 
  \thanks{To appear in: 
M. J\"unger, T. Liebling, D. Naddef, G. Nemhauser,
W. Pulleyblank, G.~Reinelt, G.~Rinaldi, and L. Wolsey (eds.), 
\emph{50 Years of Integer Programming 1958--2008:
The Early Years and State-of-the-Art Surveys}, Springer-Verlag, 2009, ISBN 3540682740.
}}

\titlerunning{Nonlinear Integer Programming}
\author{Raymond Hemmecke, Matthias K\"oppe, Jon Lee and Robert Weismantel}

\institute{Raymond Hemmecke \at Otto-von-Guericke-Universit\"at Magdeburg,
  FMA/IMO, Universit\"atsplatz~2, 39106 Magdeburg, Germany, \email{hemmecke@imo.math.uni-magdeburg.de}
\and Matthias K\"oppe \at University of California,
  Davis, Dept.~of
  Mathematics, One Shields Avenue, Davis, CA, 95616, USA,
  \email{mkoeppe@math.ucdavis.edu}
\and Jon Lee \at IBM T.J. Watson Research Center,
  PO Box 218, Yorktown Heights, NY, 10598, USA,
  \email{jonlee@us.ibm.com}
\and Robert Weismantel \at Otto-von-Guericke-Universit\"at Magdeburg,
  FMA/IMO, Universit\"atsplatz~2, 39106 Magdeburg, Germany, \email{weismant@imo.math.uni-magdeburg.de}}
%
%
\maketitle

\makeatletter
\long\def\@abstract#1{\noindent\textbf{\abstractname.} #1\par
}
\makeatother

\vspace{-15ex}
\abstract{Research efforts of the past fifty years have led to a development of
linear integer programming as a mature discipline of mathematical
optimization. Such a  level of maturity has not been reached
when one considers nonlinear systems subject to integrality requirements
for the variables. This chapter is dedicated to this topic.

The primary goal is a study of a simple version of general
nonlinear integer problems, where all constraints are still linear. 
Our focus is on the computational complexity of the problem, 
which varies significantly
with the type of nonlinear objective function in
combination with the underlying combinatorial structure.  
Numerous boundary cases of complexity emerge, which
sometimes surprisingly lead even to polynomial time algorithms.

We also cover recent successful approaches for more general classes of
problems.  Though no positive theoretical efficiency results are available,
nor are they likely to ever be available, these seem to be the currently most 
successful and interesting approaches for solving practical problems.

It is our belief that the study of algorithms motivated by theoretical
considerations and those motivated by our desire to solve practical instances
should and do inform one another. So it is with this viewpoint that we present
the subject, and it is in this direction that we hope to spark further
research.  }

\clearpage
\section{Overview}
\label{s:overview}

In the past decade, \emph{nonlinear} integer programming has gained a lot of
mindshare.  Obviously many important applications demand that we be able to
handle nonlinear
objective functions and constraints.  Traditionally, nonlinear mixed-integer
programs have been handled in the context of the field of global optimization,
where the main focus is on numerical algorithms to solve nonlinear continuous optimization
problems and where integrality constraints were considered as an afterthought,
using branch-and-bound over the integer variables.  In the past few years,
however, researchers from the field of integer programming have increasingly
studied nonlinear mixed-integer programs from their point of view.
Nevertheless, this is generally considered a very young field, and most of the
problems and methods are not as well-understood or stable as in the case of
linear mixed-integer programs.

Any contemporary review of nonlinear mixed-integer programming will therefore be
relatively short-lived.  For this reason, our primary focus is on a classification of
nonlinear mixed-integer problems from the point of view of computational
complexity, presenting theory and algorithms for the efficiently solvable
cases.  The hope is that at least this part of the chapter will still be
valuable in the next few decades.  However, we also cover recent successful approaches
for more general classes of problems.
Though no positive theoretical efficiency results are available --- nor are they
likely to ever be available, these seem to be the
currently most successful and interesting approaches for
solving practical problems. It is our belief that
the study of algorithms motivated by theoretical considerations and
those motivated by our desire to solve practical instances should
and do inform one another. So it is with this viewpoint that
we present the subject, and it is in this direction that
we hope to spark further research.

Let us however also remark that the selection of the material that we discuss in
this chapter is subjective. There are topics that some researchers associate
with ``nonlinear integer programming'' that are not covered here. Among them are
pseudo-Boolean optimization, max-cut and quadratic assignment as well as general
0/1 polynomial programming. There is no doubt that these topics are
interesting, but, in order to keep this chapter
focused, we refrain from going into these topics.  Instead
we refer the interested reader to the references \cite{GoemansWilliamson95b}
on max-cut, 
\cite{buchheim-rinaldi:2007} for recent advances in general 0/1 polynomial
programming, and the excellent surveys \cite{boros-hammer:2002} on
pseudo-Boolean optimization and
\cite{Pardalos+Rendl+Wolkowicz:1994,Burkard+Cela+Pardalos+Pitsoulis:1998} on
the quadratic assignment problem.

\smallbreak

A general model of mixed-integer programming could be written as
\begin{equation}
\begin{aligned}
  \hbox{max/min}\quad & f(x_1,\dots,x_n)\\
  \hbox{s.t.}\quad & g_1(x_1,\dots,x_n) \leq 0 \\
  & \quad\vdots \\
  & g_m(x_1,\dots,x_n) \leq 0 \\
  & \mathbf x\in\R^{n_1} \times \Z^{n_2},
\end{aligned} \label{eq:nonlinear-over-nonlinear}
\end{equation}
where $f,g_1,\dots,g_m\colon \R^n\to\R$ are arbitrary nonlinear functions.
However, in parts of the chapter, we study a rather restricted model of nonlinear
integer programming, where the nonlinearity is confined to the objective
function, i.e., the following model:
\begin{equation}
\begin{aligned}
  \hbox{max/min}\quad & f(x_1,\dots,x_n)\\
  \hbox{subject to} \quad & A \mathbf x \leq \mathbf b\\
  & \mathbf x\in \R^{n_1}  \times \Z^{n_2},
\end{aligned} \label{eq:nonlinear-over-polyhedron}
\end{equation}
where $A$ is a rational matrix and $\mathbf b$ is a rational vector.
It is clear that this model is still NP-hard, and that it is much more expressive
and much harder to solve than integer linear programs. \smallbreak

We start out with a few fundamental hardness results that help
to get a picture of the complexity situation of the problem.

Even in the pure continuous case, nonlinear optimization is known to be hard.
\begin{theorem}
  Pure continuous polynomial optimization over polytopes
  ($n_2=0$) in varying dimension is NP-hard.
  Moreover, there does not exist a fully polynomial time
  approximation scheme ({\small FPTAS}) (unless
  $\mathrm{P}=\mathrm{NP}$).
\end{theorem}
Indeed the max-cut problem can be modeled as minimizing a quadratic form over
the cube $[-1,1]^n$, and thus inapproximability follows from a result by
H\aa{}stad~\cite{Hastad:inapprox97}.  On the other hand, pure continuous
polynomial optimization problems over
polytopes ($n_2=0$) can be solved in polynomial time when the dimension $n_1$
is fixed.  This follows from a much more general result on the computational
complexity of approximating the solutions to general algebraic formulae over
the real numbers by Renegar~\cite{Renegar:1992:Approximating}.

However, as soon as we add just two integer variables, we get a hard problem
again:
\begin{theorem} \label{th:deg4-dim2-hard}
  The problem of minimizing a degree-$4$ polynomial over the lattice
  points of a convex polygon is NP-hard.
\end{theorem}
This is based on the NP-completeness of the problem whether there exists a
positive integer $x<c$ with $x^2\equiv a \pmod{b}$; see
\cite{GareyJohnson79,deloera-hemmecke-koeppe-weismantel:intpoly-fixeddim}.\smallbreak

We also get hardness results that are much worse than just NP-hardness.
The negative solution of Hilbert's tenth problem by Matiyasevich
\cite{matiyasevich-1970,matiyasevich-1993}, based on earlier work by Davis,
Putnam, and Robinson, implies that nonlinear integer programming is
\emph{incomputable}, i.e., there cannot exist any general algorithm.
(It is clear that for cases where finite bounds for all variables are known,
an algorithm trivially exists.)
Due to a later strengthening of the negative result by Matiyasevich (published
in \cite{jones-1982}), there also
cannot exist any such general algorithm
for even a small fixed number of integer variables; see
\cite{deloera-hemmecke-koeppe-weismantel:intpoly-fixeddim}.
\begin{theorem}
  \label{th:polyopt-incomputable}
  The problem of minimizing a linear form over polynomial constraints in
  at most 10 integer variables is not computable by a recursive function.
\end{theorem}
Another consequence, as shown by
Jeroslow~\cite{jeroslow-1973:quadratic-ip-uncomputable},
is that even integer quadratic programming is incomputable.
\begin{theorem}
  \label{th:quadopt-incomputable}
  The problem of minimizing a linear form over quadratic constraints
    in integer variables is not computable by a recursive function.
\end{theorem}

How do we get positive complexity results and practical methods?  One way is to consider subclasses of the
objective functions and constraints.  First of all, for the problem of \emph{concave
  minimization} or \emph{convex maximization} which we
  study in Section~\ref{s:convex-max}, we can make use of the property
that optimal solutions can always be found on the boundary (actually on the
set of vertices) of the convex hull
of the feasible solutions. On the other hand, as in the pure continuous case,
\emph{convex minimization}, which we address in Section~\ref{s:convex-min}),
is much easier to handle, from both theoretical and practical viewpoints, than the general case.
Next, in Section \ref{s:general-polynomial}, we
study the general case of polynomial optimization, as well as practical
approaches for handling the important case of quadratic functions.
Finally, in Section~\ref{s:global}, we briefly describe the practical approach of
global optimization.


For each of these subclasses covered in sections
\ref{s:convex-max}--\ref{s:global}, we discuss positive complexity results,
such as \emph{polynomiality results in fixed dimension}, if
available (Sections \ref{s:convex-max-fixed-dim},
\ref{s:convex-min-fixed-dim}, \ref{s:fptas}),
including some \emph{boundary cases of complexity} in Sections \ref{s:n-fold-convex-max},
\ref{s:n-fold-min}, and \ref{s:univariate-opt}, and discuss \emph{practical
algorithms} (Sections \ref{s:reduction-to-linear}, \ref{s:outer-approx},
\ref{s:sos-programming}, \ref{s:quadratics}, \ref{s:sBB}).

We end the chapter with conclusions (Section~\ref{s:conclusions}), including a
table that summarizes the complexity situation of the problem
(Table~\ref{t:overview-nonlinear-over-nonlinear}).

\section{Convex integer maximization}
\label{s:convex-max}

\subsection{Fixed dimension}
\label{s:convex-max-fixed-dim}

Maximizing a convex function over the integer points in a polytope in fixed
dimension can be done in polynomial time.  To see this, note that the optimal
value is taken on at a vertex of the convex hull of all feasible integer
points. But when the dimension is fixed, there is only a polynomial number of
vertices, as Cook et al.~\cite{cook-hartmann-kannan-mcdiarmid-1992} showed.
\begin{theorem}
  Let $P = \{\, \mathbf x\in\R^n : A\mathbf x\leq\mathbf b\,\}$ be a rational polyhedron
  with $A\in\Q^{m\times n}$ and let $\phi$ be the largest binary encoding size
  of any of the rows of the system~$A\mathbf x\leq\mathbf b$.  Let $P^{\mathrm{I}} =
  \mathop{\mathrm{conv}}(P\cap\Z^n)$ be the integer hull of~$P$.  Then the number of vertices
  of~$P^{\mathrm{I}}$ is at most $2 m^n{(6n^2\phi)}^{n-1}$.
\end{theorem}
Moreover, Hartmann \cite{hartmann-1989-thesis} gave an algorithm for
enumerating all the vertices, which runs in polynomial time in fixed
dimension.

By using Hartmann's algorithm, we can therefore compute all the vertices of
the integer hull~$P^{\mathrm{I}}$, evaluate the convex objective function on
each of them and pick the best.  This simple method already provides a
polynomial-time algorithm.

\subsection{Boundary cases of complexity}
\label{s:n-fold-convex-max}

In the past fifteen years algebraic geometry and commutative algebra
tools have shown their exciting potential to study problems in integer
optimization (see \cite{BertsimasWeismantel05,Thomas:01} and
references therein). The presentation in this section is based on the papers
\cite{DeLoera+Hemmecke+Onn+Weismantel:08,Onn+Rothblum:04}.

The first key lemma, extending results of \cite{Onn+Rothblum:04} for
combinatorial optimization, shows that when a suitable geometric
condition holds, it is possible to efficiently reduce the convex integer maximization
problem to the solution of polynomially many linear integer
programming counterparts.  As we will see, this condition holds
naturally for a broad class of problems in variable dimension. To
state this result, we need the following terminology. A
\emph{direction} of an edge (i.e., a one-dimensional face) $e$ of a polyhedron $P$ is any
nonzero scalar multiple of $\mathbf u-\mathbf v$ with $\mathbf u,\mathbf v$ any two distinct points in
$e$.  A set of vectors \emph{covers all edge-directions of $P$} if it
contains a direction of each edge of $P$.  A \emph{linear integer
programming oracle} for matrix $A\in\Z^{m\times n}$ and vector
$\mathbf b\in\Z^m$ is one that, queried on $\mathbf w\in\Z^n$, solves the linear
integer program $\max\{\mathbf w^\top\mathbf x: A\mathbf x=\mathbf b,\ \mathbf x\in{\mathbb N}^n\}$, that is, either returns
an optimal solution $\mathbf x\in{\mathbb N}^n$, or asserts that the program is
infeasible, or asserts that the objective function $\mathbf w^\top\mathbf x$ is unbounded.

\begin{lemma}\label{EdgeDirections}
For any fixed $d$ there is a strongly polynomial oracle-time algorithm
that, given any vectors $\mathbf w_1,\dots,\mathbf w_d\in\Z^n$, matrix
$A\in\Z^{m\times n}$ and vector $\mathbf b\in\Z^m$ endowed with a linear
integer programming oracle, finite set $E\subset\Z^n$ covering all
edge-directions of the polyhedron $\mathop{\mathrm{conv}}\{\mathbf x\in{\mathbb N}^n\,:\, A\mathbf x=\mathbf b\}$, and
convex functional $c:\R^d\longrightarrow\R$ presented by a comparison
oracle, solves the convex integer program $$\max\, \{c(\mathbf w_1^\top \mathbf x,\dots,\mathbf w_d^\top
\mathbf x):\ A\mathbf x=\mathbf b,\ \mathbf x\in{\mathbb N}^n\}\ .$$
\end{lemma}

Here, \emph{solving} the program means that the algorithm
either returns an optimal solution $\mathbf x\in{\mathbb N}^n$, or asserts
the problem is infeasible, or asserts the polyhedron
$\{\mathbf x\in\R^n_+\,:\, A\mathbf x=\mathbf b\}$ is unbounded; and \emph{strongly polynomial
oracle-time} means that the number of arithmetic operations and calls to
the oracles are polynomially bounded in $m$ and $n$, and the size of the
numbers occurring throughout the algorithm is polynomially bounded in
the size of the input (which is the number of bits in the binary
representation of the entries of $\mathbf w_1,\dots,\mathbf w_d,A,\mathbf b,E$).

Let us outline the main ideas behind the proof to Lemma
\ref{EdgeDirections}, and let us point out the difficulties that one
has to overcome. Given data for a convex integer maximization problem
$\max\{c(\mathbf w_1^\top \mathbf x,\dots,\mathbf w_d^\top \mathbf x):\ A\mathbf x=\mathbf b,\ \mathbf x\in{\mathbb N}^n\}$, consider the
polyhedron $P:=\mathop{\mathrm{conv}}\{\mathbf x\in{\mathbb N}^n:A\mathbf x=\mathbf b\}\subseteq\R^n$ and its linear
transformation
$Q:=\{(\mathbf w_1^\top \mathbf x,\dots,\mathbf w_d^\top \mathbf x):\ \mathbf x\in P\}\subseteq\R^d$. Note that $P$ has typically exponentially many vertices and is not accessible
computationally. Note also that, because $c$ is convex, there is an
optimal solution $\mathbf x$ whose image $(\mathbf w_1^\top \mathbf x,\dots,\mathbf w_d^\top \mathbf x)$ is a vertex
of $Q$. So an important ingredient in the solution is to construct the
vertices of $Q$. Unfortunately, $Q$ may also have exponentially many
vertices even though it lives in a space $\R^d$ of fixed
dimension. However, it can be shown that, when the number of
{\em edge-directions} of $P$ is polynomial, the number of vertices of
$Q$ is polynomial. Nonetheless, even in this case, it is not possible
to construct these vertices directly, because the number of vertices of
$P$ may still be exponential. This difficulty can finally be overcome
by using a suitable {\em zonotope}. See
\cite{DeLoera+Hemmecke+Onn+Weismantel:08,Onn+Rothblum:04} for more
details.

Let us now apply Lemma \ref{EdgeDirections} to a broad
(in fact, \emph{universal}) class of convex integer
maximization problems. Lemma \ref{EdgeDirections} implies that these problems can be solved in polynomial time.
Given an $(r+s)\times t$ matrix $A$,
let $A_1$ be its $r\times t$ sub-matrix consisting of the first $r$
rows and let $A_2$ be its $s\times t$ sub-matrix consisting of the
last $s$ rows. Define the \emph{n-fold matrix} of $A$ to be the
following $(r+ns)\times nt$ matrix,
$$A^{(n)}:= ({\bf 1}_n\otimes A_1)\oplus(I_n \otimes A_2)=
\left(
\begin{array}{ccccc}
  A_1    & A_1    & A_1    & \cdots & A_1    \\
  A_2  & 0      & 0      & \cdots & 0      \\
  0  & A_2      & 0      & \cdots & 0      \\
  \vdots & \vdots & \ddots & \vdots & \vdots \\
  0  & 0      & 0      & \cdots & A_2      \\
\end{array}
\right).
$$
Note that $A^{(n)}$ depends on $r$ and $s$: these will be indicated by
referring to $A$ as an ``$(r+s)\times t$ matrix.''

We establish the following theorem, which asserts that
convex integer maximization over $n$-fold systems of a fixed matrix $A$,
in variable dimension $nt$, are solvable in polynomial time.

\begin{theorem}\label{Theorem: Main n-fold}
For any fixed positive integer $d$ and fixed $(r+s)\times t$ integer
matrix $A$ there is a polynomial oracle-time algorithm that,
given $n$, vectors $\mathbf w_1,\dots,\mathbf w_d\in\Z^{nt}$ and $\mathbf b\in\Z^{r+ns}$,
and convex function $c:\R^d\longrightarrow\R$ presented by a
comparison oracle, solves the convex n-fold integer maximization problem
$$\max\, \{c(\mathbf w_1^\top \mathbf x,\dots,\mathbf w_d^\top \mathbf x):\ A^{(n)}\mathbf x=\mathbf b,\ \mathbf x\in{\mathbb N}^{nt}\}\ .$$
\end{theorem}

The equations defined by an n-fold matrix have the following,
perhaps more illuminating, interpretation: splitting the variable vector
and the right-hand side vector into components of suitable sizes,
$\mathbf x=(\mathbf x^1,\dots,\mathbf x^n)$ and $\mathbf b=(\mathbf b^0,\mathbf b^1,\dots,\mathbf b^n)$, where $\mathbf b^0\in\Z^r$
and $\mathbf x^k\in{\mathbb N}^t$ and $\mathbf b^k\in \Z^s$ for $k=1,\dots,n$, the equations
become $A_1(\sum_{k=1}^n \mathbf x^k)=\mathbf b^0$ and $A_2\mathbf x^k=\mathbf b^k$ for $k=1,\dots,n$.
Thus, each component $\mathbf x^k$ satisfies a system of constraints defined
by $A_2$ with its own right-hand side $\mathbf b^k$, and the sum
$\sum_{k=1}^n \mathbf x^k$ obeys constraints determined by $A_1$ and $\mathbf b^0$
restricting the ``common resources shared by all components''.

Theorem \ref{Theorem: Main n-fold} has various applications, including
multiway transportation problems, packing problems, vector
partitioning and clustering. For example, we have the following
corollary providing the first polynomial time solution of convex
$3$-way transportation.

\begin{corollary}\label{Threeway}{\bf (convex 3-way transportation)}
For any fixed $d,p,q$ there is a polynomial oracle-time algorithm that,
given $n$, arrays $\mathbf w_1,\dots,\mathbf w_d\in\Z^{p\times q\times n}$,
$\mathbf u\in\Z^{p\times q}$, $\mathbf v\in\Z^{p\times n}$, $\mathbf z\in\Z^{q\times n}$,
and convex $c:\R^d\longrightarrow\R$ presented by comparison oracle,
solves the convex integer 3-way transportation problem
$$\max\{\,c(\mathbf w_1^\top \mathbf x, \dots, \mathbf w_d^\top \mathbf x)
 \ :\ \mathbf x\in{\mathbb N}^{p\times q\times n}\,,\ \sum_i x_{i,j,k}=z_{j,k}
\,,\ \sum_j x_{i,j,k}=v_{i,k}\,,\ \sum_k x_{i,j,k}=u_{i,j}\,\}\ .$$
\end{corollary}

Note that in contrast, if the dimensions of two sides of the tables
are variable, say, $q$ and $n$, then already the standard \emph{linear}
integer $3$-way transportation problem over such tables is NP-hard,
see \cite{DeLoera+Onn:2004,DeLoera+Onn:2006a,DeLoera+Onn:2006b}.

In order to prove Theorem \ref{Theorem: Main n-fold}, we need to
recall some definitions. The \emph{Graver basis} of an integer matrix
$A\in\Z^{m\times n}$, introduced in \cite{Graver:75}, is a canonical
finite set ${\cal G}(A)$ that can be defined as follows. For
each of the $2^n$ orthants $\mathcal{O}_j$ of $\R^n$ let $H_j$ denote the
inclusion-minimal Hilbert basis of the pointed rational polyhedral cone
$\ker(A)\cap\mathcal{O}_j$. Then the Graver basis $\mathcal{G}(A)$ is defined
to be the union $\mathcal{G}(A)=\cup_{i=1}^{2^n} H_j\setminus\{\mathbf0\}$ over
all these $2^n$ Hilbert bases. By this definition, the Graver basis
$\mathcal{G}(A)$ has a nice representation property: every
$\mathbf z\in\ker(A)\cap\Z^n$ can be written as a sign-compatible nonnegative
integer linear combination $\mathbf z=\sum_i \alpha_i \mathbf g_i$ of Graver basis
elements $\mathbf g_i\in\mathcal{G}(A)$. This follows from the simple observation
that $\mathbf z$ has to belong to some orthant $\mathcal{O}_j$ of $\R^n$ and thus
it can be represented as a sign-compatible nonnegative integer linear
combination of elements in $H_j\subseteq\mathcal{G}(A)$. For more details
on Graver bases and the currently fastest procedure for computing them
see \cite{Sturmfels96,Hemmecke:2003b,4ti2}.

Graver bases have another nice property: They contain all edge
directions in the integer hulls within the polytopes $P_{\mathbf b}=\{\,\mathbf x: A\mathbf x=\mathbf b,\, \mathbf x
\geq \mathbf 0\,\}$ as $\mathbf b$ is varying. We include a direct proof here.

\begin{lemma}\label{GraverEdgeDirections} For
every integer matrix $A\in\Z^{m\times n}$ and every integer vector
$\mathbf b\in{\mathbb N}^m$, the Graver basis $\mathcal{G}(A)$ of $A$ covers all edge-directions
of the polyhedron $\mathop{\mathrm{conv}}\{\,\mathbf x\in{\mathbb N}^n : A\mathbf x=\mathbf b\,\}$ defined by $A$ and $\mathbf b$.
\end{lemma}
\begin{proof} Consider any edge $e$ of $P:=\mathop{\mathrm{conv}}\{\,\mathbf x\in{\mathbb N}^n:A\mathbf x=\mathbf b\,\}$ and
pick two distinct points $\mathbf u,\mathbf v\in e\cap{\mathbb N}^n$. Then $\mathbf g:=\mathbf u-\mathbf v$ is in
$\ker(A)\cap\Z^n\setminus\{\mathbf0\}$ and hence, by the representation property of the
Graver basis $\mathcal{G}(A)$, $\mathbf g$ can be written as a finite
\emph{sign-compatible} sum $\mathbf g=\sum \mathbf g^i$ with $\mathbf g^i\in\mathcal{G}(A)$ for
all $i$. Now, we claim that $\mathbf u-\mathbf g^i\in P$ for all $i$. To see this,
note first that $\mathbf g^i\in\mathcal{G}(A)\subseteq\ker(A)$ implies $A\mathbf g^i=\mathbf0$ and
hence $A(\mathbf u-\mathbf g^i)=A\mathbf u=\mathbf b$; and second, note that $\mathbf u-\mathbf g^i\geq \mathbf0$: indeed, if
$g^i_j\leq 0$ then $u_j-g^i_j\geq u_j\geq 0$; and if $g^i_j>0$ then
sign-compatibility implies $g^i_j\leq g_j$ and therefore $u_j-g^i_j\geq
u_j-g_j=v_j\geq 0$.

Now let $\mathbf w\in\R^n$ be a linear functional uniquely maximized over $P$
at the edge $e$. Then for all $i$, as just proved, $\mathbf u-\mathbf g^i\in P$
and hence $\mathbf w^\top \mathbf g^i\geq 0$. But $\sum \mathbf w^\top \mathbf g^i= \mathbf w^\top \mathbf g=\mathbf w^\top \mathbf u-\mathbf w^\top \mathbf v=0$,
implying that in fact, for all $i$, we have $\mathbf w^\top \mathbf g^i=0$ and therefore
$\mathbf u-\mathbf g^i\in e$. This implies that each $\mathbf g^i$ is a direction of the edge
$e$ (in fact, moreover, all $\mathbf g^i$ are the same, so $\mathbf g$ is a
multiple of some Graver basis element).
\end{proof}

In Section \ref{s:n-fold-min}, we show that for fixed matrix $A$, the size of the Graver basis of $A^{(n)}$ increases only polynomially in $n$ implying Theorem \ref{Theorem: n-fold convex minimization} that states that certain integer convex $n$-fold \emph{minimization} problems can be solved in polynomial time when the matrix $A$ is kept fixed. As a special case, this implies that the integer \emph{linear} $n$-fold \emph{maximization} problem can be solved in polynomial time when the matrix $A$ is kept fixed. Finally, combining these results with Lemmas \ref{EdgeDirections} and \ref{GraverEdgeDirections}, we can now prove Theorem \ref{Theorem: Main n-fold}.

\begin{proof}[of Theorem \ref{Theorem: Main n-fold}]
The algorithm underlying Proposition \ref{Theorem: n-fold convex minimization} provides a polynomial time
realization of a linear integer programming oracle for $A^{(n)}$ and $\mathbf b$.
The algorithm underlying Proposition \ref{GraverComputation} allows
to compute the Graver basis $\mathcal{G}(A^{(n)})$ in time which is polynomial
in the input. By Lemma \ref{GraverEdgeDirections},
this set $E:=\mathcal{G}(A^{(n)})$ covers all edge-directions of the polyhedron
$\mathop{\mathrm{conv}}\{\,\mathbf x\in{\mathbb N}^{nt} : A^{(n)}\mathbf x=\mathbf b\,\}$ underlying the convex integer program.
Thus, the hypothesis of Lemma \ref{EdgeDirections} is satisfied and hence
the algorithm underlying Lemma \ref{EdgeDirections} can be used to
solve the convex integer maximization problem in polynomial time.
\end{proof}

\subsection{Reduction to linear integer programming}
\label{s:reduction-to-linear}

In this section it is our goal to develop a basic understanding about
when a discrete polynomial programming problem can be tackled with
classical linear integer optimization techniques. We begin to study
polyhedra related to polynomial integer programming. The presented
approach applies to problems of the kind
\[
  \max \{f(\mathbf x) : A\mathbf x= \mathbf b,\; \mathbf x\in\Z_+^n\}
\]
with convex polynomial function $f$, as well as to models such as
\begin{equation}
  \label{IPP}
  \max \{ \mathbf c^\top \mathbf x :  \mathbf x \in K_I \},
\end{equation}
where $K_I$ denotes the set of all integer points of a compact
\emph{basic-closed} semi-algebraic set $K$ described by a finite
number of polynomial inequalities, i.~e.,
\begin{equation*}
  \label{semi-algebraic_set}
  K = \{ \mathbf x \in \R^n : p_i(\mathbf x) \leq 0,\; i\in M,\;\; \mathbf l \leq \mathbf x \leq \mathbf u \}.
\end{equation*}
We assume that $p_i \in \Z[\mathbf x]:=\Z[x_1,\ldots,x_n]$, for all $i \in M = \{1,\ldots,m\}$, and $\mathbf l,\mathbf u \in \Z^n$.

One natural idea is to derive a linear description of the convex
hull of $K_I$. Unfortunately, $\mathop{\mathrm{conv}}(K_I)$ might contain interior
integer points that are not elements of $K$, see Figure~\ref{fig:x1x2leq1}.
\begin{figure}[th]
  \begin{center}
    \ifpdf
    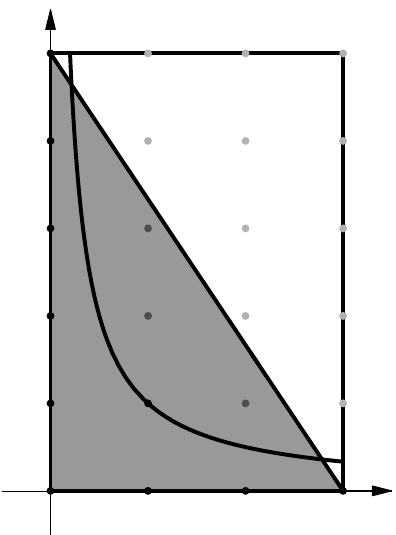
    \else
    \input{fig-x1x2leq1.pstex_t}
    \fi
    \caption{\small
      Let
      $K = \{ \mathbf x \in \R^2\; |\;
      x_1 x_2 - 1 \leq 0,\;
      0 \leq x_1 \leq 3,\;
      0 \leq x_2 \leq 5\}$.
      The point $(1,2)$ is contained in $\mathop{\mathrm{conv}}(K_I)$. But it violates
      the constraint $x_1x_2-1 \leq 0$.
    }
    \label{fig:x1x2leq1}
  \end{center}
\end{figure}
On the other hand, if a description of $\mathop{\mathrm{conv}}(K_I)$ is at hand, then the
mathematical optimization problem of solving (\ref{IPP}) can be reduced to
optimizing the linear objective function $\mathbf c^\top \mathbf x$ over
the polyhedron $\mathop{\mathrm{conv}}(K_I)$. This is our first topic of interest. In what follows we denote for a set $D$ the projection of $D$ to a set of variables $x$ by the symbol $D_x$.

\begin{definition}
For polynomials $p_1,\ldots,p_m:\; \Z^n \to \Z$ we define the
polyhedron associated with the vector $\mathbf p(\mathbf x)=(p_1(\mathbf x),\ldots,p_m(\mathbf x))$ of
polynomials as
\[
  P_{\mathbf p} = \mathop{\mathrm{conv}} \Big( \big\{ \big(\mathbf x, \mathbf p(\mathbf x) \big)\in\R^{n+m}\;
		\big|\; \mathbf x \in [\mathbf l,\mathbf u] \cap \Z^n \big\}\Big).
\]
\end{definition}

The significance of the results below is that they allow us to
reformulate the nonlinear integer optimization problem
$\max\{f(\mathbf x): A\mathbf x = \mathbf b,\; \mathbf x \in \Z_+^n\}$ as the  optimization problem
$\min\{\pi: A\mathbf x = \mathbf b,\; f(\mathbf x) \leq \pi,\; \mathbf x \in \Z_+^n\}$.
This in turn has the same objective function value as the
linear integer program:
$\min\{\pi: A\mathbf x = \mathbf b,\; (\mathbf x,\pi) \in P_f,\; \mathbf x \in \Z_+^n\}$.
In this situation, an H-description or V-description of the polyhedron $P_f$
is sufficient to reduce
the original nonlinear optimization problem to a linear integer
problem.

\begin{proposition}
  \label{prop:linrelaxofKI}
   For a vector of polynomials $\mathbf p \in \Z[\mathbf x]^m$, let
  \[
    K_I = \{ \mathbf x \in \Z^n : \mathbf p(\mathbf x) \leq \mathbf 0,\;
    \mathbf l \leq \mathbf x \leq \mathbf u \}.
  \]
  Then,
  \begin{equation}
    \label{eq:linrelaxofKI}
    \mathop{\mathrm{conv}}(K_I) \subseteq
    \big( P_{\mathbf p}  \cap \{ (\mathbf x, {\bm{\pi}}) \in \R^{n+m}\; \big|\; {\bm{\pi}} \leq \mathbf0\}\big)_{\mathbf x}.
  \end{equation}
\end{proposition}

\begin{proof}
  It suffices to show that
  \begin{math}
    K_I \subseteq
    \big(P_{\mathbf p}  \cap \{ (\mathbf x, {\bm{\pi}}) \in \R^{n+m} : {\bm{\pi}} \leq \mathbf0\}
    \big)_{\mathbf x}
  \end{math}.
  Let us consider $\mathbf x \in K_I \subseteq [\mathbf l,\mathbf u] \cap \Z^n$. By definition,
  $\big(\mathbf x,\mathbf p(\mathbf x) \big) \in P_{\mathbf p}$. Moreover, we have $\mathbf p(\mathbf x) \leq \mathbf0$. This implies
  \begin{math}
    \big(\mathbf x, \mathbf p(\mathbf x) \big) \in
    \big( P_{\mathbf p}  \cap \{ (\mathbf x,{\bm{\pi}}) \in \R^{n+m} : {\bm{\pi}} \leq \mathbf0 \} \big),
  \end{math}
  and thus,
  \[
    \mathbf x \in
    \big(
    P_{\mathbf p} \cap \{ (\mathbf x, {\bm{\pi}}) \in \R^{n+m} : {\bm{\pi}} \leq \mathbf0 \}
    \big)_{\mathbf x}.
  \]
\end{proof}

Note that even in the univariate case, equality in
Formula~(\ref{eq:linrelaxofKI}) of Proposition~\ref{prop:linrelaxofKI}
does not always hold. For instance, if
\begin{math}
  K_I = \{ x \in \Z : x^2 - 5 \leq 0,\; -3 \leq x \leq 5 \}
\end{math},
then
\begin{displaymath}
  \mathop{\mathrm{conv}}(K_I) = [-2,2] \neq
  [-2.2, 2.2] \subseteq
  \big(P_{x^2}\; \cap\; \{ (x, \pi) \in \R^2 : \pi - 5 \leq 0 \}
  \big)_x. \end{displaymath}
Although even in very simple cases the sets $\mathop{\mathrm{conv}}(K_I)$ and
\begin{math}
    \big(
    P_{\mathbf p}  \cap \{ (\mathbf x,{\bm{\pi}}) :  {\bm{\pi}} \leq \mathbf0 \}
    \big)_{\mathbf x}
\end{math}
differ, it is still possible that the integer points in $K_I$ and
\begin{math}
  \big(P_{\mathbf p}  \cap \{ (\mathbf x,{\bm{\pi}}) : {\bm{\pi}} \leq \mathbf0 \} \big)_{\mathbf x}
\end{math}
are equal. In our example with
$K_I = \{ x \in \Z : x^2 - 5 \leq 0,\; -3 \leq x \leq 5 \}$,
we then obtain,
\begin{displaymath}
  K_I = \{-2,-1,0,1,2\} = [-2.2, 2.2] \cap \Z.
\end{displaymath}

Of course, for any $\mathbf p \in \Z[\mathbf x]^m$ we have that
  \begin{equation}
    \label{proposfund}
    K_I = \{ \mathbf x \in \Z^n : \mathbf p(\mathbf x) \leq \mathbf0,\;
    \mathbf l \leq \mathbf x \leq \mathbf u \}
    \quad \subseteq\quad
    \big(
    P_{\mathbf p}  \cap \{ (\mathbf x, {\bm{\pi}}) : {\bm{\pi}} \leq \mathbf0\}
    \big)_{\mathbf x} \cap \Z^n.
  \end{equation}
The key question here is when equality holds in Formula~(\ref{proposfund}).

\begin{theorem}
  \label{th:charofKI}
  Let $\mathbf p \in \Z[\mathbf x]^m$ and $K_I = \{\mathbf x \in \Z^n : \mathbf p(\mathbf x) \leq \mathbf0,\;
  \mathbf l \leq \mathbf x \leq \mathbf u\}$.
  Then,
  \begin{displaymath}
    K_I \;=\;
    \big( P_{\mathbf p}  \cap \{ (\mathbf x, {\bm{\pi}}) : {\bm{\pi}} \leq \mathbf0\} \big)_{\mathbf x}
    \; \cap\; \Z^n
  \end{displaymath}
  holds if every  polynomial
  $p' \in \{ p_i : i =1,\ldots,m \}$ satisfies
  the following condition
  \begin{equation}
    \label{eq:iff-condition}
    p' \Big( \sum_{\mathbf k} \lambda_{\mathbf k} \mathbf k \Big)\; -\;
    \sum_{\mathbf k} \lambda_{\mathbf k} p' (\mathbf k) \quad <\quad  1,
  \end{equation}
 for all
  \begin{math}
    \lambda_{\mathbf k} \geq 0,\; \mathbf k \in [\mathbf l,\mathbf u] \cap \Z^n,\;
    \sum_{\mathbf k} \lambda_{\mathbf k} = 1
    \mbox{ and }
    \sum_{\mathbf k} \lambda_{\mathbf k} \mathbf k \in \Z^n.
  \end{math}
\end{theorem}

\begin{proof}
  Using Formula~(\ref{proposfund}), we have to show that
  \begin{math}
    \big( P_{\mathbf p}  \cap \{ (\mathbf x,{\bm{\pi}}) : {\bm{\pi}} \leq \mathbf0\}
    \big)_{\mathbf x}
    \cap \Z^n
    \subseteq K_I
  \end{math}
  if all $p_i$, $i \in \{1,\ldots,m\}$, satisfy
  Equation~(\ref{eq:iff-condition}).
  Let
  \begin{math}
    \mathbf x \in
    \big( P_{\mathbf p}  \cap \{ (\mathbf x, {\bm{\pi}}) \in \R^{m+n} : {\bm{\pi}} \leq \mathbf0\}
    \big)_{\mathbf x} \cap \Z^n.
  \end{math}
  Then, there exists a ${\bm{\pi}} \in \R^m$ such that
  \begin{displaymath}
    (\mathbf x, {\bm{\pi}}) \in
    P_{\mathbf p}  \cap \{ (\mathbf x,{\bm{\pi}}) \in \R^{n+m} : {\bm{\pi}} \leq 0 \}.
  \end{displaymath}
  By definition, ${\bm{\pi}}\leq \mathbf0$. Furthermore, there must exist nonnegative real numbers
  $\lambda_{\mathbf k} \geq 0$, $\mathbf k \in [\mathbf l,\mathbf u]\cap \Z^n$, such that $\sum_{\mathbf k} \lambda_{\mathbf k} = 1$ and
  $(\mathbf x,{\bm{\pi}}) = \sum_{\mathbf k} \lambda_{\mathbf k} (\mathbf k,\mathbf p(\mathbf k))$.
  Suppose that there exists an index $i_0$ such that
the inequality $p_{i_0}(\mathbf x)\leq 0$ is violated.
  The fact that $p_{i_0} \in \Z[\mathbf x]$ and $\mathbf x \in \Z^n$, implies that
  $p_{i_0}(\mathbf x) \geq 1$. Thus, we obtain
  \begin{equation*}
    \sum_{\mathbf k} \lambda_{\mathbf k} p_{i_0}(\mathbf k) = \pi_{i_0}
    \leq  0 < 1 \leq p_{i_0}(\mathbf x) =
    p_{i_0} \Big( \sum_{\mathbf k} \lambda_{\mathbf k} \mathbf k \Big),
  \end{equation*}
  or equivalently,
  \begin{math}
    p_{i_0} \big( \sum_{\mathbf k} \lambda_{\mathbf k} \mathbf k \big) - \sum_{\mathbf k}
    \lambda_{\mathbf k}  p_{i_0}(\mathbf k) \geq   1.
  \end{math}
Because this is a contradiction to our assumption, we have that $p_i(\mathbf x) \leq 0$ for all $i$. Hence, $\mathbf x \in K_I$.
  This proves the claim.
\end{proof}

The next example illustrates the statement of
Theorem~\ref{th:charofKI}.
\begin{example}
  \label{ex:cps}
  Let $p \in \Z[\mathbf x]$, $\mathbf x \mapsto p(\mathbf x):= 3 x_1^2 + 2 x_2^2-19$. We consider the semi-algebraic set
  \begin{equation*}
    K = \{ \mathbf x \in \R^2\; |\; p(\mathbf x) \leq 0,\;
    0 \leq x_1 \leq 3,\; 0 \leq x_2 \leq 3 \}
    \mbox{\quad and\quad }
    K_I = K \cap \Z^2.
  \end{equation*}
  It turns out that the convex hull of $K_I$ is described by
  $x_1 + x_2 \leq 3$, $0 \leq x_1 \leq 2$ and $0 \leq x_2$.
  Notice that the poly\-nomial~$p$ is convex.
  This condition ensures that $p$ satisfies
  Equation~(\ref{eq:iff-condition}) of Theorem~\ref{th:charofKI}.
  We obtain in this case
  \begin{equation*}
    \begin{array}{llllll}
      \big(
      P_p  \cap \big\{ (\mathbf x,\pi) \in \R^3 :  \pi \leq 0\}
      \big)_{\mathbf x}
     & = \;  \{ & \mathbf x \in \R^2\; : & 9 x_1 + 6 x_2   \leq 29,\;\\
      && & 3 x_1 + 10 x_2  \leq 31,\;\\
      && & 9 x_1 + 10 x_2  \leq 37,\\
      && & 15x_1 + 2 x_2   \leq 37,\;\\
      && & 15x_1 + 6 x_2  \leq 41,\;\\
      && &  -x_1 \leq 0,\; 0 \leq x_2  \leq 3 \big\}.
    \end{array}
  \end{equation*}
  The sets $K$, $\mathop{\mathrm{conv}}(K_I)$ and
  $\big(P_p  \cap \{ (\mathbf x,\pi)\in \R^3 : \pi \leq 0\} \big)_{\mathbf x}$
  are illustrated in Figure \ref{fig:ex_cps}.
  Note that here $K_I  = \big(P_p  \cap \{ (\mathbf x,\pi) : \pi \leq 0\} \big)_{\mathbf x} \cap \Z^2$.
\end{example}

\begin{figure}[th]
  \centering
  \ifpdf
    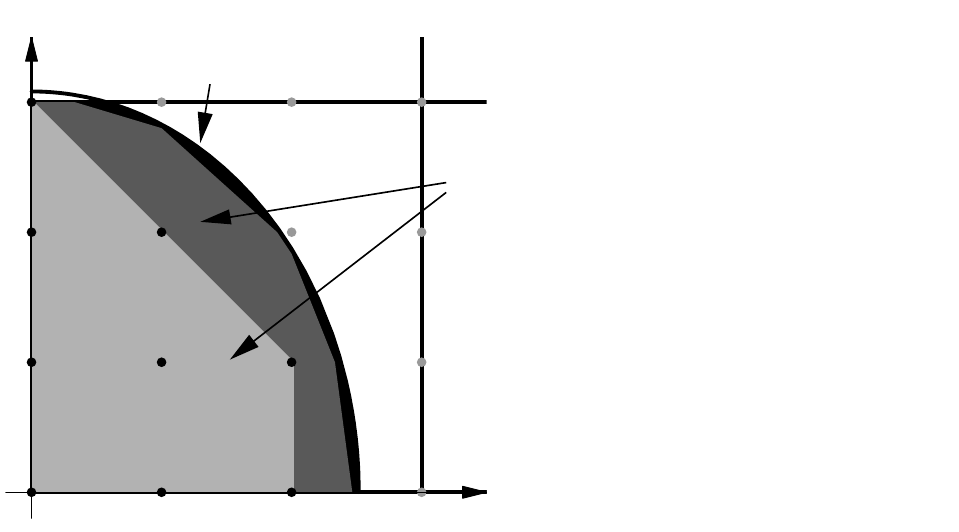
    \else
    \input{fig-intconv-proj.pstex_t}
    \fi
  \caption{\small  Illustration of Theorem \ref{th:charofKI} in Example
    \ref{ex:cps}.}
  \label{fig:ex_cps}
\end{figure}

\bigskip
Next we introduce a large class of
nonlinear functions for which one can ensure that equality holds in
Formula~(\ref{proposfund}).

\begin{definition}
  \label{def:integer-convex}
  \emph{(Integer-Convex Polynomials)}
   \\
  A polynomial $f\in\Z[\mathbf x]$ is said to be
  \emph{integer-convex} on $[\mathbf l,\mathbf u]\subseteq\R^n$,
  if for any finite subset of nonnegative real numbers
  \begin{math}
    \{ \lambda_{\mathbf k} \}_{\mathbf k \in [\mathbf l,\mathbf u] \cap \Z^n} \subseteq \R_+
  \end{math}
  with
  \begin{math}
    \sum_{\mathbf k \in [\mathbf l,\mathbf u] \cap \Z^n} \lambda_{\mathbf k} = 1
  \end{math}
  and
  \begin{math}
    \sum_{\mathbf k \in [\mathbf l,\mathbf u] \cap \Z^n} \lambda_{\mathbf k} \mathbf k \in [\mathbf l,\mathbf u] \cap \Z^n,
  \end{math}
  the following inequality holds:
  \begin{equation}
    \label{eq:jensen}
    f \Big( \sum_{\mathbf k \in [\mathbf l,\mathbf u] \cap \Z^n} \lambda_{\mathbf k} \mathbf k \Big)
    \quad \leq\quad \sum_{\mathbf k \in [\mathbf l,\mathbf u] \cap \Z^n} \lambda_{\mathbf k} f(\mathbf k).
  \end{equation}
  If (\ref{eq:jensen}) holds strictly for all
  $\{ \lambda_{\mathbf k} \}_{\mathbf k \in [\mathbf l,\mathbf u] \cap \Z^n} \subseteq \R_+$,
  and $\mathbf x \in [\mathbf l,\mathbf u] \cap \Z^n$ such that
  $\sum_{\mathbf k} \lambda_{\mathbf k} = 1$, $\mathbf x =\sum_{\mathbf k} \lambda_{\mathbf k} \mathbf k$, and
  $\lambda_{\mathbf x} < 1$, then the polynomial $f$ is called \emph{strictly integer-convex} on $[\mathbf l,\mathbf u]$.
\end{definition}

  By definition, a (strictly) convex polynomial is (strictly)
  in\-te\-ger-con\-vex. Conversely, a (strictly) integer-convex polynomial is not necessarily (strictly)
  convex. Figure~\ref{fig:integer-convex} gives an example.

\begin{figure}[th]
\begin{center}
  \ifpdf
    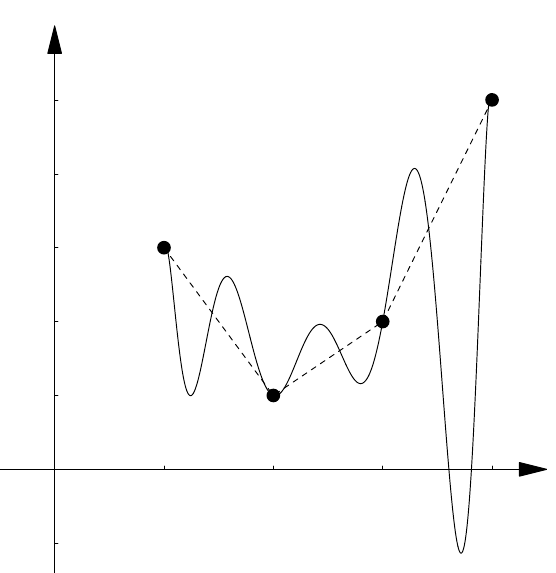
    \else
    \input{fig-intconv-polynomial.pstex_t}
    \fi
\caption{\small The graph of an integer-convex polynomial $p$ on $[1,4]$.} \label{fig:integer-convex}
\end{center}
\end{figure}

Integer convexity is inherited under taking conic combinations and
applying a composition rule.
\begin{enumerate}[(a)]
  \item
    For any finite number of integer-convex polynomials
    \begin{math}
      f_s \in \Z[\mathbf x],\; s \in \{1,...,t\},
    \end{math}
    on $[\mathbf l,\mathbf u]$, and nonnegative integers
    \begin{math}
      a_s \in \Z_+,\; s \in \{1,\ldots,t\},
    \end{math}
    the polynomial $f \in \Z[\mathbf x]$,
    \begin{math}
      \mathbf x \mapsto f(\mathbf x) := \sum_{s=1}^t a_s f_s(\mathbf x),
    \end{math}
    is integer-convex on $[\mathbf l,\mathbf u]$.
  \item
    Let $\mathbf l,\mathbf u \in \Z^n$ and let $h \in \Z[\mathbf x]$,
    \begin{math}
      \mathbf x \mapsto h(\mathbf x):= \mathbf c^\top \mathbf x + \gamma,
    \end{math}
    be a linear function. Setting $W=\big\{\mathbf c^\top \mathbf x+\gamma : \mathbf x \in [\mathbf l,\mathbf u]\big\}$, for every integer-convex univariate polynomial $q \in \Z[\mathbf w]$, the function
    \begin{math}
      p \in \Z[\mathbf x],\; \mathbf x \mapsto p(\mathbf x):= q(h(\mathbf x))
    \end{math}
    is integer-convex on $[\mathbf l,\mathbf u]$.
  \end{enumerate}

Indeed, integer-convex polynomial functions capture a lot of
combinatorial structure. In particular, we can characterize the set of
all vertices in an associated polyhedron. Most importantly, if $f$ is
integer-convex on $[\mathbf l,\mathbf u]$, then this ensures that for any integer
point $\mathbf x \in [\mathbf l,\mathbf u]$ the value~$f(\mathbf x)$ is not underestimated by all
$\pi\in\R$ with $(\mathbf x,\pi)\in P_{f}$, where $P_f$ is
the polytope associated with the graph of the polynomial $f\in\Z[\mathbf x]$.
\begin{theorem}
  \label{prop:decomposition_of_icfct}
For a polynomial $f \in \Z[\mathbf x]$ and $\mathbf l, \mathbf u \in \Z^n$, $\mathbf l+\mathbf 1 <\mathbf u$, let
$$
    P_{f} = \mathop{\mathrm{conv}}
    \Big(
    \big\{
    \big(\mathbf x, f(\mathbf x) \big) \in \Z^{n+1}\; \big|\;
    \mathbf x \in [\mathbf l,\mathbf u] \cap \Z^n
    \big\}
    \Big).
$$
Then, $f$ is integer-convex on $[\mathbf l,\mathbf u]$ is equivalent to the condition
that for all $(\mathbf x,\pi) \in P_{f},\; \mathbf x \in \Z^n$  we have that $f(\mathbf x)
\leq \pi$. Moreover, if  $f$ is strictly integer-convex on $[\mathbf l,\mathbf u]$,
then for every $\mathbf x \in [\mathbf l,\mathbf u] \cap\Z^n$, the point $(\mathbf x, f(\mathbf x))$ is a
vertex of $P_f$.
\end{theorem}

\begin{proof}
First let us assume that $f$ is integer-convex on $[\mathbf l,\mathbf u]$. Let
$(\mathbf x,\pi) \in P_{f}$ such that $\mathbf x \in \Z^n$. Then, there exist
nonnegative real numbers
  \begin{math}
    \{\lambda_{\mathbf k}\}_{\mathbf k \in [\mathbf l,\mathbf u] \cap \Z^n} \subseteq \R_+,
  \end{math}
  \begin{math}
    \sum_{\mathbf k} \lambda_{\mathbf k} = 1,
  \end{math}
  such that
  \begin{math}
    ( \mathbf x, \pi) = \sum_{\mathbf k} \lambda_{\mathbf k} \left( \mathbf k, f(\mathbf k) \right)
  \end{math}.
  It follows that
  \begin{displaymath}
    f(\mathbf x) = f \Big( \sum_{\mathbf k} \lambda_{\mathbf k} \mathbf k \Big) \quad \leq\quad
    \sum_{\mathbf k} \lambda_{\mathbf k} f (\mathbf k) = \pi.
  \end{displaymath}

Next we assume that $f$ is not integer-convex on $[\mathbf l,\mathbf u]$.
Then, there exists a subset of nonnegative real numbers
  \begin{math}
    \{ \lambda_{\mathbf k}\}_{\mathbf k \in [\mathbf l,\mathbf u] \cap \Z^n} \subseteq \R_+
  \end{math}
  with
  \begin{math}
    \sum_{\mathbf k} 
    \lambda_{\mathbf k} = 1
  \end{math}
  such that
  \begin{equation*}
    \mathbf x := \sum_{\mathbf k} 
    \lambda_{\mathbf k} \mathbf k \in [\mathbf l,\mathbf u] \cap \Z^n
    \mbox{ and }
    \pi :=\sum_{\mathbf k} \lambda_{\mathbf k} f (\mathbf k) <
    f \Big( \sum_{\mathbf k} \lambda_{\mathbf k} \mathbf k \Big) = f(\mathbf x).
  \end{equation*}
  But then,
  \begin{math}
    (\mathbf x, \pi) = \sum_{\mathbf k} \lambda_{\mathbf k} (\mathbf k,f(\mathbf k)) \in P_{f}
  \end{math}
  violates the inequality
  \begin{math}
    f(\mathbf x) \leq \pi.
  \end{math}
This is a contradiction to the assumption.

If $f$ is strictly integer-convex on $[\mathbf l,\mathbf u]$, then for each
$\mathbf x\in  [\mathbf l,\mathbf u] \cap \Z^n$, we have that
\[
  f(\mathbf x) < \sum_{\mathbf k \in [\mathbf l,\mathbf u] \cap \Z^n\backslash \{\mathbf x\}}  \lambda_{\mathbf k} f(\mathbf k),
\]
for all $\lambda_{\mathbf k} \in \R_+$, $\mathbf k\in [\mathbf l,\mathbf u]\cap\Z^n\setminus\{\mathbf x\}$, with
\begin{math}
    \mathbf x = \sum_{\mathbf k} \lambda_{\mathbf k} \mathbf k
\end{math}
and $\sum_{\mathbf k} \lambda_{\mathbf k} =1$. Thus, every point $\big(\mathbf x, f(\mathbf x)\big)$,
$\mathbf x \in [\mathbf l,\mathbf u] \cap \Z^n$, is a vertex of $P_{f}$.
\end{proof}

\section{Convex integer minimization}
\label{s:convex-min}

The complexity of the case of convex integer minimization is set apart from
the general case of integer polynomial optimization by the existence of
bounding results for the coordinates of optimal solutions.  Once a finite
bound can be computed, it is clear that an algorithm for minimization exists.
Thus the fundamental incomputability result for integer polynomial
optimization (Theorem \ref{th:polyopt-incomputable}) does not apply to the case of
convex integer minimization.

The first bounds for the optimal integer solutions to convex minimization
problems were proved by \cite{khachiyan:1983:polynomial-programming,
tarasov-khachiyan-1980}.  We present the sharpened bound that was
obtained by
\cite{bank-krick-mandel-dolerno-1991,bank-heintz-krick-mandel-solerno-1993}
for the more general case of quasi-convex polynomials.  This bound is a
consequence of an efficient theory of quantifier elimination over the reals;
see \cite{Renegar:1992:CCGc}.
\begin{theorem}\label{th:encoding-length-convex-optimum}
  Let $f, g_1,\dots,g_m \in\Z[x_1,\dots,x_n]$ be quasi-convex polynomials
  of degree at most~$d\geq2$, whose coefficients have a binary encoding length
  of at most~$\ell$.  Let
  \begin{displaymath}
    F = \bigl\{\, \mathbf x\in\R^n : g_i(\mathbf x) \leq  0\quad
    \text{for $i=1,\dots,m$} \,\bigr\}
  \end{displaymath}
  be the (continuous) feasible region.  If the integer minimization
  problem
  \begin{math}
    \min\{\, f(\mathbf x): \mathbf x\in F\cap\Z^n \,\}
  \end{math}
  is bounded, there
  exists a radius~$R\in\Z_+$ of binary encoding length at most $(md)^{\mathrm
    O(n)} \ell$ such that
  \begin{displaymath}
    \min\bigl\{\, f(\mathbf x): \mathbf x\in F\cap\Z^n \,\bigr\}
    = \min\bigl\{\, f(\mathbf x): \mathbf x\in F\cap\Z^n, \ \mathopen\| \mathbf x
    \mathclose\| \leq R \,\bigr\}.
  \end{displaymath}
\end{theorem}

\subsection{Fixed dimension}
\label{s:convex-min-fixed-dim}

In fixed dimension, the problem of convex integer minimization can be solved
using variants of Lenstra's algorithm \cite{Lenstra83} for
integer programming.  Indeed, when the dimension~$n$ is fixed, the bound~$R$
given by Theorem \ref{th:encoding-length-convex-optimum} has a binary encoding
size that is bounded polynomially by the input data.  Thus, a Lenstra-type
algorithm can be started with a ``small'' (polynomial-size) initial outer
ellipsoid that includes a bounded part of the feasible region containing an
optimal integer solution.

The first algorithm of this kind for convex integer minimization
was announced by Khachiyan \cite{khachiyan:1983:polynomial-programming}.
In the following we present the variant of Lenstra's algorithm due to Heinz
\cite{heinz-2005:integer-quasiconvex}, which seems to yield the best known
complexity bound for the problem.  The complexity result is the following.
\begin{theorem}\label{th:heinz-complexity}
  Let $f, g_1,\dots,g_m \in\Z[x_1,\dots,x_n]$ be quasi-convex polynomials of
  degree at most~$d\geq2$, whose coefficients have a binary encoding length of
  at most~$\ell$.  There exists an algorithm running in time $m
  \ell^{\mathrm{O}(1)} d^{\mathrm{O}(n)} 2^{\mathrm{O}(n^3)}$ that computes a
  minimizer~$\mathbf x^*\in\Z^n$ of the problem~\eqref{eq:nonlinear-over-nonlinear}
  or reports that no minimizer exists.  If the algorithm outputs a
  minimizer~$\mathbf x^*$, its binary encoding size is $\ell
  d^{\mathrm{O}(n)}$.
\end{theorem}

A complexity result of greater generality was presented by
Khachiyan and Porkolab~\cite{khachiyan-porkolab:00}.  It covers the case of
minimization of convex polynomials over the integer points in convex
semialgebraic sets given by \emph{arbitrary} (not necessarily quasi-convex)
polynomials.
\begin{theorem}\label{th:khachiyan-porkolab-complexity}
  Let $Y\subseteq\R^k$ be a convex set given by
  \begin{displaymath}
    Y = \bigl\{\, \mathbf y\in\R^k :
    \mathrm{Q}_1\mathbf x^1\in\R^{n_1}\colon \cdots\ \mathrm{Q}_\omega\mathbf
    x^\omega\in\R^{n_\omega}\colon
    P(\mathbf y, \mathbf x^1,\dots,\mathbf x^\omega) \,\bigr\}
  \end{displaymath}
  with quantifiers $\mathrm Q_i\in\{\exists,\forall\}$, where $P$ is a Boolean
  combination of polynomial inequalities
  \begin{displaymath}
    g_i(\mathbf y, \mathbf x^1,\dots,\mathbf x^\omega) \leq 0,\quad i=1,\dots,m
  \end{displaymath}
  with degrees at most~$d\geq2$ and coefficients of binary encoding size at
  most~$\ell$.  There exists an algorithm for solving the problem
  \begin{math}
    \min \{\, y_k : \mathbf y \in Y\cap\Z^k \,\}
  \end{math}
  in time $\ell^{\mathrm{O}(1)} (md)^{\mathrm{O}(k^4) \prod_{i=1}^\omega \mathrm{O}(n_i)}$.
\end{theorem}
When the dimension~$k+\sum_{i=1}^\omega n_i$ is fixed, the algorithm runs in
polynomial time.  For the case of convex minimization where the feasible region is described by
convex polynomials, the complexity bound of
Theorem \ref{th:khachiyan-porkolab-complexity}, however, translates to $\ell^{\mathrm{O}(1)}
m^{\mathrm{O}(n^2)} d^{\mathrm{O}(n^4)}$, which is worse than the bound of
Theorem \ref{th:heinz-complexity}
\cite{heinz-2005:integer-quasiconvex}.\medbreak

In the remainder of this subsection, we describe the ingredients of the
variant of Lenstra's algorithm due to Heinz.  The algorithm starts out by
``rounding'' the feasible region, by applying the shallow-cut ellipsoid method
to find proportional inscribed and circumscribed ellipsoids.  It is well-known
\cite{GroetschelLovaszSchrijver88} that the shallow-cut ellipsoid method only
needs an initial circumscribed ellipsoid that is ``small enough'' (of
polynomial binary encoding size -- this follows from
Theorem \ref{th:encoding-length-convex-optimum}) and an implementation of a
\emph{shallow separation oracle}, which we describe below.


For a positive-definite matrix~$A$ we denote by $\mathcal{E}(A,\mathbf{\hat x})$ the ellipsoid $\{\, \mathbf
x\in\R^n : {(\mathbf x - \mathbf{\hat x})}^\top A (\mathbf x - \mathbf{\hat x}) \leq 1 \,\}$.

\begin{lemma}[Shallow separation oracle]
  Let $g_0,\dots,g_{m+1}\in\Z[\mathbf x]$ be quasi-convex polynomials of degree at
  most~$d$, the binary encoding sizes of whose coefficients are at most~$r$.
  Let the (continuous) feasible region~$F = \{\,\mathbf x\in\R^n :
  g_i(\mathbf x) < 0\,\}$ be contained in the ellipsoid~$\mathcal{E}(A,\mathbf{\hat x})$,
  where $A$ and $\mathbf{\hat x}$ have binary encoding size at most~$\ell$.
  There exists an algorithm with running time $m (lnr)^{\mathrm{O}(1)}
  d^{\mathrm{O}(n)}$ that outputs
  \begin{enumerate}[\rm(a)]
  \item ``true'' if
    \begin{equation}\label{eq:tough-ellipsoid}
      \mathcal{E}((n+1)^{-3} A, \mathbf{\hat x}) \subseteq F \subseteq \mathcal{E}(A, \mathbf{\hat
        x});
    \end{equation}
  \item otherwise, a vector $\mathbf c\in\Q^n\setminus\{\mathbf0\}$ of binary encoding
    length $(l+r) (dn)^{\mathrm{O}(1)}$ with
    \begin{equation}\label{eq:shallow-cut}
      F \subseteq \mathcal{E}(A, \mathbf{\hat x}) \cap
      \bigl\{\, \mathbf x \in\R^n : \mathbf c^\top (\mathbf x - \mathbf{\hat x}) \leq \tfrac1{n+1}
      (\mathbf c^\top A\mathbf c)^{1/2} \,\bigr\}.
    \end{equation}
  \end{enumerate}
\end{lemma}
\begin{proof}
  We give a simplified sketch of the proof, without hard complexity estimates.
  By applying an affine transformation to $F\subseteq\mathcal{E}(A,\mathbf{\hat x})$, we
  can assume that $F$ is contained in the unit ball~$\mathcal{E}(I,\mathbf0)$.  Let us
  denote as usual
  by $\mathbf e_1,\dots,\mathbf e_n$ the unit vectors and by $\mathbf e_{n+1},\dots,\mathbf
  e_{2n}$ their negatives.  The algorithm first constructs numbers
  $\lambda_{i1}, \dots, \lambda_{id} > 0$ with
  \begin{equation}\label{eq:lambda-bounds}
    \frac{1}{n+\frac32} < \lambda_{i1} < \dots < \lambda_{id} < \frac1{n+1}
  \end{equation}
  and the corresponding point sets
  \begin{math}
    B_i = \{\, \mathbf x_{ij} := \lambda_{ij} \mathbf
    e_i : j=1,\dots,d\,\};
  \end{math}
  see Figure \ref{fig:shallowcut-0}\,(a).  The choice of the bounds
  \eqref{eq:lambda-bounds} for~$\lambda_{ij}$ will ensure that we either find
  a large enough inscribed ball for~(a) or a deep enough cut for~(b).
  \begin{figure}[t]
    \centering
    (a)\hspace{-.5cm}
    \ifpdf
    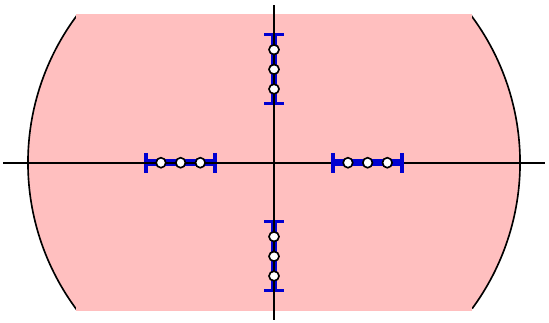
    \else
    \input{shallowcut-0.pstex_t}
    \fi
    \quad(b)\hspace{-.5cm}
    \ifpdf
    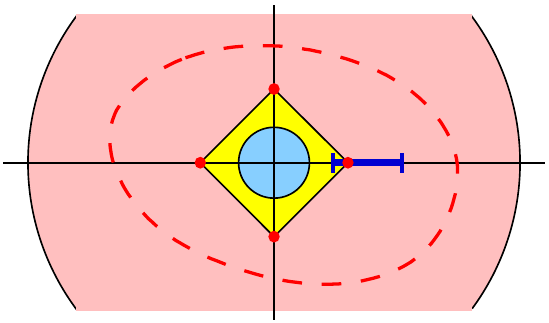
    \else
    \input{shallowcut.pstex_t}
    \fi
    \caption{The implementation of the shallow separation oracle.
      \textbf{(a)}~Test points $\mathbf x_{ij}$ in the circumscribed ball~$\mathcal{E}(1,\mathbf0)$.
      \textbf{(b)}~Case~I: All
      test points $\mathbf x_{i1}$ are (continuously) feasible; so their convex hull
      (a~cross-polytope) and its inscribed ball $\mathcal{E}((n+1)^{-3},\mathbf0)$ are
      contained in the (continuous) feasible region~$F$.
    }
    \label{fig:shallowcut-0}
    \label{fig:shallowcut}
  \end{figure}%
  Then the algorithm determines the (continuous) feasibility
  of the center~$\mathbf0$ and the $2n$ innermost points $\mathbf x_{i,1}$.\smallbreak

  \emph{Case~I.}
  If $\mathbf x_{i,1}\in F$ for $i=1,\dots, 2n$, then the cross-polytope
  $\mathop{\mathrm{conv}}\{\,\mathbf x_{i,1} : i = 1,\dots,2n\,\}$ is contained in~$F$; see
  Figure \ref{fig:shallowcut}\,(b).
  An easy
  calculation shows that the ball $\mathcal{E}((n+1)^{-3},\mathbf0)$ is contained in the
  cross-polytope and thus in~$F$; see Figure \ref{fig:shallowcut}.  Hence the
  condition in~(a) is satisfied and the algorithm outputs ``true''.\smallbreak

  \emph{Case~II.}  We now discuss the case when the center~$\mathbf 0$ violates a
  polynomial inequality $g_0(\mathbf x)<0$ (say).  Let $F_0 = \{\, \mathbf
  x\in\R^n: g_0(\mathbf x)<0\,\}\supseteq F$.  Due to convexity of~$F_0$, for all
  $i=1,\dots,n$, one set of each pair $B_i\cap F_0$ and $B_{n+i}\cap F_0$ must be
  empty; see~Figure~\ref{fig:shallowcut-2b}\,(a).  Without loss of generality, let us
  assume $B_{n+i}\cap F_0=\emptyset$ for all~$i$.
  \begin{figure}[t]
    \centering
    (a)\hspace{-.5cm}
    \ifpdf
    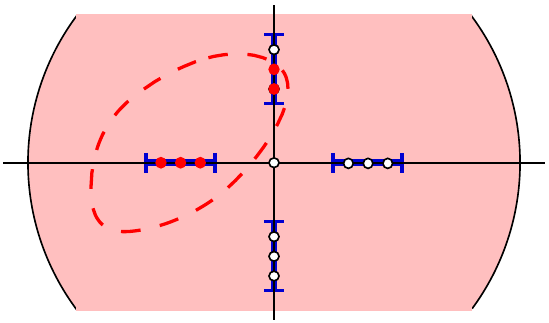
    \else
    \input{shallowcut-2b.pstex_t}
    \fi
    \quad(b)\hspace{-.5cm}
    \ifpdf
    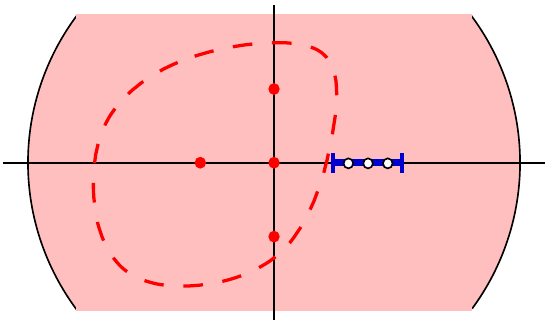
    \else
    \input{shallowcut-2a.pstex_t}
    \fi
    \caption{The implementation of the shallow separation oracle.  \textbf{(a)}~Case~II:
      The center~$\mathbf 0$ violates a polynomial inequality $g_0(\mathbf x)<0$
      (say).  Due to convexity, for all $i=1,\dots,n$, one set of each pair
      $B_i\cap F$ and $B_{n+i}\cap F$ must be empty.
      \textbf{(b)}~Case~III: A
      test point $\mathbf x_{k1}$ is infeasible, as it violates an
      inequality $g_0(\mathbf x)<0$ (say).  However, the center~$\mathbf 0$ is
      feasible at least for this inequality.}
    \label{fig:shallowcut-2a}
    \label{fig:shallowcut-2b}
  \end{figure}%
  We can determine whether a $n$-variate polynomial function of known maximum
  degree~$d$ is constant by evaluating it on $(d+1)^n$ suitable points (this
  is a consequence of the Fundamental Theorem of Algebra).  For our case of
  quasi-convex polynomials, this can be improved; indeed, it suffices to test
  whether the gradient $\nabla g_0$ vanishes on the $nd$ points in the set
  $B_1\cup\dots\cup B_{n}$.  If it does, we know that~$g_0$ is constant, thus
  $F=\emptyset$, and so can we return an arbitrary vector~$\mathbf c$.  Otherwise,
  there is a point $\mathbf x_{ij}\in B_i$ with $\mathbf c:= \nabla g_0(\mathbf x_{ij})\neq\mathbf0$;
  we return this vector as the desired normal vector of a shallow cut.
  Due to the choice of $\lambda_{ij}$ as a number smaller than~$\frac1{n+1}$,
  the cut is deep enough into the ellipsoid~$\mathcal{E}(A,\mathbf{\hat x})$,
  so that~\eqref{eq:shallow-cut} holds.
  \smallbreak

  \emph{Case~III.}  The remaining case to discuss is when $\mathbf0 \in F$ but
  there exists a $k \in\{1,\dots,2n\}$ with $\mathbf x_{k,1}\notin F$.  Without
  loss of generality, let $k=1$, and let $\mathbf x_{1,1}$ violate the polynomial
  inequality $g_0(\mathbf x)<0$, i.e., $g_0(\mathbf x_{1,1})\geq0$; see
  Figure \ref{fig:shallowcut-2a}\,(b).
  We consider the univariate polynomial
  $\phi(\lambda) = g_0(\lambda \mathbf e_i)$.  We have $\phi(0) = g_0(\mathbf0)<0$
  and $\phi(\lambda_{1,1}) \geq 0$, so $\phi$ is not constant.  Because $\phi$ has
  degree at most $d$, its derivative~$\phi'$ has degree at most~$d-1$, so
  $\phi'$ has at most $d-1$ roots.  Thus, for at least one of the $d$ different values
  $\lambda_{1,1}, \dots, \lambda_{1,d}$, say $\lambda_{1,j}$, we must have
  $\phi'(\lambda_{1,j})\neq0$.  This implies that $\mathbf c := \nabla g_0(\mathbf
  x_{1,j}) \neq\mathbf0$.  By convexity, we have $\mathbf x_{1,j}\notin F$, so we can
  use $\mathbf c$ as the normal vector of a shallow cut.
\end{proof}

By using this oracle in the shallow-cut ellipsoid method, one obtains
the following result.
\begin{corollary}\label{th:shallow-cut-ellipsoid-method}
  Let $g_0,\dots,g_{m}\in\Z[\mathbf x]$ be quasi-convex polynomials of degree at
  most~$d\geq2$.  Let the (continuous) feasible region~$F = \{\,\mathbf x\in\R^n :
  g_i(\mathbf x) \leq 0\,\}$ be contained in the ellipsoid~$\mathcal{E}(A_0,\mathbf0)$, given
  by the positive-definite matrix~$A_0\in\Q^{n\times n}$.  Let
  $\epsilon\in\Q_{>0}$ be given.  Let the entries of $A_0$ and the
  coefficients of all monomials of $g_0,\dots,g_{m}$ have binary encoding size
  at most~$\ell$.

  There exists an algorithm with running time $m (\ell
  n)^{\mathrm{O}(1)} d^{\mathrm{O}(n)}$ that computes a positive-definite
  matrix~$A\in\Q^{n\times n}$ and a point~$\mathbf{\hat x}\in\Q^n$ with
  \begin{enumerate}[\rm(a)]
  \item either $\mathcal{E}((n+1)^{-3} A, \mathbf{\hat x}) \subseteq F \subseteq \mathcal{E}(A, \mathbf{\hat x})$
  \item or $F \subseteq \mathcal{E}(A, \mathbf{\hat x})$ and $\mathop{\mathrm{vol}} \mathcal{E}(A, \mathbf{\hat x})
    <\epsilon$.
  \end{enumerate}
\end{corollary}

Finally, there is a lower bound for the volume of a continuous feasible
region~$F$ that can contain an integer point.

\begin{lemma}
  Under the assumptions of Theorem \ref{th:shallow-cut-ellipsoid-method},
  if $F\cap\Z^n\neq \emptyset$, then there exists an $\epsilon\in\Q_{>0}$ of binary
  encoding size $\ell (dn)^{\mathrm{O}(1)}$ with $\mathop{\mathrm{vol}} F>\epsilon$.
\end{lemma}

On the basis of these results, one obtains a Lenstra-type algorithm for the
decision version of the convex integer minimization problem with the desired
complexity.  By applying binary search, the optimization problem can be
solved, which provides a proof of Theorem \ref{th:heinz-complexity}.

\subsection{Boundary cases of complexity}
\label{s:n-fold-min}

In this section we present an optimality certificate for problems of
the form
\[
\min\{f(\mathbf x):A\mathbf x=\mathbf b,\mathbf l\leq \mathbf x\leq \mathbf u,\mathbf x\in\Z^n\},
\]
where $A\in\Z^{d\times n}$, $\mathbf b\in\Z^d$, $\mathbf l,\mathbf u\in\Z^n$, and where
$f:\R^n\rightarrow\R$ is a separable convex function, that is,
$f(\mathbf x)=\sum\limits_{i=1}^n f_i(x_i)$ with convex functions
$f_i:\R\rightarrow\R$, $i=1,\ldots,n$. This certificate then
immediately leads us to a oracle-polynomial time algorithm to solve
the separable convex integer minimization problem at hand. Applied to separable convex $n$-fold integer minimization problems, this gives a polynomial time algorithm for their solution \cite{Hemmecke+Onn+Weismantel:08}.

For the construction of the optimality certificate, we exploit a nice
super-additivity property of separable convex functions.

\begin{lemma}\label{Lemma: Superadditivity of separable convex functions}
Let $f:\R^n\rightarrow\R$ be a separable convex function and let
$\mathbf h_1,\ldots \mathbf h_k\in\R^n$ belong to a common orthant of $\R^n$, that is,
they all have the same sign pattern from $\{\geq 0, \leq 0\}^n$. Then,
for any $\mathbf x\in\R^n$ we have
\[
f\left(\mathbf x+\sum_{i=1}^k \mathbf h_i\right)-f(\mathbf x)\geq \sum_{i=1}^k [f(\mathbf x+\mathbf h_i)-f(\mathbf x)].
\]
\end{lemma}

\begin{proof}
The claim is easy to show for $n=1$ by induction. If, w.l.o.g.,
$h_2\geq h_1\geq 0$ then convexity of $f$ implies
$[f(x+h_1+h_2)-f(x+h_2)]/h_1\geq [f(x+h_1)-f(x)]/h_1$, and thus
$f(x+h_1+h_2)-f(x)\geq [f(x+h_2)-f(x)]+ [f(x+h_1)-f(x)]$. The claim
for general $n$ then follows from the separability of $f$ by adding
the superadditivity relations of each one-parametric convex summand of $f$.
\end{proof}

A crucial role in the following theorem is again played by the Graver
basis $\mathcal{G}(A)$ of $A$. Let us remind the reader that the Graver
basis $\mathcal{G}(A)$ has a nice representation property due to its
definition: every $\mathbf z\in\ker(A)\cap\Z^n$ can be written as a
sign-compatible nonnegative integer linear combination $\mathbf z=\sum_i
\alpha_i \mathbf g_i$ of Graver basis elements $\mathbf g_i\in\mathcal{G}(A)$. This
followed from the simple observation that $\mathbf z$ has to belong to some
orthant $\mathcal{O}_j$ of $\R^n$ and thus it can be represented as a
sign-compatible nonnegative integer linear combination of elements in
$H_j\subseteq\mathcal{G}(A)$ belonging to this orthant. Note that by the integer Carath\'eodory
property of Hilbert bases, at most $2\cdot\dim(\ker(A))-2$ vectors are
needed in such a representation \cite{Seboe:90}. It is precisely this
simple representation property of $\mathcal{G}(A)$ combined with the
superadditivity of the separable convex function $f$ that turns
$\mathcal{G}(A)$ into an optimality certificate for $\min\{f(\mathbf x):A\mathbf x=\mathbf b,\mathbf l\leq
\mathbf x\leq \mathbf u,\mathbf x\in\Z^n\}$.

\begin{theorem}\label{Theorem: Convex integer minimization}
Let $f:\R^n\rightarrow\R$ be a separable convex function given by a
comparison oracle that when queried on $\mathbf x,\mathbf y\in\Z^n$ decides whether
$f(\mathbf x)<f(\mathbf y)$, $f(\mathbf x)=f(\mathbf y)$, or $f(\mathbf x)>f(\mathbf y)$. Then $\mathbf x_0$ is an optimal
feasible solution to $\min\{f(\mathbf x):A\mathbf x=\mathbf b,\mathbf l\leq \mathbf x\leq \mathbf u,\mathbf x\in\Z^n\}$ if and
only if for all $\mathbf g\in\mathcal{G}(A)$ the vector $\mathbf x_0+\mathbf g$ is not feasible or
$f(\mathbf x_0+\mathbf g)\geq f(\mathbf x_0)$.
\end{theorem}

\begin{proof}
Assume that $\mathbf x_0$ is not optimal and let $\mathbf x_{\min}$ be an optimal
solution to the given problem. Then $\mathbf x_{\min}-\mathbf x_0\in\ker(A)$ and thus
it can be written as a sign-compatible nonnegative integer linear
combination $\mathbf x_{\min}-\mathbf x_0=\sum_i \alpha_i \mathbf g_i$ of Graver basis
elements $\mathbf g_i\in\mathcal{G}(A)$. We show that one of the $\mathbf g_i$ must be an
improving vector, that is, for some $\mathbf g_i$ we have that $\mathbf x_0+\mathbf g_i$ is
feasible and $f(\mathbf x_0+\mathbf g_i)<f(\mathbf x_0)$.

For all $i$, the vector $\mathbf g_i$ has the same sign-pattern as
$\mathbf x_{\min}-\mathbf x_0$ and it is now easy to check that the coordinates of $\mathbf x_0+\mathbf g_i$ lie between
the corresponding coordinates of $\mathbf x_0$ and $\mathbf x_{\min}$. This implies in particular
$\mathbf l\leq \mathbf x_0+\mathbf g_i\leq \mathbf u$. Because $\mathbf g_i\in\ker(A)$, we also have
$A(\mathbf x_0+\mathbf g_i)=\mathbf b$ for all $i$. Consequently, for all $i$ the vector
$\mathbf x_0+\mathbf g_i$ would be a feasible solution. It remains to show that one of
these vectors has a strictly smaller objective value than $\mathbf x_0$.

Due to the superadditivity from Lemma \ref{Lemma: Superadditivity of
separable convex functions}, we have
\[
0\geq f(\mathbf x_{\min})-f(\mathbf x_0)=
f\left(\mathbf x_0+\sum_{i=1}^{2n-2} \alpha_i \mathbf g_i\right)-f(\mathbf x_0)\geq
\sum_{i=1}^k \alpha_i[f(\mathbf x_0+\mathbf g_i)-f(\mathbf x_0)].
\]
Thus, at least one of the summands $f(\mathbf x_0+\mathbf g_i)-f(\mathbf x_0)$ must be
negative and we have found an improving vector for $\mathbf z_0$ in
$\mathcal{G}(A)$. \end{proof}

We now turn this optimality certificate into a polynomial oracle-time
algorithm to solve the separable convex integer minimization problem
$\min\{f(\mathbf x):A\mathbf x=\mathbf b,\mathbf l\leq \mathbf x\leq \mathbf u,\mathbf x\in\Z^n\}$.
For this, we call $\alpha \mathbf g$ a \emph{greedy} Graver improving
vector if $\mathbf x_0+\alpha \mathbf g$ is feasible and such that $f(\mathbf x_0+\alpha \mathbf g)$
is minimal among all such choices of $\alpha\in\Z_+$ and
$\mathbf g\in\mathcal{G}(A)$. Then the following result holds.

\begin{theorem}\label{Theorem: Main theorem convex integer minimization}
Let $f:\R^n\rightarrow\R$ be a separable convex function given by a
comparison oracle. Moreover, assume that $|f(\mathbf x)|<M$ for all
$\mathbf x\in\{\mathbf x:A\mathbf x=\mathbf b,\mathbf l\leq \mathbf x\leq \mathbf u,\mathbf x\in\Z^n\}$. Then any feasible solution
$\mathbf x_0$ to $\min\{f(\mathbf x):A\mathbf x=\mathbf b,\mathbf l\leq \mathbf x\leq \mathbf u,\mathbf x\in\Z^n\}$ can be augmented
to optimality by a number of greedy Graver augmentation steps that is polynomially bounded in the encoding lengths of
$A$, $\mathbf b$, $\mathbf l$, $\mathbf u$, $M$, and $\mathbf x_0$.
\end{theorem}

\begin{proof}
Assume that $\mathbf x_0$ is not optimal and let $\mathbf x_{\min}$ be an optimal
solution to the given problem. Then $\mathbf x_{\min}-\mathbf x_0\in\ker(A)$ and thus
it can be written as a sign-compatible nonnegative integer linear
combination $\mathbf x_{\min}-\mathbf x_0=\sum_i \alpha_i \mathbf g_i$ of at most $2n-2$
Graver basis elements $\mathbf g_i\in\mathcal{G}(A)$. As in the proof of Theorem \ref{Theorem: Convex integer minimization},
sign-compatibility implies that for all $i$ the coordinates of $\mathbf x_0+\alpha_i \mathbf g_i$ lie
between the corresponding coordinates of $\mathbf x_0$ and
$\mathbf x_{\min}$. Consequently, we have $\mathbf l\leq \mathbf x_0+\alpha_i \mathbf g_i\leq \mathbf u$.
Because $\mathbf g_i\in\ker(A)$, we also have $A(\mathbf x_0+\alpha_i \mathbf g_i)=\mathbf b$ for all
$i$. Consequently, for all $i$ the vector $\mathbf x_0+\alpha_i \mathbf g_i$ would be
a feasible solution.

Due to the superadditivity from Lemma \ref{Lemma: Superadditivity of
separable convex functions}, we have
\[
0\geq f(\mathbf x_{\min})-f(\mathbf x_0)=
f\left(\mathbf x_0+\sum_{i=1}^{2n-2} \alpha_i \mathbf g_i\right)-f(\mathbf x_0)\geq
\sum_{i=1}^k [f(\mathbf x_0+\alpha_i \mathbf g_i)-f(\mathbf x_0)].
\]
Thus, at least one of the summands $f(\mathbf x_0+\alpha_i \mathbf g_i)-f(\mathbf x_0)$ must be smaller
than $\frac{1}{2n-2}[f(\mathbf x_{\min})-f(\mathbf x_0)]$, giving an improvement that
is at least $\frac{1}{2n-2}$ times the maximal possible improvement
$f(\mathbf x_{\min})-f(\mathbf x_0)$. Such a geometric improvement, however, implies
that the optimum is reached in a number of greedy augmentation
steps which is polynomial in the encoding lengths of $A$, $\mathbf b$, $\mathbf l$,
$\mathbf u$, $M$, and $\mathbf x_0$ \cite{Ahuja+Magnanti+Orlin}.
\end{proof}

Thus, once we have a polynomial size Graver basis, we get a polynomial time algorithm to solve the convex integer minimization problem at hand.

For this, let us consider again $n$-fold systems (introduced in Section \ref{s:n-fold-convex-max}). Two nice stabilization results established by Ho\c{s}ten and Sullivant \cite{Hosten+Sullivant:07} and Santos and Sturmfels \cite{Santos+Sturmfels:03} immediately imply that if $A_1$ and $A_2$ are kept fixed, then the size of the Graver basis increases only polynomially in the number $n$ of copies of $A_1$ and $A_2$.

\begin{proposition}\label{GraverComputation}
For any fixed $(r+s)\times t$ integer matrix $A$ there is a polynomial time
algorithm that, given any $n$, computes the Graver basis $\mathcal{G}(A^{(n)})$ of
the n-fold matrix $A^{(n)}=({\bf 1}_n\otimes A_1)\oplus(I_n \otimes A_2)$.
\end{proposition}

Combining this proposition with Theorem \ref{Theorem: Main theorem convex integer minimization}, we get following nice result from \cite{Hemmecke+Onn+Weismantel:08}.

\begin{theorem}\label{Theorem: n-fold convex minimization}
Let $A$ be a fixed integer $(r+s)\times t$ matrix and let $f:\R^{nt}\rightarrow\R$ be any separable convex function given by a comparison oracle. Then there is a polynomial time algorithm that, given $n$, a right-hand side vector $\mathbf b\in\Z^{r+ns}$ and some bound $|f(\mathbf x)|<M$ on $f$ over the feasible region, solves the n-fold convex integer programming problem
$$\min\{f(\mathbf x):\ A^{(n)}\mathbf x=\mathbf b,\ \mathbf x\in{\mathbb N}^{nt}\}.$$
\end{theorem}

Note that by applying an approach similar to Phase I of the simplex method one can also compute an initial feasible solution $x_0$ to the $n$-fold integer program in polynomial time based on greedy Graver basis directions \cite{DeLoera+Hemmecke+Onn+Weismantel:08,Hemmecke:2003b}.

We wish to point out that the presented approach can be
generalized to the mixed-integer situation and also to more general
objective functions that satisfy a certain
superadditivity/subadditivity condition, see
\cite{Hemmecke+Koeppe+Weismantel:08,Lee+Onn+Weismantel:08} for more
details. Note that for mixed-integer convex problems one may only expect an approximation result,
as there need not exist a rational optimum. In fact, already a mixed-integer greedy augmentation
vector can be computed only approximately. Nonetheless, the technical difficulties when adjusting the
proofs for the pure integer case to the mixed-integer situation can be overcome
\cite{Hemmecke+Koeppe+Weismantel:08}. It should be noted, however, that the Graver
basis of $n$-fold matrices does not show a stability result similar to the pure integer case as presented
in \cite{Hosten+Sullivant:07,Santos+Sturmfels:03}. Thus, we do not get a nice polynomial time algorithm for
solving mixed-integer convex $n$-fold problems.

\subsection{Practical algorithms}
\label{s:outer-approx}

In this section, the methods that we look at,
aimed at formulations having convex continuous relaxations,
are driven by O.R./engineering approaches, transporting and
motivated by successful mixed-integer linear programming technology and
smooth continuous nonlinear programming technology.
In Section \ref{PracGen} we discuss general algorithms that
make few assumptions beyond those that are typical for
convex continuous nonlinear programming.
In Section \ref{PracSOCP} we present some more specialized
techniques aimed at  convex quadratics.

\subsubsection{General algorithms}\label{PracGen}

Practical, broadly applicable approaches to general mixed-integer nonlinear programs
are aimed at problems involving convex minimization over a convex set
with some additional integrality restriction. Additionally, for the sake of
obtaining well-behaved continuous relaxations, a certain amount of
smoothness is usually assumed. Thus, in this section, the model that
we focus on is
\begin{equation}
\begin{aligned}
  \hbox{min}\quad & f(\mathbf x,\mathbf y)\\
  \hbox{s.t.}\quad & g(\mathbf x,\mathbf y) \leq \mathbf 0 \\
      & \mathbf l\le \mathbf y \le \mathbf u\\
  & \mathbf x\in\R^{n_1},\ \mathbf y \in \Z^{n_2},
\end{aligned} \label{eq:P}\tag*{$(\mbox{P}[\mathbf l,\mathbf u])$}
\end{equation}
where $f\colon \R^n\to\R$  and $g\colon \R^n\to\R^m$
are twice continuously-differentiable
convex functions,
 $\mathbf l\in (\Z\cup\{-\infty\})^{n_2}$,
$\mathbf u\in (\Z\cup\{+\infty\})^{n_2}$, and $\mathbf l\le \mathbf u$.
 It is also helpful to assume that
the feasible region of the relaxation of $(\mbox{P}[\mathbf l,\mathbf u])$ obtained by
replacing $\mathbf y\in \Z^{n_2}$
with $\mathbf y\in\R^{n_2}$ is bounded. We denote this continuous relaxation by  $(\mbox{P}_\R[\mathbf l,\mathbf u])$.

To describe the various algorithmic approaches, it is helpful to
define some related subproblems of  $(\mbox{P}[\mathbf l,\mathbf u])$ and associated relaxations.
Our notation is already designed for this.
For vector $\mathbf l'\in (\Z\cup\{-\infty\})^{n_2}$
and $\mathbf u'\in (\Z\cup\{+\infty\})^{n_2}$, with $\mathbf l \le \mathbf l'\le \mathbf u' \le \mathbf u$,
we have the \emph{subproblem} $(\mbox{P}[\mathbf l',\mathbf u'])$  and its
associated continuous relaxation $(\mbox{P}_\R[\mathbf l',\mathbf u'])$.

Already, we can see how the family of relaxations $(\mbox{P}_\R[\mathbf l',\mathbf u'])$
leads to the obvious extension of the Branch-and-Bound Algorithm of mixed-integer linear programming.
Indeed, this approach was experimented with in \cite{MR878885}.
The \emph{Branch-and-Bound Algorithm} for mixed-integer nonlinear programming
has been implemented as {\tt MINLP-BB} \cite{minlpbb}, with continuous nonlinear-programming
subproblem relaxations solved with the active-set solver {\tt filterSQP}
and also as {\tt SBB},  with associated subproblems
 solved with any of {\tt CONOPT}, {\tt SNOPT} and {\tt MINOS}.
Moreover,
Branch-and-Bound is available as an algorithmic option in
the actively developed code {\tt Bonmin}
\cite{bonmin,BMUM,bonminOptima}, which can be used as a
callable library, as a stand-alone solver, via the modeling languages {\tt AMPL} and {\tt GAMS},
and is available in source-code form, under the Common Public License,
from COIN-OR \cite{bonmincode},
available for running on NEOS \cite{bonminNEOS}.
By default, relaxations of subproblems are solved with
the interior-point solver {\tt Ipopt} (whose
availability options include all of those for {\tt Bonmin}), though
there is also an
interface to {\tt filterSQP}.
The Branch-and-Bound Algorithm in {\tt Bonmin}
includes effective strong branching and SOS branching.
It can also be used as a robust heuristic on problems for which the
relaxation $(\mbox{P}_\R)$ does not have a convex feasible region,
by setting negative `cutoff gaps'.

Another type of algorithmic approach emphasizes continuous nonlinear
programming over just the continuous variables of the formulation.
For fixed $\bar{\mathbf y} \in\Z^{n_2}$, we define
\begin{equation}
\begin{aligned}
  \hbox{min}\quad & f(\mathbf x,\mathbf y)\\
  \hbox{s.t.}\quad & g(\mathbf x,\mathbf y) \leq \mathbf0 \\
  & \mathbf y = \bar{\mathbf y} \\
  & \mathbf x\in\R^{n_1}.
\end{aligned} \label{eq:Py}\tag*{$(\mbox{P}^{\bar{\mathbf y}})$}
\end{equation}
Clearly any feasible solution to such a continuous nonlinear-programming
subproblem $(\mbox{P}^{\bar{\mathbf y}})$
yields an upper bound on the optimal value of $(\mbox{P}[\mathbf l,\mathbf u])$.
When $(\mbox{P}^{\bar{\mathbf y}})$ is infeasible, we may consider the continuous
nonlinear-programming feasibility subproblem
\begin{equation}
\begin{aligned}
  \hbox{min}\quad & \sum_{i=1}^m w_i\\
  \hbox{s.t.}\quad & g(\mathbf x,\mathbf y) \leq \mathbf w \\
  & \mathbf y = \bar{\mathbf y}\\
  & \mathbf x\in\R^{n_1}\\
  & \mathbf w\in \R^m_+.
\end{aligned} \label{eq:PFy}\tag*{$(\mbox{F}^{\bar{\mathbf y}})$}
\end{equation}

If we can find a way to couple the solution of upper-bounding
problems $(\mbox{P}^{\bar{\mathbf y}})$ (and the closely related
feasibility subproblems $(\mbox{F}^{\bar{\mathbf y}})$) with
a lower-bounding procedure exploiting the convexity assumptions,
then we can hope to build an
iterative procedure that will converge to a global
optimum of  $(\mbox{P}[\mathbf l,\mathbf u])$. Indeed, such a procedure is the
\emph{Outer-Approximation (OA) Algorithm} \cite{MR866413,MR918874}.
Toward this end, for a finite set of ``linearization points''
\[
{\mathcal K}:=\left\{\left(\mathbf x^k\in\R^{n_1},\mathbf y^k\in\R^{n_2}\right) ~:~ k=1,\ldots,K\right\},
\]
we define the mixed-integer \emph{linear} programming
relaxation
\begin{equation}
\begin{aligned}
  \hbox{min}\quad & z\\
   \hbox{s.t.}\quad  & \nabla f(\mathbf x^k,\mathbf y^k)^\top
  \left(
  \begin{matrix}
  \mathbf x-\mathbf x^k\\
  \mathbf y- \mathbf y^k
  \end{matrix}
  \right) + f(\mathbf x^k,\mathbf y^k) \leq z, \quad \forall\ \left(\mathbf x^k,\mathbf y^k\right)\in {\mathcal K}\\
 & \nabla g(\mathbf x^k,\mathbf y^k)^\top
  \left(
  \begin{matrix}
  \mathbf x-\mathbf x^k\\
  \mathbf y- \mathbf y^k
  \end{matrix}
  \right) + g(\mathbf x^k,\mathbf y^k) \leq 0, \quad \forall\ \left(\mathbf x^k,\mathbf y^k\right)\in {\mathcal K}\\
  & \mathbf x\in\R^{n_1}\\
  & \mathbf y\in \R^{n_2}, \quad \mathbf l\le \mathbf y \le \mathbf u \\
    & z\in \R.\\
\end{aligned} \label{eq:MILPk}\tag*{$(\mbox{P}^{{\mathcal K}}[\mathbf l,\mathbf u])$}
\end{equation}
We are now able to concisely state the basic OA Algorithm.

\begin{algorithm}[OA Algorithm]~\smallskip\par
\label{algo:OA}
\noindent {\em Input:} The mixed-integer nonlinear program $(\mbox{P}[\mathbf l,\mathbf u])$.

\noindent {\em Output:} An optimal solution $\left(\mathbf x^*,\mathbf y^*\right)$.

\begin{enumerate}[\rm\ 1.]
\item Solve the nonlinear-programming relaxation $(\mbox{P}_\R)$, let $\left(\mathbf x^1,\mathbf y^1\right)$ be an optimal solution,  and let $K:=1$, so that initially we have
${\mathcal K}=\left\{\left(\mathbf x^1,\mathbf y^1\right)\right\}$.
\item\label{MILPstep} Solve the mixed-integer linear programming relaxation $(\mbox{P}^{{\mathcal K}}[\mathbf l,\mathbf u])$, and let $\left(\mathbf x^*,\mathbf y^*,z^*\right)$
be an optimal solution. If $\left(\mathbf x^*,\mathbf y^*,z^*\right)$ corresponds to a feasible
solution of $(\mbox{P}[\mathbf l,\mathbf u])$ (i.e, if $f(\mathbf x^*,\mathbf y^*) \le z^*$ and $g(\mathbf x^*,\mathbf y^*)\le \mathbf 0$),
then STOP (with the optimal solution  $(\mathbf x^*,\mathbf y^*)$ of  $(\mbox{P}[\mathbf l,\mathbf u])$).
\item \label{cutstep1} Solve the continuous nonlinear-programming subproblem $(\mbox{P}^{\bar{\mathbf y}^*})$.
  \begin{enumerate}[\rm\ i.]
  \item Either a feasible solution $\left(\mathbf x^*,\mathbf y^*, z^*\right)$ is obtained,
  \item or  $(\mbox{P}^{\bar{\mathbf y}^*})$ is infeasible, in which case we solve the
  nonlinear-programming feasibility subproblem $(\mbox{F}^{\bar{\mathbf y}})$, and let its solution be   $\left(\mathbf x^*,\mathbf y^*, \mathbf u^*\right)$
\end{enumerate}
\item \label{cutstep2}  In either case, we augment the set ${\mathcal K}$ of linearization points,
by letting $K:=K+1$ and $\left(\mathbf x^K,\mathbf y^K\right):=\left(\mathbf x^*,\mathbf y^*\right)$.
\item GOTO \ref{MILPstep}.
\end{enumerate}
\end{algorithm}

Each iteration of Steps \ref{cutstep1}-\ref{cutstep2} generate a linear cut
that can improve the mixed-integer linear programming relaxation $(\mbox{P}^{{\mathcal K}}[\mathbf l,\mathbf u])$
that is repeatedly solved in Step \ref{MILPstep}. So clearly the sequence of
optimal objective values
for $(\mbox{P}^{{\mathcal K}}[\mathbf l,\mathbf u])$ obtained in  Step \ref{MILPstep} corresponds to a nondecreasing sequence of lower bounds on the
optimum value of $(\mbox{P}[\mathbf l,\mathbf u])$. Moreover each linear cut returned from Steps \ref{cutstep1}-\ref{cutstep2} cuts off the previous solution of
$(\mbox{P}^{{\mathcal K}}[\mathbf l,\mathbf u])$ from Step \ref{MILPstep}.
A precise proof of convergence (see for example \cite{bonmin}) uses these simple observations, but it also requires
an additional assumption that is standard in continuous nonlinear programming (i.e.
a ``constraint qualification'').

Implementations of OA include {\tt DICOPT} \cite{gamsdicopt}
which can be used with either of the mixed-integer
linear programs codes {\tt Cplex} and {\tt Xpress-MP},
in conjunction with any of the continuous nonlinear programming codes
{\tt CONOPT}, {\tt SNOPT} and {\tt MINOS} and is available with {\tt GAMS}.
Additionally {\tt Bonmin} has OA as an algorithmic option,
which can use {\tt Cplex}  or the COIN-OR code {\tt Cbc} as its mixed-integer
linear programming solver,  and {\tt Ipopt} or {\tt FilterSQP} as
its  continuous nonlinear programming solver.

Generalized Benders Decomposition \cite{MR0327310} is a technique that is closely related
to and substantially predates the OA Algorithm. In fact, one can regard
OA as a proper strengthening of Generalized Benders Decomposition (see  \cite{MR866413,MR918874}),
so as a practical tool, we view it as superseded by OA.

Substantially postdating the development of the OA Algorithm is the
simpler and closely related \emph{Extended Cutting Plane (ECP) Algorithm}
introduced in \cite{WesPet1995}.
The original  ECP Algorithm is a straightforward
generalization of \emph{Kelley's Cutting-Plane Algorithm} \cite{kelley1960}  for
convex continuous nonlinear programming (which predates
the development of the OA Algorithm).
Subsequently, the ECP Algorithm has been enhanced and further developed
(see, for example \cite{WesPorn2002,WesLun2005}) to handle,
for example, even pseudo-convex functions.

The motivation for the ECP Algorithm
is that continuous nonlinear programs are expensive to
solve, and all that the associated solutions give us are further linearization points for
$(\mbox{P}^{{\mathcal K}}[\mathbf l,\mathbf u])$. So the ECP Algorithm dispenses altogether with the solution
of  continuous nonlinear programs. Rather, in the most rudimentary version,
after each solution of the mixed-integer linear program $(\mbox{P}^{{\mathcal K}}[\mathbf l,\mathbf u])$,
the most violated constraint (i.e, of $f(\mathbf x^*,\mathbf y^*) \le z$ and $g(\mathbf x^*,\mathbf y^*)\le \mathbf0$)
is linearized and appended to $(\mbox{P}^{{\mathcal K}}[\mathbf l,\mathbf u])$.
This simple iteration is enough to easily establish convergence (see \cite{WesPet1995}).
It should be noted that for the case in which there are no integer-constrained variables,
then at each step $(\mbox{P}^{{\mathcal K}}[\mathbf l,\mathbf u])$ is just a continuous linear program and
we exactly recover Kelley's Cutting-Plane Algorithm for convex continuous nonlinear programming.

It is interesting to note that Kelley, in his seminal paper \cite{kelley1960},
already considered application of his approach to \emph{integer} nonlinear programs.
In fact, Kelley cited Gomory's seminal work on integer programming \cite{Gom58,Gom58b}
which was also taking place in the same time period, and he discussed how the approaches
could be integrated.

Of course, many practical improvements can be made
to the rudimentary ECP Algorithm. For example,
more constraints can be linearized at each iteration.
An implementation of the ECP Algorithm is the code {\tt Alpha-ECP}
(see \cite{WesLun2005})
which uses {\tt Cplex}  as its mixed-integer
linear programming solver and is available with {\tt GAMS}.
The general experience is that for mildly nonlinear problems,
an ECP Algorithm can outperform an OA Algorithm. But
for a highly nonlinear problem, the performance of
the ECP Algorithm is limited by the performance of Kelley's
Cutting-Plane Algorithm, which can be quite poor on
highly-nonlinear purely continuous problems. In such cases,
it is typically better to use an OA Algorithm, which will handle the
nonlinearity in a more sophisticated manner.

In considering again the performance of an OA Algorithm
on a mixed-integer nonlinear program $(\mbox{P}[\mathbf l,\mathbf u])$, rather than
the  convex continuous nonlinear programming problems $(\mbox{P}^{\bar{\mathbf y}^*})$ and
$(\mbox{F}^{\bar{\mathbf y}})$ being too time
consuming to solve (which led us to the ECP Algorithm), it can
be the case that solution of the mixed-integer linear programming problems
$(\mbox{P}^{{\mathcal K}}[\mathbf l,\mathbf u])$ dominate the running time.
Such a situation led to the \emph{Quesada-Grossmann Branch-and-Cut Algorithm} \cite{QG1992}.
The viewpoint is that the mixed-integer linear programming problems
$(\mbox{P}^{{\mathcal K}}[\mathbf l,\mathbf u])$ are solved by a Branch-and-Bound or Branch-and-Cut
Algorithm. During the solution of the mixed-integer linear programming problem
$(\mbox{P}^{{\mathcal K}}[\mathbf l,\mathbf u])$, whenever a new solution is found (i.e., one that has
the variables $\mathbf y$ integer), we interrupt the solution process for $(\mbox{P}^{{\mathcal K}}[\mathbf l,\mathbf u])$, and we
solve the convex continuous nonlinear programming problems $(\mbox{P}^{\bar{\mathbf y}^*})$ to derive
new outer approximation cuts that
are appended to mixed-integer linear programming problem
$(\mbox{P}^{{\mathcal K}}[\mathbf l,\mathbf u])$. We then continue with the solution process
for $(\mbox{P}^{{\mathcal K}}[\mathbf l,\mathbf u])$. The Quesada-Grossmann Branch-and-Cut Algorithm
is available as an option in {\tt Bonmin}.

Finally, it is clear that the essential scheme of
the Quesada-Grossmann Branch-and-Cut Algorithm admits enormous flexibility.
The \emph{Hybrid Algorithm} \cite{bonmin} incorporates two important enhancements.

First, we can seek to further improve the linearization $(\mbox{P}^{{\mathcal K}}[\mathbf l,\mathbf u])$
by solving convex continuous nonlinear programming problems  at additional nodes of the mixed-integer linear programming
Branch-and-Cut tree for $(\mbox{P}^{{\mathcal K}}[\mathbf l,\mathbf u])$ --- that is, not just when
solutions are found having $\mathbf y$ integer. In particular, at any node
$(\mbox{P}^{{\mathcal K}}[\mathbf l',\mathbf u'])$ of the
mixed-integer linear programming
Branch-and-Cut tree, we can solve the
associated convex continuous nonlinear programming subproblem
$(\mbox{P}_\R[\mathbf l',\mathbf u'])$: Then, if in the solution  $(\mathbf x^*, \mathbf y^*)$
we have that $\mathbf y^*$ is
integer, we may update the incumbent and fathom the node; otherwise, we
append $(\mathbf x^*, \mathbf y^*)$ to the set ${\mathcal K}$ of linearization points.
In the extreme case, if we solve these continuous nonlinear programming subproblems
at every node, we essentially have the
Branch-and-Bound Algorithm for mixed-integer nonlinear programming.

A second enhancement is based on working harder to find a
solution $(\mathbf x^*, \mathbf y^*)$ with  $\mathbf y^*$
integer at selected nodes $(\mbox{P}^{{\mathcal K}}[\mathbf l',\mathbf u'])$
 of the mixed-integer linear programming
Branch-and-Cut tree.
The idea is that at a node $(\mbox{P}^{{\mathcal K}}[\mathbf l',\mathbf u'])$,
we perform a time-limited
mixed-integer linear programming
Branch-and-Bound Algorithm.
If we are successful, then we will have
found a solution to the node with $(\mathbf x^*, \mathbf y^*)$ with  $\mathbf y^*$ integer,
and then we perform an OA iteration (i.e., Steps \ref{cutstep1}-\ref{cutstep2})
on $(\mbox{P}[\mathbf l',\mathbf u'])$
which will improve the linearization $(\mbox{P}^{{\mathcal K}}[\mathbf l',\mathbf u'])$.
We can then repeat this until we have solved the
mixed-integer nonlinear program $(\mbox{P}[\mathbf l',\mathbf u'])$ associated with the node.
If we do this without time limit at the root node $(\mbox{P}[\mathbf l,\mathbf u])$, then the
entire procedure reduces to the OA Algorithm.
The Hybrid Algorithm was developed for and first made available as part of {\tt Bonmin}.

{\tt FilMint} \cite{AbhLL2006} is another successful
modern code, also based on enhancing the general
framework of the Quesada-Grossmann Branch-and-Cut Algorithm.
The main additional innovation introduced with {\tt FilMint} is the
idea of using ECP cuts rather than only performing OA iterations
for getting cuts to improve the linearizations $(\mbox{P}^{{\mathcal K}}[\mathbf l',\mathbf u'])$.
Subsequently, this feature was also added to {\tt Bonmin}.
{\tt FilMint} was put together from the continuous
nonlinear programming active-set code {\tt FilterSQP},
and the mixed-integer linear programming code {\tt MINTO}.
Currently, {\tt FilMint} is only generally available
via NEOS \cite{filmintNeos}.

It is worth mentioning that just as for mixed-integer linear
programming, effective heuristics can and should be
used to provide good upper bounds quickly. This can
markedly improve the performance of any of the algorithms
described above.
Some examples of work in this direction are \cite{bonminHeur}
and \cite{efpump}.

\subsubsection{Convex quadratics and second-order cone programming}\label{PracSOCP}

Though we will not go into any details, there is
considerable algorithmic work and associated software
that seeks to leverage more specialized
(but still rather general and powerful) nonlinear models
and existing convex continuous nonlinear-programming algorithms
for the associated relaxations. In this direction, recent work has focused
on conic programming relaxations (in particular, the semi-definite and
second-order cones). On the software side, we point to work on
the binary quadratic and max-cut problems (via semi-definite
relaxation)
\cite{RRW,rendl-rinaldi-wiegele-2008} with the code {\tt Biq Mac}
\cite{BiqMac}. We also note that {\tt Cplex} (v11)
has a capability aimed at solving mixed-integer quadratically-constrained
programs that have a convex continuous relaxation. 

One important direction for approaching quadratic models is at the modeling level.
This is particulary useful for the convex case, where there is a strong and
appealing relationship between
quadratically constrained programming and second-order cone programming (SOCP).
A \emph{second-order cone constraint} is one that expresses that the Euclidean norm of
an affine function should be no more than another affine function.
An \emph{SOCP} problem consists of minimizing a linear function over
a finite set of second-order cone constraints.
Our interest in SOCP stems from the fact that (continuous) convex quadratically constrained programming problems
can be reformulated as SOCP problems
(see \cite{LVBL}).
The appeal is that very efficient interior-point algorithms
have been developed for solving SOCP problems (see \cite{GoldLiuWang}, for example),
and there is considerable mature software available that has functionality for
efficient handling of SOCP problems; see, for example:
{\tt SDPT3} \cite{sdpt3} (GNU GPL open-source license; Matlab) ,
{\tt SeDuMi} \cite{sedumiJ} (GNU GPL open-source license; Matlab),
{\tt LOQO} \cite{loqo} (proprietary; C library with interfaces to AMPL and Matlab),
{\tt MOSEK} \cite{mosek} (proprietary; C library with interface to Matlab),
{\tt Cplex} \cite{cplex} (proprietary; C library). Note also that {\tt MOSEK} and {\tt Cplex}
can handle integer variables as well; one can expect that the approaches essentially marry
specialized SOCP solvers with Branch-and-Bound and/or Outer-Approximation Algorithms. 
Further branch-and-cut methods for mixed-integer
SOCP, employing linear and convex quadratic cuts \cite{cezik-iyengar:2005} and a careful treatment
of the non-differentiability inherent in the SOCP constraints, have recently been proposed \cite{Drewes}.


Also in this vein is recent work by G\"unl\"uk and Linderoth \cite{GunlukLind,GunLind2}.
Among other things, they demonstrated that many practical mixed-integer quadratically constrained
programming formulations have substructures that admit extended formulations that can be easily strengthened
as certain integer SOCP problems. This approach is well known in the mixed-integer linear programming
literature. Let
\begin{equation*}
Q:=\Big\{w\in\R,\, \mathbf x\in\R^{n}_+,\,\mathbf z\in \{0,1\}^{n} ~:~
w  \ge\sum_{i=1}^n r_i x_i^2,\, u_i z_i\ge x_i\ge l_i z_i, \ i=1,2,\ldots,n\Big\},
\end{equation*}
where $r_i\in\R_+$ and $u_i,l_i\in \R$ for all $i=1,2,\ldots,n$.
The set $Q$ appears in several formulations as a substructure.
Consider the following extended formulation of $Q$
\begin{eqnarray*}
\bar{Q} &:=& \Big\{w\in\R,~ \mathbf x\in\R^{n},\mathbf y\in\R^{n},\mathbf z\in\R^{n} ~:~ w  \ge\sum_i r_i y_i,\\
&& \qquad (x_i,y_i,z_i)\in S_i,~~ i=1,2,\ldots,n\Big\},
\end{eqnarray*}
where
\begin{equation*}
S_i := \Big\{(x_i,y_i,z_i)\in\R^{2}\times\{0,1\} ~:~
  y_i\ge x_i^2,\ ~ u_i z_i\ge x_i\ge l_i z_i,~x_i\ge0\Big\},
\end{equation*}
and $u_i,l_i\in \R$. The convex hull of
each $S_i$ has the form 
\begin{equation*}
S_i^c~ :=~\left\{(x_i,y_i,z_i)\in\R^{3} ~:~
	y_i z_i\ge x_i^2,~ u_i z_i\ge x_i\ge l_i z_i,~1\ge z_i\ge0,~x_i,y_i\ge0\right\}
\end{equation*}
(see \cite{ceria-soares:1999,GunlukLind,GunLind2,tawarmalani-sahinidis:01}).
Note that $x_i^2-y_iz_i$ is \emph{not} a convex function, but nonetheless
$S_i^c$ is a convex set. Finally, we can state the
result of  \cite{GunlukLind,GunLind2}, which also follows from a more general result of
\cite{Hiriart-Urruty-Lemarechal:93:part-two},  that the convex hull of
the extended formulation $\bar{Q}$ has the form
\begin{eqnarray*}
\bar{Q}^c & := & \Big\{w\in\R,~ \mathbf x\in\R^{n},~ \mathbf y\in\R^{n},~ \mathbf z\in \R^{n} ~:~
w  \ge\sum_{i=1}^n r_i y_i,\\
&& \qquad (x_i,y_i,z_i)\in S^c_i,~ i=1,2,\ldots,n\Big\}.
\end{eqnarray*}
Note that all of the nonlinear constraints describing the $S_i^c$ and
$\bar{Q}^c$ are rather simple quadratic constraints.
Generally, it is well known that even the ``restricted hyperbolic constraint''
\[
y_i z_i \ge \sum_{k=1}^n x_k^2~,~ {\mathbf x}\in\R^n,~ y_i\ge 0,~  z_i \ge 0
\]
(more general than the nonconvexity in $S_i^c$)
can be reformulated as the second-order cone constraint
\[
\left\|
\left(
\begin{array}{c}
2 \mathbf x \\
y_i -z_i
\end{array}
\right)
\right\|_2
\le
y_i+z_i~.
\]

In this subsection, in the interest of concreteness and brevity, 
we have focused our attention on 
convex quadratics and second-order cone programming. 
However, it should be noted that a related approach, with broader applicability
(to all convex objective functions)
is presented in \cite{FrangGent1}, and a computational
comparison is available in \cite{FrangGent3}. Also, it is 
relevant that many convex non-quadratic functions are representable as
second-order cone programs (see \cite{AAG2009,Ben-Tal-Nemirovski:2001}).

\section{Polynomial optimization}
\label{s:general-polynomial}


In this section, we focus our attention on the study of optimization models
involving polynomials only, but without any assumptions on convexity or concavity.
It is worth emphasizing the fundamental result
of Jeroslow (Theorem \ref{th:quadopt-incomputable})
that even pure integer quadratically constrained programming
is \emph{undecidable}.
One can however avoid this daunting pitfall by bounding the variables, and
this is in fact an assumption that we should and do make for the
purpose of designing practical approaches.
From the point of view of most applications that
we are aware of, this is a very reasonable assumption.
We must be cognizant of the fact that the geometry of even quadratics on boxes is
daunting from the point of view of mathematical programming; see Figure \ref{fig:quadr}.
Specifically, the convex envelope of the graph of the product $x_1x_2$ on a box
deviates badly from the graph, so relaxation-based methods are intrinsically handicapped.
It is easy to see, for example, that for $\delta_1,\delta_2>0$, $x_1x_2$ is strictly
convex on the line segment joining $(0,0)$ and $(\delta_1,\delta_2)$; while
$x_1x_2$ is strictly
concave on the line segment joining $(\delta_1,0)$ and $(0,\delta_2)$ .

  \begin{figure}[h]
    \centering
    \ifpdf
    \includegraphics[height=5cm]{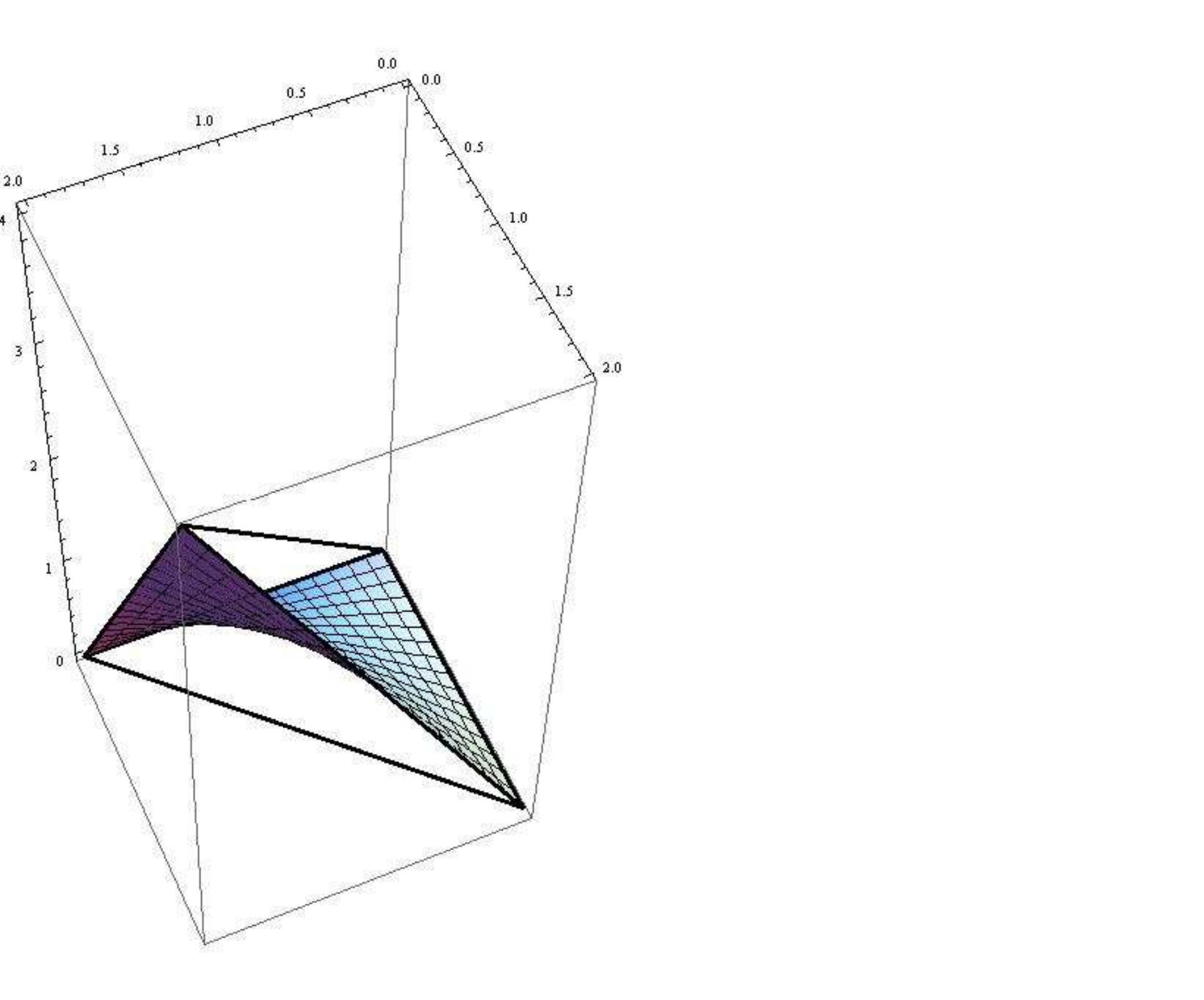}
    \else
    \includegraphics[height=5cm]{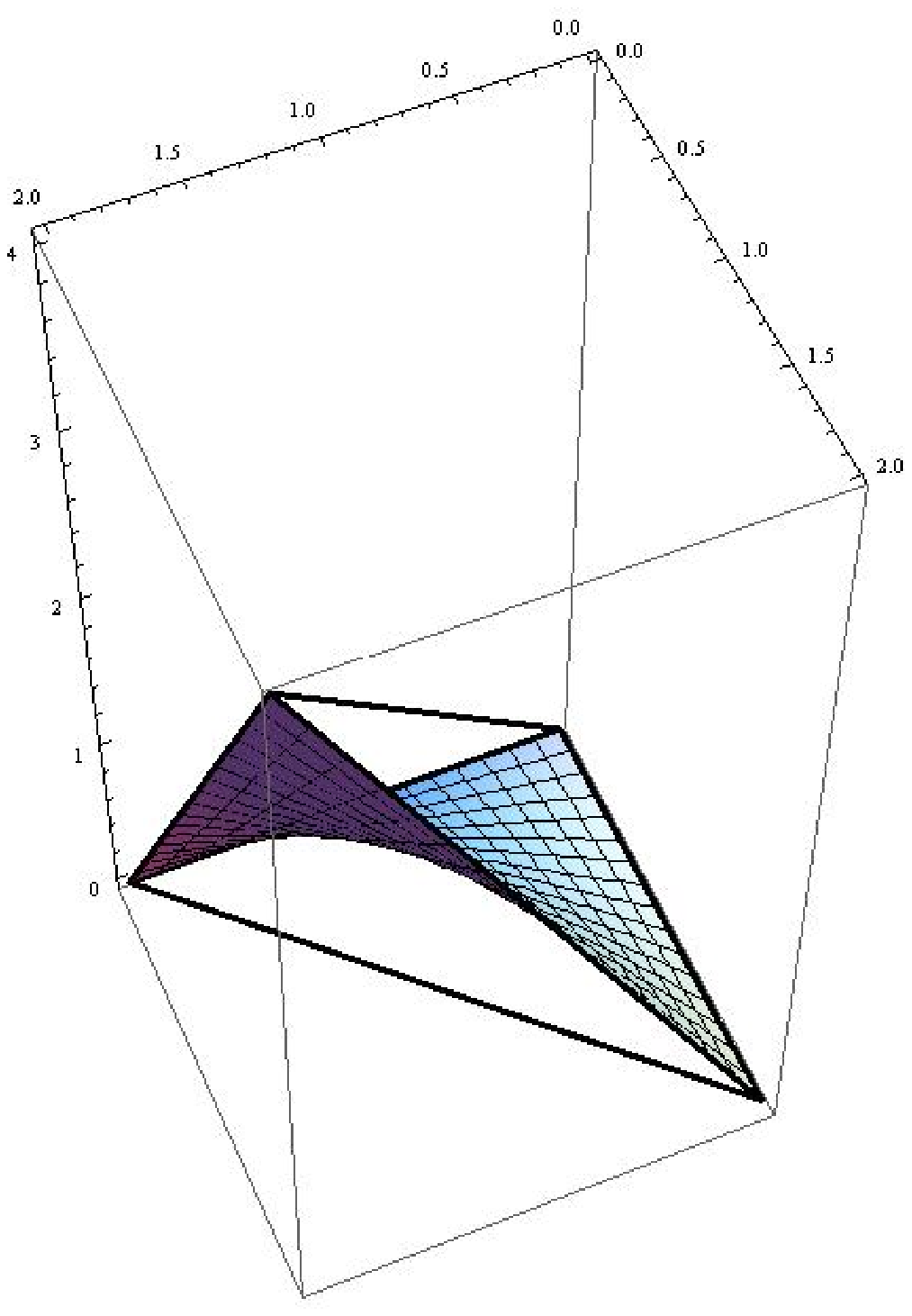}
    \fi
    \caption{Tetrahedral convex envelope of the graph of the product $x_1x_2$ on a box
    }
    \label{fig:quadr}
  \end{figure}%

Despite these difficulties, we have positive results.
In Section \ref{s:fptas}, the highlight is
a \emph{fully polynomial time approximation scheme (FPTAS)} for problems
involving maximization of a polynomial in fixed dimension,
over the mixed-integer points in a polytope. In Section
\ref{s:sos-programming},  we broaden our focus to allow
feasible regions defined by inequalities in polynomials (i.e.,
semi-algebraic sets). In this setting, we do not present (nor could we expect)
complexity results as strong as for linear constraints, but rather we
show how tools of semi-definite programming are being developed to
provide, in a systematic manner, strengthened relaxations.
Finally, in Section \ref{s:quadratics}, we describe recent
computational advances in the special case of semi-algebraic programming
for which all of the functions are quadratic --- i.e.,
\emph{mixed-integer quadratically constrained programming} (MIQCP).

\subsection{Fixed dimension and linear constraints: An FPTAS}
\label{s:fptas}

As we pointed out in the introduction (Theorem \ref{th:deg4-dim2-hard}),
optimizing degree-4 polynomials over problems with two integer variables is
already a hard problem.  Thus, even when we fix the dimension, we cannot get a
polynomial-time algorithm for solving the optimization problem.  The best we
can hope for, even when the number of both the continuous and the integer
variables is fixed, is an approximation result.

\begin{definition}{\bf(FPTAS)}
\begin{enumerate}[\rm(a)]
\item An algorithm ${\cal  A}$
is an  \emph{${\epsilon}$-approximation algorithm}
for a maximization  problem with optimal cost $f_{\max}$,
if for each instance  of the problem of encoding length~$n$,
${\cal  A}$  runs in polynomial time in $n$ and returns
a feasible solution with cost   $f_{\rm {\cal  A}} $, such that
$f_{\rm {\cal  A}}    \geq (1-\epsilon) \cdot  f_{\max}$.
\item A family $\{\mathcal A_\epsilon\}_\epsilon$ of $\epsilon$-approximation
  algorithms is a \emph{fully polynomial time approximation scheme (FPTAS)} if
  the running time of~$\mathcal A_\epsilon$ is polynomial in the encoding size
  of the instance and $1/\epsilon$.
\end{enumerate}
\end{definition}
Indeed it is possible to obtain an FPTAS for general polynomial optimization
of mixed-integer feasible sets in polytopes
\cite{deloera-hemmecke-koeppe-weismantel:intpoly-fixeddim,DeloeraHemmeckeKoeppeWeismantel06,deloera-hemmecke-koeppe-weismantel:mixedintpoly-fixeddim-fullpaper}.  To
explain the method of the FPTAS, we need to review the
theory of \emph{short rational generating functions} pioneered by Barvinok
\cite{Barvinok94,BarviPom}.  The FPTAS itself appears in Section \ref{s:actual-fptas}.

\subsubsection{Introduction to rational generating functions}

We explain the theory on a simple, one-dimensional example.  Let us consider the
set~$S$ of integers in the interval $P=[0,\dots,n]$; see
the top of Figure \ref{fig:1d-brion}\,(a).
We associate with~$S$ the polynomial
\begin{math}
  g(S; z) = z^0 + z^1 + \dots + z^{n-1} + z^n;
\end{math}
i.e., every integer $\alpha\in S$ corresponds to a monomial~$z^\alpha$
with coefficient~$1$ in the polynomial~$g(S; z)$.
This polynomial is called the \emph{generating function} of~$S$ (or of~$P$).
From the viewpoint of computational complexity, this generating function is of
exponential size (in the encoding length of~$n$), just as an explicit list of
all the integers~$0$, $1$, \dots, $n-1$, $n$ would be.  However, we can observe that
$g(S; z)$ 
is a finite geometric series, so there exists
a simple summation formula that expresses
it 
in a much more compact way:
\begin{equation}
  \label{eq:1d-example-formula-summed}
  g(S; z) = z^0 + z^1 + \dots + z^{n-1} + z^n
  = \frac{1-z^{n+1}}{1-z}.
\end{equation}
The ``long'' polynomial has a ``short'' representation as a rational function.
The encoding length of this new formula is \emph{linear} in the encoding
length of~$n$.

Suppose now someone presents to us a finite set~$S$ of integers as a
generating
function $g(S;z)$.  Can we decide whether the set is nonempty?  In fact, we
can do something much stronger even -- we can \emph{count} the integers in the
set~$S$, simply by evaluating at $g(S;z)$ at $z=1$.  On our example we have
\begin{math}
  |S| = g(S;1) = 1^0 + 1^1 + \dots + 1^{n-1} + 1^n = n + 1.
\end{math}
We can do the same on the shorter, rational-function formula
if we are careful with the (removable) singularity $z=1$.  We just compute
the limit
\begin{displaymath}
  |S| = \lim_{z\to1} g(S;z) = \lim_{z\to1} \frac{1-z^{n+1}}{1-z}
  = \lim_{z\to1} \frac{-(n+1) z^n}{-1} = n+1
\end{displaymath}
using the Bernoulli--l'H\^opital rule.  Note that we have avoided to carry out
a polynomial division, which would have given us the long polynomial again.

The summation formula~\eqref{eq:1d-example-formula-summed} can
also be written in a slightly different way:
\begin{equation}
  \label{eq:1d-example-formula-summed-basic}
  g(S; z)
  = \frac{1}{1-z} - \frac{z^{n+1}}{1-z}
  = \frac{1}{1-z} + \frac{z^n}{1-z^{-1}}
\end{equation}
Each of the two summands on the right-hand side can be viewed as the summation
formula of an infinite geometric series:
\begin{subequations}
  \label{eq:1d-example-summands}
  \begin{align}
    g_1(z) &= \frac{1}{1-z} = z^0 + z^1 + z^2 + \dots,\\
    g_2(z) &= \frac{z^n}{1-z^{-1}} = z^n + z^{n-1} + z^{n-2} + \dots.
  \end{align}
\end{subequations}
The two summands have a geometrical interpretation.  If we view each geometric
series as the generating function of an (infinite) lattice point set, we
arrive at the picture shown in Figure \ref{fig:1d-brion}.
\begin{figure}[t]
  \centering
  \hfil(a)\quad\subfigure{%
    \begin{minipage}[b]{.6\linewidth}
      \ifpdf
    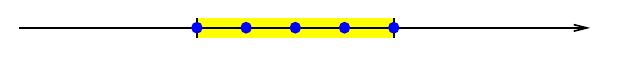
    \else
    \input{1d-example-0.pstex_t}
    \fi\par
      $=$\par\vspace{1ex} \ifpdf
    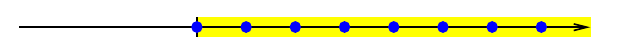
    \else
    \input{1d-example-s1.pstex_t}
    \fi\par
      $+$\par\vspace{1ex} \ifpdf
    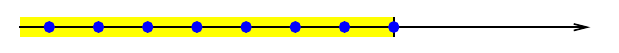
    \else
    \input{1d-example-s2.pstex_t}
    \fi
    \end{minipage}
  }
  \hfil(b)\quad\subfigure{
    \begin{minipage}[b]{.25\linewidth}
      \ifpdf
    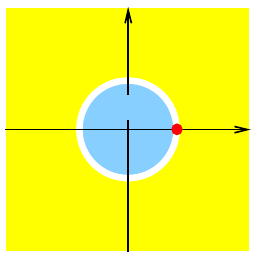
    \else
    \input{1d-convergence.pstex_t}
    \fi
    \end{minipage}
  }
  \caption{\textbf{(a)} One-dimensional Brion theorem. \textbf{(b)} The domains of convergence
    of the Laurent series.}
  \label{fig:1d-brion}
  \label{fig:1d-convergence}
\end{figure}
We remark that all integer points in the
interval $[0,n]$ are covered twice, and also all integer points outside the
interval are covered once.
This phenomenon is due to the \emph{one-to-many
  correspondence} of rational functions to their Laurent series.  When we
consider Laurent series of the function $g_1(z)$ about $z=0$, the pole $z=1$
splits the complex plane into two domains of convergence
(Figure \ref{fig:1d-convergence}):  For $|z| < 1$, the
power series
\begin{math}
  z^0 + z^1 + z^2 + \dots
\end{math}
converges to $g_1(z)$.  As a matter of fact, it converges absolutely and
uniformly on every compact subset of the open circle $\{\, z\in\C: |z| < 1
\,\}$.  For $|z| > 1$, however, the
series 
diverges.  On the other hand,
the Laurent series
\begin{math}
  -z^{-1} - z^{-2} - z^{-3} - \dots
\end{math}
converges (absolutely and compact-uniformly) on the open circular ring $\{\,
z\in\C: |z| > 1 \,\}$ to the function $g_1(z)$, whereas it diverges for $|z| <
1$.  The same holds for~$g_2(z)$.
Altogether we have:
\begin{align}
  g_1(z) & =
  \begin{cases}
    z^0 + z^1 + z^2 + \dots & \text{for $|z| < 1$} \\
    -z^{-1} - z^{-2} - z^{-3} - \dots & \text{for $|z| > 1$}
  \end{cases}\\
  g_2(z) &=
  \begin{cases}
    -z^{n+1} - z^{n+2} - z^{n+3} - \dots & \text{for $|z| < 1$}\\
    z^n + z^{n-1} + z^{n-2} + \dots & \text{for $|z| > 1$}
  \end{cases}
\end{align}
We can now see that the phenomenon we observed in
formula~\eqref{eq:1d-example-summands} and Figure \ref{fig:1d-brion} is due to
the fact that we had picked two Laurent series for the summands $g_1(z)$ and
$g_2(z)$ that do not have a common domain of convergence;
the situation of
formula~\eqref{eq:1d-example-summands} and Figure \ref{fig:1d-brion}
appears again in the
$d$-dimensional case as \emph{Brion's Theorem}. \smallbreak

Let us now consider a \emph{two-dimensional} cone~$C$ spanned by the vectors $\mathbf b_1=(\alpha, -1)$ and $\mathbf
b_2=(\beta, 1)$; see Figure \ref{fig:2d-friendly-tiling} for an example with
$\alpha=2$ and $\beta=4$.
\begin{figure}[t]
  \centering
  \hfil(a)\quad\subfigure{%
    \begin{minipage}[b]{.52\linewidth}
      \ifpdf
    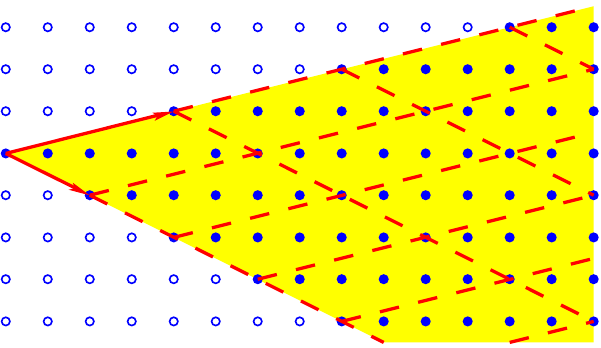
    \else
    \input{2d-friendly-tiling.pstex_t}
    \fi
    \end{minipage}
  }
  \hfil(b)\quad\subfigure{
    \begin{minipage}[b]{.3\linewidth}
      \ifpdf
    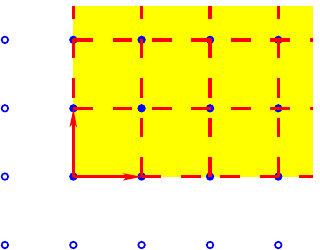
    \else
    \input{2d-friendly-tiling-2.pstex_t}
    \fi
    \end{minipage}
  }
\caption{\textbf{(a)} Tiling a rational two-dimensional cone with copies of the
  fundamental parallelepiped. \textbf{(b)} The semigroup~$S\subseteq\Z^2$ generated by
  $\mathbf b_1$ and $\mathbf b_2$ is a linear image of~$\Z_+^2$}
\label{fig:2d-friendly-tiling}
  \label{fig:2d-friendly-tiling-2}
\end{figure}
We would like to write down a generating function for the integer points in
this cone.  We apparently need a generalization of the geometric series, of
which we made use in the one-dimensional case.  The key observation now is
that using copies of the half-open \emph{fundamental parallelepiped},
\begin{math}
  \Pi = \bigl\{\, \lambda_1 \mathbf b_1 + \lambda_2 \mathbf b_2 :
  \lambda_1 \in [0,1), \lambda_2 \in [0,1) \,\bigr\},
\end{math}
the cone can be \emph{tiled}:
\begin{equation}
  C = \bigcup_{\mathbf s\in S} (\mathbf s + \Pi)
  \quad\text{where}\quad
  S = \{\, \mu_1 \mathbf b_1 + \mu_2 \mathbf b_2 : (\mu_1,\mu_2)\in\Z_+^2 \,\}
\end{equation}
(a disjoint union).  Because we have chosen \emph{integral}
generators $\mathbf b_1, \mathbf b_2$, the integer points are ``the
same'' in each copy of the fundamental parallelepiped.  Therefore, also the
integer points of~$C$ can be tiled by copies of~$\Pi\cap\Z^2$;
on the other hand, we can see $C\cap\Z^2$ as a finite disjoint union of copies
of~$S$, shifted by the integer points of~$\Pi$:
\begin{equation}
  C\cap\Z^2
  = \bigcup_{\mathbf s\in S}
  \bigl({\mathbf s + (\Pi\cap \Z^2)}\bigr)
  = \bigcup_{\mathbf x\in \Pi\cap \Z^2} (\mathbf x + S).
  \label{eq:shifted-copies-of-S}
\end{equation}
The set~$S$ is just the image of
$\Z_+^2$ under the matrix $(\mathbf b_1, \mathbf b_2)=\smash{\bigl(\begin{smallmatrix}
  \alpha & \beta \\
  -1 & 1
\end{smallmatrix}\bigr)}$;
cf. Figure \ref{fig:2d-friendly-tiling}. 
Now $\Z_+^2$ is the direct product of~$\Z_+$ with itself,
whose generating function is the geometric series
\begin{math}
  g(\Z_+; z) = z^0 + z^1 + z^2 + z^3 + \dots = \frac1{1-z}.
\end{math}
We thus obtain the generating function as a product,
\begin{math}
  g(\Z_+^2; z_1,z_2)
  =  \frac1{1-z_1} \cdot \frac1{1-z_2}.
\end{math}
Applying the linear transformation~$(\mathbf b_1,\mathbf b_2)$, 
\begin{displaymath}
  g(S; z_1,z_2) = \frac1{(1-z_1^\alpha z_2^{-1})(1-z_1^\beta z_2^1)}.
\end{displaymath}
From 
\eqref{eq:shifted-copies-of-S} it is now clear that
\begin{math}
  g(C; z_1,z_2) = \sum_{\mathbf x\in\Pi\cap\Z^2} z_1^{x_1} z_2^{x_2} g(S;z_1,z_2);
\end{math}
the multiplication with the monomial $z_1^{x_1} z_2^{x_2}$ corresponds to the
shifting of the set~$S$ by the vector $(x_1, x_2)$.
In our example, it is easy to see that
\begin{math}
  \Pi\cap\Z^2 = \{\, (i, 0): i = 0, \dots, \alpha+\beta-1 \,\}.
\end{math}
Thus
\begin{displaymath}
  g(C;z_1,z_2) = \frac{z_1^0 + z_1^1 + \dots + z_1^{\alpha+\beta-2} +
  z_1^{\alpha+\beta-1} }{(1-z_1^\alpha z_2^{-1})(1-z_1^\beta z_2^1)}.
\end{displaymath}
Unfortunately, this formula has an exponential size as the numerator contains
$\alpha+\beta$ summands.
To make the formula shorter, we need to recursively break the cone into
``smaller'' cones, each of which have a much shorter formula.  We have observed
that the length of the formula is
determined by the number of integer points in the fundamental parallelepiped,
the \emph{index} of the cone.
\emph{Triangulations} usually do not help to reduce the
index significantly, as another two-dimensional example shows.  Consider the cone $C'$ generated by $\mathbf
b_1=(1,0)$ and $\mathbf b_2=(1,\alpha)$; see Figure \ref{fig:positive-decomp}.
\begin{figure}[t]
  \begin{center}
    (a)
    \begin{minipage}[b]{.28\linewidth}
      \ifpdf
    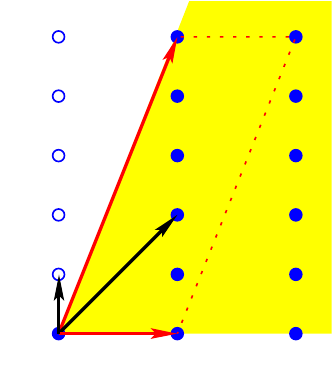
    \else
    \input{positive-decomp-1.pstex_t}
    \fi
    \end{minipage}\quad
    (b)
    \begin{minipage}[b]{.28\linewidth}
      \ifpdf
    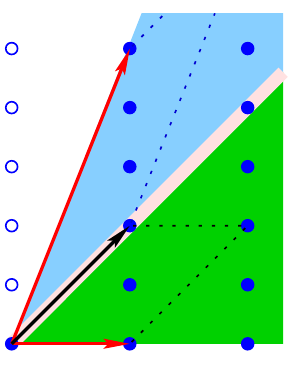
    \else
    \input{positive-decomp-2.pstex_t}
    \fi
    \end{minipage}\quad
    (c)
    \begin{minipage}[b]{.28\linewidth}
      \ifpdf
    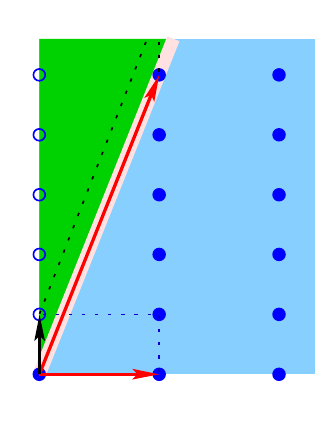
    \else
    \input{signed-decomp-3.pstex_t}
    \fi
    \end{minipage}
  \end{center}
  \caption[A triangulation and a signed decomposition of a cone]
  {\textbf{(a)} A cone  of index~$5$ generated by $\mathbf b^1$ and $\mathbf b^2$.
    \textbf{(b)}
    A triangulation of the cone into the two cones spanned by $\{\mathbf b^1,\mathbf w\}$ and $\{\mathbf
    b^2,\mathbf w\}$, having an index of $2$ and~$3$, respectively.  We have
    the inclusion-exclusion formula $g({\mathop{\mathrm{cone}}\{\mathbf b_1,\mathbf b_2\}};\mathbf z) =
    g({\mathop{\mathrm{cone}}\{\mathbf b_1,\mathbf w\}}; \mathbf z) + g({\mathop{\mathrm{cone}}\{\mathbf b_2,\mathbf w\}};\mathbf z) -
    g({\mathop{\mathrm{cone}}\{\mathbf w\}};\mathbf z)$;
    here the one-dimensional cone spanned by $\mathbf w$
    needed to be subtracted.
    \textbf{(c)} A signed decomposition into the two unimodular cones spanned by $\{\mathbf
    b^1,\mathbf w'\}$ and $\{\mathbf
    b^2,\mathbf w'\}$.  We have the inclusion-\penalty0 exclusion formula $g({\mathop{\mathrm{cone}}\{\mathbf b_1,\mathbf b_2\}};\mathbf z)
    = g({\mathop{\mathrm{cone}}\{\mathbf b_1,\mathbf w'\}};\mathbf z) - g({\mathop{\mathrm{cone}}\{\mathbf b_2,\mathbf w'\}};\mathbf z) +
    g({\mathop{\mathrm{cone}}\{\mathbf w'\}};\mathbf z)$.}
  \label{fig:positive-decomp}
  \label{fig:signed-decomp}
\end{figure}
We have
\begin{math}
  \Pi'\cap \Z^2 = \{(0,0)\} \cup \{\, (1,i) : i = 1,\dots,\alpha-1 \,\},
\end{math}
so the rational generating function would have $\alpha$ summands in the
numerator, and thus have exponential size.  Every attempt to use
triangulations to reduce the size of the formula fails in this example.  The
choice of an interior vector~$\mathbf w$ in Figure \ref{fig:positive-decomp}, for
instance, splits the cone of index~$5$ into two cones of index~$2$ and $3$,
respectively -- and also a one-dimensional cone.  Indeed, every possible
triangulation of~$C'$ into unimodular cones contains at least~$\alpha$
two-dimensional cones!
The important new idea by Barvinok was to use so-called \emph{signed
  decompositions} in addition to triangulations in order to reduce the index
of a cone.  In our example, we can
choose the vector $\mathbf w = (0,1)$ \emph{from the outside of the cone}
to define cones $C_1 = \mathop{\mathrm{cone}}\{\mathbf b_1,\mathbf w\}$ and $C_2 = \mathop{\mathrm{cone}}\{\mathbf w, \mathbf
b_2\}$; see Figure \ref{fig:signed-decomp}.
Using these cones, we have
the inclusion-exclusion formula
\begin{displaymath}
g(C';z_1,z_2) = g(C_1;z_1,z_2) - g(C_2;z_1,z_2) + g(C_1\cap C_2; z_1,z_2)
\end{displaymath}
It turns out that all cones $C_1$ and $C_2$ are unimodular,
and we obtain the rational generating function by summing up those of the subcones,
\begin{align*}
  g(C';z_1,z_2) &= \frac{1}{(1-z_1)(1-z_2)} - \frac{1}{(1-z_1^1 z_2^\alpha)
    (1-z_2)} + \frac{1}{1-z_1 z_2^\alpha}.
\end{align*}

\subsubsection{Barvinok's algorithm for short rational generating functions}

We now present the general definitions and results.  Let $P\subseteq\R^d$ be a
rational polyhedron.  We first define its \emph{generating function} as the
\emph{formal Laurent series} $\tilde g(P; \mathbf z) = \sum_{{\bm{\alpha}}\in
  P\cap\Z^d} \mathbf z^{{\bm{\alpha}}} \in \Z[[z_1,\dots,z_d,
z_1^{-1},\dots,z_d^{-1}]]$, i.e., without any consideration of convergence
properties.  (A formal \emph{power} series is not enough because monomials with
negative exponents can appear.) As we remarked above, this encoding of a set
of lattice points does not give an immediate benefit in terms of complexity.
We will get short formulas only when we can identify the Laurent series with
certain rational functions.
Now if $P$ is a polytope
, then $\tilde g(P; \mathbf z)$ is a
\emph{Laurent polynomial} (i.e., a \emph{finite}
sum of monomials with positive or negative integer exponents),
so it can be naturally identified with a
rational function~$g(P; \mathbf z)$.  
Convergence comes into play whenever $P$ is not bounded,
since then $\tilde g(P; \mathbf z)$ can be an infinite formal sum.
We first consider a \emph{pointed} polyhedron~$P$, i.e., $P$ does
not contain a straight line.
\begin{theorem}\label{th:laurent-series-to-ratfun}
  Let $P\subseteq\R^d$ be a pointed rational polyhedron.  Then there exists a
  non-empty open subset $U\subseteq\C^d$ such that the series~$\tilde g(P; \mathbf
  z)$ converges absolutely and uniformly on every compact subset of~$U$ to a
  rational function~$g(P; \mathbf z)\in\Q(z_1,\dots,z_d)$.
\end{theorem}
Finally, when $P$ contains an integer point and also a straight line, there
does not exist any point~$\mathbf z\in\C^d$ where the series $\tilde g(P; \mathbf z)$
converges absolutely
. In this case
we set
\begin{math}
  g(P;\mathbf z) = 0;
\end{math}
this turns out to be a consistent choice (making the map $P\mapsto g(P;\mathbf z)$ a
\emph{valuation}, i.e., a finitely additive measure).
  The rational function $g(P;\mathbf z)\in\Q(z_1,\dots,z_d)$ defined as described
  above is called the \emph{rational generating function} of~$P\cap\Z^d$.
\begin{figure}[t]
  \centering
  \hfil(a)\subfigure{\ifpdf
    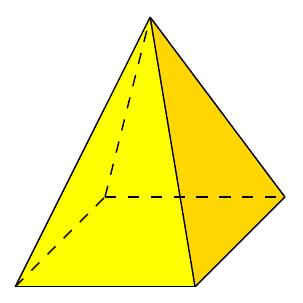
    \else
    \input{brion0.pstex_t}
    \fi}
  \hfil(b)\subfigure{\ifpdf
    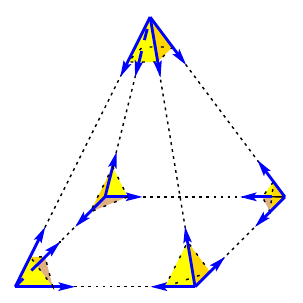
    \else
    \input{brion.pstex_t}
    \fi}
  \caption{Brion's theorem, expressing the generating function of a polyhedron
    to those of the supporting cones of all vertices}
  \label{fig:brion}
\end{figure}
\begin{theorem}[Brion \cite{Brion88}] \label{lem:bl}
  Let ${P}$ be a rational polyhedron
  and $V({P})$ be the set of vertices of~${P}$. Then,
  \begin{math}
    g({P};\mathbf{z}) = \sum_{\mathbf v \in V({P})} g({C}_{P}(\mathbf v);\mathbf{z}),
  \end{math}
  where ${C}_{P}(\mathbf v)=\mathbf v + \mathop{\mathrm{cone}}(P - \mathbf v)$ is the \emph{supporting
    cone} of the vertex~$\mathbf v$; see Figure \ref{fig:brion}.
\end{theorem}
We remark that in the case of a non-pointed polyhedron~$P$, i.e., a polyhedron
that has no vertices because it contains a straight line, both sides of the
equation are zero.\smallbreak

Barvinok's algorithm computes the rational generating function of a
polyhedron~$P$ as follows.  By Brion's theorem,
the rational generating function of a
polyhedron can be expressed as the sum of the rational generating functions of
the supporting cones of its vertices.
Every supporting
cone $\mathbf v_i+C_i$ can be triangulated to obtain simplicial cones $\mathbf
v_i+C_{ij}$.
If the dimension is fixed, these polyhedral computations all run in polynomial
time.

Now let $K$ be one of these simplicial cones,
whose \emph{basis vectors} $\mathbf b_1,\dots,\mathbf b_d$ (i.e., primitive representatives
of its extreme rays) are
the columns of some matrix $B\in\Z^{d\times d}$;
then the index of $K$ is $\left|\det B\right|$.
Barvinok's algorithm now
computes a \emph{signed decomposition} of~$K$ to
produce simplicial cones with smaller index.  To this end, it
constructs a vector $\mathbf w=\alpha_1 \mathbf b_1 + \dots + \alpha_d \mathbf
b_d\in\smash{\Z^d\setminus\{\mathbf0\}}$ with
\begin{math}
  \mathopen|\alpha_i\mathclose| \leq \mathopen|\det B\mathclose|^{-1/d}.
\end{math}
The existence of such a vector follows from Minkowski's first theorem,
and it can be constructed in polynomial time using integer programming or lattice basis
reduction followed by enumeration.
The cone is then decomposed into cones spanned by $d$ vectors from the set
$\{\mathbf b_1,\dots,\mathbf b_d,\mathbf w\}$; each of the resulting cones then has an
index at most $(\mathop{\mathrm{ind}} K)^{(d-1)/d}$.  In general, these cones form a
signed decomposition of~$K$; only if $\mathbf w$ lies inside~$K$,
they form a triangulation (see Figure \ref{fig:positive-decomp}).
The resulting cones and their intersecting proper faces (arising in an
inclusion-exclusion formula) are recursively processed, until
cones of low index (for instance unimodular cones) are obtained.
Finally, for a unimodular cone $\mathbf v+B\R_+^d$, the rational generating
function is
\begin{math}
  {\mathbf z^{\mathbf a}}/{\prod_{j=1}^d (1-\mathbf z^{\mathbf b_j})},
\end{math}
where $\mathbf a$ is the unique integer point in the fundamental parallelepiped.
We summarize Barvinok's algorithm below.\medbreak

\begin{algorithm}[Barvinok's algorithm]~\smallskip\par
\label{algo:primal-barvi}
\noindent {\em Input:}
A polyhedron $P\subset\R^d$ given by rational inequalities.

\noindent {\em Output:}
The rational generating function for $P\cap\Z^d$ in the
form
\begin{equation}
  \label{eq:generating-function-0}
  g_P(\mathbf z) = \sum_{i\in I} \epsilon_i \frac{\mathbf
    z^{\mathbf a_i}}{\prod_{j=1}^d (1-\mathbf z^{\mathbf b_{ij}})}
\end{equation}
where $\epsilon_i\in\{\pm1\}$, $\mathbf a_i\in\Z^d$, and
$\mathbf b_{ij}\in\Z^d$.
\begin{enumerate}[\rm\ 1.]
\item Compute all vertices $\mathbf v_i$ and corresponding
  supporting cones $C_i$ of $P$.
\item Triangulate $C_i$ into simplicial cones $C_{ij}$, keeping track of all
  the intersecting proper faces.
\item Apply signed decomposition to the cones $\mathbf v_i + C_{ij}$
  to obtain unimodular cones $\mathbf v_i + C_{ijl}$, keeping track of all the
  intersecting proper faces.
\item Compute the unique integer point~$\mathbf a_i$ in the fundamental parallelepiped
  of every resulting cone~$\mathbf v_i+C_{ijl}$.
\item Write down the formula~\eqref{eq:generating-function-0}.
\end{enumerate}
\end{algorithm}
We remark that it is possible to avoid computations with the intersecting proper
faces of cones (step~2 of the algorithm) entirely, using techniques such as polarization, irrational
decomposition \cite{koeppe:irrational-barvinok}, or half-open decomposition
\cite{Brion1997residue,koeppe-verdoolaege:parametric}.\smallbreak

Due to the descent of the indices in the signed decomposition
procedure, 
the depth of the decomposition tree is at most
\begin{math}
  \bigl\lfloor{1 + \frac{ \log_2 \log_2 D }{\log_2 \frac{d}{d-1} } }\bigr\rfloor,
\end{math}
where $D=\mathopen|\det B\mathclose|$.
Because at each decomposition step at most $\mathrm O(2^d)$ cones are created
and the depth of the tree is doubly logarithmic in the index of the input
cone, Barvinok could obtain a polynomiality result \emph{in fixed dimension}:
\begin{theorem}[Barvinok \cite{Barvinok94}]
  Let $d$ be fixed.  There exists a polynomial-time algorithm for computing
  the rational generating function \eqref{eq:generating-function-0} of a polyhedron $P\subseteq\R^d$ given by rational
  inequalities.
\end{theorem}

\subsubsection{The FPTAS for polynomial optimization}
\label{s:actual-fptas}

We now describe the fully polynomial-time approximation scheme,
which appeared in \cite{deloera-hemmecke-koeppe-weismantel:intpoly-fixeddim,DeloeraHemmeckeKoeppeWeismantel06,deloera-hemmecke-koeppe-weismantel:mixedintpoly-fixeddim-fullpaper}.
It makes use of the elementary relation
\begin{equation}
  \max\{s_1, \dots, s_N \} = \lim_{k \rightarrow \infty} \sqrt[k]{s_1^k +
    \dots + s_N^k},
\end{equation}
which holds for any finite set $S=\{s_1,\dots,s_N\}$ of non-negative real
numbers.  This relation can be viewed as an approximation result for
$\ell_k$-norms.  Now if $P$ is a polytope and $f$ is an objective function
non-negative on $P\cap\Z^d$, let $\mathbf x^1,\dots,\mathbf x^{N}$ denote all the feasible
integer solutions in~$P\cap\Z^d$ and collect their objective function values
$s_i=f(\mathbf x^i)$
in a vector~$\mathbf s\in\Q^N$.  Then, comparing the unit balls of the $\ell_k$-norm
and the $\ell_\infty$-norm (Figure \ref{fig:lp-norms}), we get the relation
\begin{displaymath}
  L_k :=
  {{N}}^{-1/k} \mathopen\| \mathbf s \mathclose\|_k \leq {\mathopen\| \mathbf s \mathclose\|_\infty}
  \leq \mathopen\| \mathbf s\mathclose\|_k
  =: U_k.
\end{displaymath}%
\begin{figure}[t]
  \centering
  $k=1$
  \ifpdf
    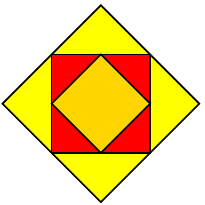
    \else
    \input{l1-norm.pstex_t}
    \fi
  \qquad
  $k=2$
  \ifpdf
    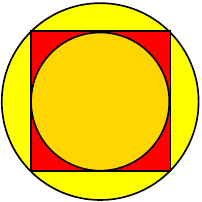
    \else
    \input{l2-norm.pstex_t}
    \fi
  \caption{Approximation properties of $\ell_k$-norms}
  \label{fig:lp-norms}
\end{figure}%
Thus, for obtaining a good approximation of the maximum, it suffices to solve
a summation problem of the polynomial function $h=f^k$ on $P\cap\Z^d$ for a
value of~$k$ that is large enough. Indeed, for $k= \bigl\lceil{(1+1/\epsilon)\log
  {{N}}}\bigr\rceil$, we obtain $U_k - L_k\leq \epsilon f(\mathbf x^{\max})$.  On the other
hand, this choice of~$k$ is polynomial in the input size (because $1/\epsilon$
is encoded in unary in the input, and $\log N$ is bounded by a polynomial in
the binary encoding size of the polytope~$P$).  Hence, when the dimension~$d$ is
fixed, we can expand the polynomial function~$f^k$ as a list of
monomials in polynomial time.

Solving the summation problem can be accomplished using short rational
generating functions as follows.
Let $g(P;\mathbf z)$ be the rational generating function of~$P\cap\Z^d$,
computed using Barvinok's algorithm.
By symbolically applying differential operators to~$g(P;\mathbf z)$, we can
compute a short rational function
representation of the Laurent polynomial
\begin{math}
  g(P,h; \mathbf z) = \sum_{{\bm{\alpha}}\in P \cap \Z^d} h({\bm{\alpha}}) \mathbf z^{{\bm{\alpha}}},
\end{math}
where each monomial $\mathbf z^{{\bm{\alpha}}}$ corresponding to an integer
point ${\bm{\alpha}}\in P\cap\Z^d$ has a coefficient that is the
value~$h({\bm{\alpha}})$.
To illustrate this, consider again the generating
function of the interval~$P=[0,4]$,
\begin{displaymath}
  g_P(z) = z^0 + z^1 + z^2 + z^3 + z^4
  {= \frac1{1-z} - \frac{z^5}{1-z}}.
\end{displaymath}
We now apply the differential operator $ z \frac{\mathrm d}{\mathrm d
    z} $ and obtain
\begin{displaymath}
  \left(z \frac{\mathrm d}{\mathrm d z}\right) g_P(z)
  = {1} z^1 + {2} z^2 + {3} z^3 + {4} z^4
  {= \frac1{(1-z)^2} - \frac{-4z^5 + 5z^4}{(1-z)^2}}
\end{displaymath}
Applying the same differential operator again, we obtain
\begin{displaymath}
  \left(z \frac{\mathrm d}{\mathrm d z}\right) \left( z
    \frac{\mathrm d}{\mathrm d z} \right) g_P(z)
  = {1} z^1 + {4} z^2 + {9}
  z^3 + {16} z^4
  {= \frac{z + z^2}{(1-z)^{3}}
    - \frac{25z^5 -39z^6+16 z^7}{(1-z)^{3}}}
\end{displaymath}
We have thus evaluated the monomial function $h(\alpha)=\alpha^2$ for
$\alpha=0,\dots,4$; the results appear as the coefficients of the
respective monomials.  The same works for several variables, using the partial
differential operators $z_i \frac\partial{\partial z_i}$ for $i =1,\dots,d$.
In fixed dimension, the size of the rational function expressions occuring in the symbolic
calculation can be bounded polynomially.  Thus one obtains the following result.

\begin{theorem}
  \label{operators}
  \begin{enumerate}[\rm(a)]
  \item Let $h(x_1,\dots,x_d) = \sum_{{\bm{\beta}}}c_{\bm{\beta}} \mathbf x^{{\bm{\beta}}}
    \in\Q[x_1,\dots,x_d]$ be a polynomial.  Define the differential operator
    $$D_h = h\left(z_1\frac{\partial}{\partial z_1},\dots, z_d\frac{\partial}{\partial
        z_d}\right) = \sum_{{\bm{\beta}}} c_{{\bm{\beta}}}
    \left(z_1\frac{\partial}{\partial z_1}\right)^{\beta_1}\dots
    \left(z_d\frac{\partial}{\partial z_d}\right)^{\beta_d}.$$
    Then $D_h$ maps the generating function $g(P;\mathbf z) = \sum_{{\bm{\alpha}}\in P\cap\Z^d}
    \mathbf z^{{\bm{\alpha}}}$ to the weighted generating function
    $(D_h g)(\mathbf z) = g(P, h; \mathbf z) = \sum_{{\bm{\alpha}}\in P\cap\Z^d} h({\bm{\alpha}}) \mathbf z^{{\bm{\alpha}}}$.
  \item
    Let the dimension $d$ be fixed.  Let $g(P;\mathbf z)$ be the Barvinok representation of the
    generating function $\sum_{{\bm{\alpha}}\in P
      \cap \Z^d}\mathbf z^{{\bm{\alpha}}}$ of $P\cap\Z^d$. Let $h\in\Q[x_1,\dots,x_d]$ be a
    polynomial, given as a list of monomials with rational coefficients~$c_{{\bm{\beta}}}$
    encoded in binary and exponents~${\bm{\beta}}$ encoded in unary. We can compute in
    polynomial time a Barvinok representation $g(P,h;\mathbf z)$ for the
    weighted generating function $\sum_{{\bm{\alpha}}\in P \cap \Z^d}
    h({\bm{\alpha}}) \mathbf z^{{\bm{\alpha}}}.$
  \end{enumerate}
\end{theorem}

Thus, we can implement the following algorithm in polynomial time.

\begin{algorithm}[Computation of bounds for the optimal value]~\smallskip
\label{Algorithm}

\noindent {\em Input:} A rational convex polytope $P \subset \R^d$;
a polynomial objective function $f \in \Q[x_1,\dots,x_d]$ 
that is non-negative over $P\cap\Z^d$,
given as a list of monomials with rational coefficients~$c_{{\bm{\beta}}}$
encoded in binary and exponents~${\bm{\beta}}$ encoded in unary;
an index~$k$, encoded in unary.\smallskip

\noindent {\em Output:} A lower bound~$L_k$ and an upper bound~$U_k$ for the maximal
function value $f^*$ of $f$ over $P\cap\Z^d$.
The bounds $L_k$ form a nondecreasing, the bounds $U_k$ a nonincreasing
sequence of bounds that both reach~$f^*$ in a finite number of steps.

\begin{enumerate}[\rm1.]
\item  Compute a short rational function expression for
  the generating function $g(P;\mathbf z)=\sum_{{\bm{\alpha}}\in P\cap\Z^d} \mathbf
  z^{{\bm{\alpha}}}$.  Using residue techniques, compute $|P \cap
  \Z^d|=g(P;\mathbf 1)$ from $g(P;\mathbf z)$.

\item Compute the polynomial~$f^k$ from~$f$.

\item From the rational function $ g(P;\mathbf z)$
  compute the rational function representation of $g(P,f^k;\mathbf z)$ of
  $\sum_{{\bm{\alpha}}\in P \cap \Z^d} f^k({\bm{\alpha}}) \mathbf z^{\bm{\alpha}}$ by
  Theorem \ref{operators}. Using residue techniques, compute
\[
L_k:=\left\lceil{\sqrt[k]{g(P,f^k;\mathbf 1)/g(P;\ve1)}}\,\right\rceil\;\;\;\text{and}\;\;\;
U_k:=\left\lfloor{\sqrt[k]{g(P,f^k;\mathbf 1)}}\right\rfloor.
\]
\end{enumerate}
\end{algorithm}
Taking the discussion of the convergence of the bounds into consideration, one
obtains the following result.

\begin{theorem}[Fully polynomial-time approximation scheme]
  Let the dimension $d$ be fixed.  Let $P\subset\R^d$ be a rational convex polytope.
  Let $f$ be a polynomial with rational coefficients that is non-negative on
  $P\cap\Z^d$, given as a list of monomials with rational coefficients~$c_{{\bm{\beta}}}$
  encoded in binary and exponents~${\bm{\beta}}$ encoded in unary.
  \begin{enumerate}[\rm(i)]
\item Algorithm~\ref{Algorithm} computes the bounds $L_k$, $U_k$ in time polynomial in
  $k$, the input size of $P$ and $f$, and the total degree~$D$. The bounds
  satisfy the following inequality:
$$
U_k-L_k \leq f^* \cdot \left(\sqrt[k]{|P \cap \Z^d|}-1 \right).
$$
\item For $k=(1+1/\epsilon)\log({|P \cap \Z^d|})$ (a number bounded by a
  polynomial in the input size),
  $L_k$ is a $(1-\epsilon)$-approximation to the optimal value $f^*$ and it
  can be computed in time polynomial in the input size, the total
  degree~$D$, and $1/\epsilon$. Similarly, $U_k$ gives a
  $(1+\epsilon)$-approximation to $f^*$.

\item With the same complexity,  by iterated bisection of $P$, we can also find
  a feasible solution $\mathbf x_\epsilon\in P\cap\Z^d$ with
  \begin{math}
    \bigl|f(\mathbf x_\epsilon) - f^*\bigr| \leq \epsilon f^*.
  \end{math}
\end{enumerate}
\end{theorem}

The mixed-integer case can be handled by \emph{discretization} of the continuous
variables.  We illustrate on an example that one needs to be careful
to pick a sequence of discretizations that actually converges.
Consider the mixed-integer \emph{linear} optimization problem depicted in
Figure~\ref{fig:example-slice-not-fulldim},
whose feasible region consists of the point $(\frac12,1)$ and the segment
$\{\,(x,0): x\in[0,1]\,\}$.  The unique optimal solution
is $x=\frac12$, $z=1$.  Now consider
the sequence of grid approximations 
where $x\in \frac1m \Z_{\geq0}$.  For even~$m$, the unique optimal solution to the
grid approximation is $x=\frac12$, $z=1$.  However, for odd~$m$, the unique
optimal solution is $x=0$, $z=0$.  Thus the full sequence of the optimal
solutions to the grid approximations does not converge because it has two limit
points; see Figure~\ref{fig:example-slice-not-fulldim}.
\begin{figure}[t]
  \begin{minipage}[t]{.3\linewidth}
    \ifpdf
    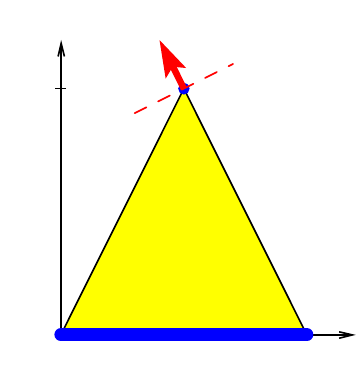
    \else
    \input{miptriangle.pstex_t}
    \fi
  \end{minipage}
  \quad
  \begin{minipage}[t]{.3\linewidth}
    \ifpdf
    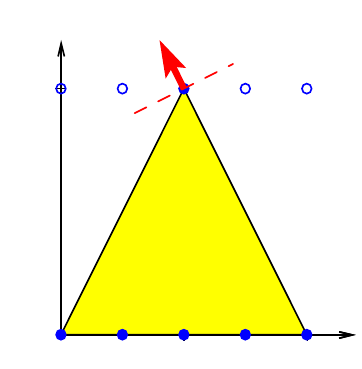
    \else
    \input{miptriangle-even.pstex_t}
    \fi
  \end{minipage}
  \quad
  \begin{minipage}[t]{.3\linewidth}
    \ifpdf
    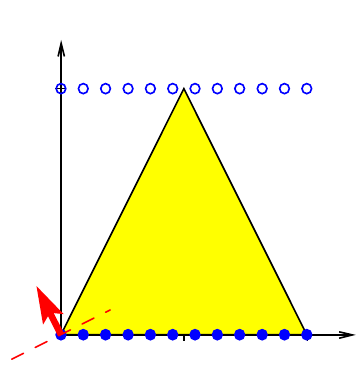
    \else
    \input{miptriangle-11.pstex_t}
    \fi
  \end{minipage}
  \caption{A mixed-integer linear optimization problem and a sequence of
    optimal solutions to grid problems with two limit points, for even $m$
    and for odd $m$}
  \label{fig:example-slice-not-fulldim}
\end{figure}

To handle polynomial objective functions that take arbitrary (positive and
negative) values on the feasible region, one can shift the objective function
by a constant that is large enough.  Then, to obtain a strong approximation
result, one iteratively reduces the constant by a factor.
Altogether we have the following result.

 \begin{theorem}[Fully polynomial-time approximation schemes]
   \label{th:mipo-fptas}
   Let the dimension $n=n_1+n_2$ be fixed.
   Let an optimization problem~\eqref{eq:nonlinear-over-polyhedron} of
   a polynomial function~$f$ over the mixed-integer points of a polytope~$P$
   and an error bound~$\epsilon$ be given, where
   \begin{enumerate}[\quad\rm({I}$_\bgroup 1\egroup$)]
   \item $f$ is given as a list of
     monomials with rational coefficients~$c_{{\bm{\beta}}}$ encoded in binary and
     exponents~${\bm{\beta}}$ encoded in unary,
   \item $P$ is given by rational inequalities in binary encoding,
   \item the rational number~$\frac1\epsilon$ is given in unary encoding.
   \end{enumerate}
   \begin{enumerate}[\rm(a)]
   \item There exists a fully polynomial time approximation scheme ({\small
       FPTAS}) for the \emph{maximization} problem for all
     polynomial functions $f(\mathbf x,\allowbreak \mathbf z)$
     that are \emph{non-negative} on the feasible region.
     That is, there exists a polynomial-time algorithm that,
     given the above data,
     computes a feasible solution $(\mathbf x_\epsilon, \mathbf
     z_\epsilon) \in P\cap\bigl(\R^{n_1}\times \Z^{n_2}\bigr)$ with
     \begin{displaymath}
       \bigl|f(\mathbf
       x_\epsilon, \mathbf z_\epsilon) - f(\mathbf x_{\max}, \mathbf z_{\max})\bigr|
       \leq \epsilon  f(\mathbf x_{\max},\mathbf z_{\max}).
     \end{displaymath}
   \item
     There exists a polynomial-time algorithm that,
     given the above data,
     computes a feasible solution $(\mathbf x_\epsilon, \mathbf
     z_\epsilon) \in P\cap\bigl(\R^{n_1}\times \Z^{n_2}\bigr)$ with
     \begin{displaymath}
       \bigl|f(\mathbf
       x_\epsilon, \mathbf z_\epsilon) - f(\mathbf x_{\max}, \mathbf z_{\max})\bigr|
       \leq \epsilon \bigl| f(\mathbf x_{\max},\mathbf z_{\max}) - f(\mathbf x_{\min}, \mathbf z_{\min}) \bigr|.
     \end{displaymath}
   \end{enumerate}
 \end{theorem}

\subsection{Semi-algebraic sets and SOS programming}
\label{s:sos-programming}

In this section we use results from algebraic geometry over the reals
to provide a convergent (and in the case of binary optimization,
finite) sequence of \index{semi-definite!relaxation}semi-definite
relaxations for the general polynomial optimization
problem\index{optimization problem!polynomial} over semi-algebraic
sets.
\begin{equation}
\label{eq:nonlinear}
\begin{array}{rl}
Z^*={\rm minimize} & f(\mathbf x) \vspace{3pt} \\
\mbox{ s.t. } & g_i(\mathbf x) \geq \mathbf0 ,~~~i=1,\ldots ,m,  \vspace{3pt} \\
& \mathbf x\in \R^n,
\end{array}
\end{equation}
where $f,g_i\in \R[\mathbf x]$ are polynomials defined as:
$$
f(\mathbf x) = \sum_{{{\bm{\alpha}}}\in\Z^n_+} f_{{{\bm{\alpha}}}} \mathbf x^{{{\bm{\alpha}}}},\qquad
g_i (\mathbf x) = \sum_{{{\bm{\alpha}}}\in\Z^n_+} g_{i,{{\bm{\alpha}}}} \mathbf x^{{{\bm{\alpha}}}},
$$
where there are only finitely many nonzero coefficients
$f_{{{\bm{\alpha}}}}$ and $g_{i,{{\bm{\alpha}}}}$. Moreover, let
$K=\{\mathbf x\in \R^n \mid g_i(\mathbf x) \geq 0 ,~i=1,\ldots ,m\}$ denote the set of
feasible solutions. Note that problem (\ref{eq:nonlinear}) can model
binary optimization, by taking $f(\mathbf x)=\mathbf c^\top \mathbf x,$
and taking as the polynomials $g_i(\mathbf x)$,
$\mathbf a_i^\top \mathbf x- b_i$, $x_j^2-x_j$ and $-x_j^2+x_j$ (to model
$x_j^2-x_j=0$). Problem (\ref{eq:nonlinear}) can also model bounded
integer optimization (using for example the equation
$(x_j-l_j)(x_j-l_j+1)\cdot\ldots\cdot(x_j-u_j)=0$ to model $l_j\leq
x_j\leq u_j$), as well as bounded mixed-integer nonlinear
optimization. Problem (\ref{eq:nonlinear}) can be written as:
\begin{equation}
\label{eq:sos2}
\begin{array}{rl}
{\rm maximize} & \gamma \vspace{3pt}\\
{\rm  s.t. } & f(\mathbf x)-\gamma \geq 0,~~\forall \; \mathbf x\in K.
\end{array}
\end{equation}
This leads us to consider conditions for polynomials
to be nonnegative over a set $K$.

\begin{definition}
  Let $p\in\R[\mathbf x]$ where $\mathbf x=(x_1,  \ldots,x_n)^\top $.
  The polynomial $p$ is called \emph{sos} (sum of squares), if there
  exist polynomials $h_1,\ldots,h_k \in \R[\mathbf x]$ such that $p =
  \sum_{i=1}^k h_i^2$.
\end{definition}

Clearly, in multiple dimensions if a polynomial can be written as a
sum of squares of other polynomials, then it is nonnegative. However,
is it possible for a polynomial in higher dimensions to be
nonnegative without being a sum of squares? The answer is yes. The
most well-known example is probably the Motzkin-polynomial
$M(x,y,z)=x^4y^2+x^2y^4+z^6-3x^2y^2z^2$, which is nonnegative without
being a sum of squares of polynomials.

The following theorem establishes a certificate of positivity of a
polynomial on the set $K$, under a certain assumption on $K$.

\begin{theorem}[\cite{Putinar93} \cite{JacobiPrestel01}]
\label{thm:putinar}
Suppose that the set $K$ is compact and there exists a polynomial
$h(\mathbf x)$ of the form
$$h(\mathbf x)=h_0(\mathbf x)+\sum_{i=1}^m h_i(\mathbf x) g_i(\mathbf x),$$
such that $\{\mathbf x\in \R^n\mid h(\mathbf x)\geq 0\}$ is compact and $h_i(\mathbf x)$,
$i=0,1,\ldots,m$, are polynomials that  have a sum of squares
representation. Then, if the polynomial $g$ is strictly positive over $K$, then there exist
$p_i \in \R[\mathbf x]$, $i=0,1,\ldots,m$,  that are sums of squares such that
\begin{equation}
\label{eq:represe} g(\mathbf x)=p_0(\mathbf x)+\sum_{i=1}^m p_i(\mathbf x) g_i(\mathbf x).
\end{equation}
\end{theorem}

Note that the number of terms in Equation (\ref{eq:represe}) is linear.
While the assumption of Theorem \ref{thm:putinar} may seem
restrictive, it is satisfied in several cases:
\begin{itemize}
\item[{\bf (a)}]\ For binary optimization problems\index{optimization
  problem!binary}, that is, when $K$ includes the inequalities
  $x^2_j\geq x_j$ and $x_j\geq x^2_j$ for all $j=1,\ldots,n$.
\item[{\bf (b)}]\ If all the $g_j$'s are linear, i.e., $K$ is a polyhedron.
\item[{\bf (c)}]\ If there is one polynomial $g_k$ such that the
  set $\{\mathbf x\in\R^n\mid g_k(\mathbf x)\geq 0\}$  is compact.
\end{itemize}

More generally, one way to ensure that the assumption of Theorem
\ref{thm:putinar} holds is to add to  $K$ the extra quadratic
constraint $g_{m+1}(\mathbf x)=a^2-\| \mathbf x\|^2\geq 0$ for some $a$ sufficiently
large. It is also important to emphasize that we do  not assume that
$K$ is convex. Notice that it may even be disconnected.

Let us now investigate algorithmically when a polynomial is a sum of
squares. As we will see this question is strongly connected to
semi-definite optimization. The idea of using
\index{semi-definite!optimization}semi-definite optimization for solving
optimization problems over polynomials is due to \cite{Shor87} and
further expanded in \cite{Lasserre01} and \cite{Parrilo03}.
We consider the vector
\[
\mathbf v_d(\mathbf x) = (\mathbf x^{{{\bm{\alpha}}}})_{|{{\bm{\alpha}}} |\leq d}=(1,x_1,\ldots,x_n,x_1^2,x_1x_2,\ldots , x_{n-1}x_n, x_n^2,\ldots, x_1^d,\ldots
,x_n^d)^\top,
\]
of all the monomials $\mathbf x^{{{\bm{\alpha}}}}$ of degree less than or equal to $d$,
which has dimension $s=\sum_{i=0}^d {n \choose i}={n+d \choose d}.$
\begin{proposition} [\cite{ChoiLamReznick:1995}]
\label{prop:sdprepresentation} \index{sum of squares} The  polynomial
$g(\mathbf x)$ of degree $2d$  has a \index{sum of squares!decomposition}sum
of squares decomposition if and only if there exists a positive
semi-definite matrix $Q$ for which $g(\mathbf x)=\mathbf v_d(\mathbf x)^\top Q \mathbf v_d(\mathbf x)$.
\end{proposition}

\proof
Suppose there exists an $s\times s $ matrix $Q\succeq 0$ for which
$g(\mathbf x)=\mathbf v_d(\mathbf x)^\top Q \mathbf v_d(\mathbf x)$. Then $Q=HH^\top$ for some
$s\times k$ matrix $H$, and thus, $$g(\mathbf x)=\mathbf v_d(\mathbf x)^\top HH^\top
\mathbf v_d(\mathbf x)=\sum_{i=1}^k (H^\top \mathbf v_d(\mathbf x))_i^2.$$
Because $(H^\top \mathbf v_d(\mathbf x))_i$ is a polynomial, then  $g(\mathbf x)$ is expressed as  a sum of
squares of the polynomials $(H^\top \mathbf v_d(\mathbf x))_i$.

Conversely, suppose that $g(\mathbf x)$ has a sum of squares decomposition
$g(\mathbf x)=\sum_{i=1}^{\ell} h_i(\mathbf x)^2.$ Let $\mathbf h_i$ be the vector
of coefficients of the polynomial $h_i(\mathbf x)$, i.e.,
$h_i(\mathbf x)=\mathbf h_i^\top \mathbf v_d(\mathbf x).$ Thus,
$$g(\mathbf x)=\sum_{i=1}^{\ell} \mathbf v_d(\mathbf x)^\top \mathbf h_i  \mathbf h_i^\top \mathbf v_d(\mathbf x)=
\mathbf v_d(\mathbf x)^\top Q \mathbf v_d(\mathbf x),$$
with $Q=\sum_{i=1}^{\ell}  \mathbf h_i  \mathbf h_i^\top \succeq 0$, and the
proposition follows. \qed

Proposition \ref{prop:sdprepresentation} gives rise to an algorithm.
Given a polynomial $f(\mathbf x) \in \R[x_1,\ldots,x_n]$ of degree $2d$.
In order to compute the minimum value $f^*= \min\{f(\mathbf x) \mid \mathbf x \in
\R^n\}$ we introduce an artificial variable $\lambda$ and determine
$$\max\{\lambda \mid \lambda \in \R ,\; f(\mathbf x) - \lambda \geq 0\}.$$
With the developments above, we realize that we can
determine a lower bound for $f^*$ by computing the value
$$p^{sos}\;=\; \max\{\lambda \mid \lambda \in \R ,\; f(\mathbf x) - \lambda
\mbox{ is sos}\}\; \leq\; f^*.$$
The latter task can be accomplished by setting up a semi-definite
program. In fact, if we denote by $f_{{\bm{\alpha}}}$ the coefficient of the
monomial $\mathbf x^{{\bm{\alpha}}}$ in the polynomial $f$, then $f(\mathbf x) - \lambda
\mbox{ is sos}$ if  and only if there exists an $s\times s $ matrix
$Q\succeq 0$ for which $f(\mathbf x)-\lambda=\mathbf v_d(\mathbf x)^\top Q \mathbf v_d(\mathbf x)$. Now
we can compare the coefficients on both sides of the latter
equation. This leads to the SOS-program
$$\begin{array}{llllllll}
&p^{sos}& = & \max        &  \lambda\\
&       &   & \mbox{ st. }&  f_{\mathbf 0} - \lambda & = &  Q_{{\mathbf 0},{\mathbf 0}}\\
&       &   &             & \sum_{{\bm{\beta}},{\bm{\gamma}},\;{\bm{\beta}}+{\bm{\gamma}} ={\bm{\alpha}}} Q_{{\bm{\beta}},{\bm{\gamma}}}& = &  f_{{\bm{\alpha}}}\\
&       &   &             & Q= ( Q_{{\bm{\beta}},{\bm{\gamma}}})_{{\bm{\beta}},{\bm{\gamma}}}& \succeq & 0.
\end{array}$$

In a similar vein, Theorem \ref{thm:putinar} and Proposition
\ref{prop:sdprepresentation} jointly imply that we can use
\index{semi-definite!optimization}semi-definite optimization to
provide a sequence of semi-definite relaxations for the optimization
problem (\ref{eq:sos2}). Assuming that the set $K$ satisfies the
assumption of Theorem \ref{thm:putinar}, then if $f(\mathbf x)-\gamma > 0$
for all $ \mathbf x\in K$, then
\begin{equation}
\label{eq:sos3}
f(\mathbf x)-\gamma=p_0(\mathbf x)+\sum_{i=1}^m p_i(\mathbf x) g_i(\mathbf x),
\end{equation}
where $p_i(\mathbf x)$, $i=0,1,\ldots ,m$ have a sum of squares representation.
\index{semi-definite!relaxation}Theorem \ref{thm:putinar} does not
specify the degree of the polynomials $p_i(\mathbf x)$. Thus, we select a
bound $2d$ on the degree of the polynomials $p_i(\mathbf x)$, and we apply
Proposition \ref{prop:sdprepresentation} to each of the polynomials
$p_i(\mathbf x)$, that is, $p_i(\mathbf x)$ is a sum of squares if and only if
$p_i(\mathbf x)=\mathbf v_d(\mathbf x)^\top Q_i \mathbf v_d(\mathbf x)$ with $Q_i \succeq 0$,
$i=0,1,\ldots ,m$. Substituting to Eq. (\ref{eq:sos3}), we obtain that
$\gamma, Q_i$, $i=0,1,\ldots ,m$, satisfy linear equations
that we denote as $L(\gamma, Q_0,Q_1,\ldots ,Q_m)=0.$
Thus, we can find a lower bound to problem (\ref{eq:nonlinear})
by solving the \index{semi-definite!optimization}semi-definite
optimization problem
\begin{equation}
\label{eq:sos4}
\begin{array}{rl}
Z_d={\rm max} & \gamma \vspace{3pt}\\
\mbox{ st. } & L(\gamma, Q_0,Q_1,\ldots ,Q_m)=0, \vspace{3pt}\\
& Q_i \succeq 0,~i=0,1,\ldots ,m.
\end{array}
\end{equation}
Problem (\ref{eq:sos4}) involves semi-definite optimization
over $m+1$ $s\times s$ matrices. From the construction
we get the relation $Z_d\leq Z^*$. It turns out that
 as  $d$ increases, $Z_d$ converges to $Z^*$. Moreover,
for binary optimization, there exists a finite $d$ for which
$Z_d=Z^*$ \cite{Laurent01}.

Problem (\ref{eq:sos4}) provides a systematic way to find convergent
semi-definite relaxations to problem (\ref{eq:nonlinear}).
While the approach  is both general (it applies to very
general nonconvex problems including nonlinear mixed-integer
optimization problems) and insightful   from a theoretical
point of view, it is only  practical for  values of $d=1,2$, as
large scale semi-definite optimization problems cannot be
solved in practice.
 In many situations, however,
$Z_1$ or $Z_2$ provide strong bounds. Let us consider an example.

\begin{example}\label{exam:sos5}
Let us minimize
$f(x_1,x_2)=2x_1^4+2x_1^3x_2-x_1^2x_2^2+5x_2^4$ over $\R^2$. We
attempt to write
$$\begin{array}{rcl}
f(x_1,x_2)& = & 2x_1^4+2x_1^3x_2-x_1^2x_2^2+5x_2^4 \vspace{5pt}\\
& = & \left(\begin{array}{c}
x_1^2 \vspace{3pt}\\ x_2^2 \vspace{3pt}\\x_1x_2 \end{array} \right)^\top
\left[\begin{array}{lll}
q_{11} & q_{12} & q_{13} \vspace{3pt}\\
q_{12} & q_{22} & q_{23} \vspace{3pt}\\
q_{13} & q_{23} & q_{33}
\end{array}
\right]
\left(\begin{array}{c}
x_1^2 \vspace{3pt}\\ x_2^2 \vspace{3pt}\\x_1x_2 \end{array} \right)
\vspace{5pt}\\
& = & q_{11}x_1^4+q_{22}x_2^4+
(q_{13}+2q_{12})x_1^2x_2^2+2q_{13}x_1^3x_2+2q_{23}x_1x_2^3.
\end{array}$$
In order to have an identity, we obtain
$$q_{11}=2,~q_{22}=5,~q_{33}+2q_{12}=-1,~ 2q_{13}=2, ~q_{23}=0.$$
Using semi-definite optimization, we find a particular solution
such that $Q\succeq 0$ is given by
$$Q=\left[\begin{array}{rrr}
2 & -3 & 1 \vspace{3pt}\\
-3 & 5 & 0 \vspace{3pt}\\
1 & 0 & 5
\end{array}
\right]
=HH^\top,\qquad H=\frac{1}{\sqrt{2}}
\left[\begin{array}{rr}
2  & 0  \vspace{3pt}\\
-3 & 1  \vspace{3pt}\\
1  & 3
\end{array}
\right]. $$
It follows that  $f(x_1,x_2)=\frac{1}{2}(2x_1^2-3x_2^2+x_1x_2)^2+
\frac{1}{2}(x_2^2+3x_1x_2)^2,$
and thus  the optimal solution value is $\gamma^*=0$
and the optimal solution is $x_1^*=x_2^*=0$.
\end{example}

\subsection{Quadratic functions}
\label{s:quadratics}


In this section, we focus on instances of polynomial programming where the
functions are all quadratic.
The specific form of the
\emph{mixed-integer quadratically constrained programming} problem that we consider is
\begin{equation}
\begin{aligned}
  \hbox{min}\quad & q_0(\mathbf x)\\
  \hbox{s.t.}\quad & q(\mathbf x) \leq 0 \\
      & \mathbf l\le \mathbf x \le \mathbf u\\
  & x_i \in\R && \mbox{ for } i=1,\ldots,k\\
  &  x_i \in\Z && \mbox{ for } i=k+1,\ldots,n,\\
\end{aligned} \label{eq:MIQCP}\tag*{$(\mbox{MIQCP}[\mathbf l,\mathbf u])$}
\end{equation}
where $q_0\colon \R^n\to\R$  and $q\colon \R^n\to\R^m$
are quadratic,
 $\mathbf l, \mathbf u\in \Z^n$, and $\mathbf l\le \mathbf u$.
 We denote the continuous relaxation by  $(\mbox{MIQCP}_\R[\mathbf l,\mathbf u])$.
 We emphasize that we are not generally making any convexity/concavity assumptions on the
quadratic functions $q_i$, so when we do require any such assumptions
we will state so explicitly.

Of course one can write a binary constraint $y_i\in\{0,1\}$ as
the (nonconvex) quadratic inequality $y_i(1-y_i)\le 0$
in the bound-constrained variable $0\le y_i \le 1$.
So, in this way, the case of binary variables $y_i$ can be seen
as the special case of $(\mbox{MIQCP}[\mathbf l,\mathbf u])$ with
no discrete variables (i.e., $k=n$). So, in a sense, the topic of
\emph{mixed-binary quadratically constrained programming}
can be seen as a special case of \emph{(purely continuous) quadratically constrained programming.}
We are not saying that it is necessarily useful to do this from a computational viewpoint,
but it makes it clear that the scope of even the purely continuous
quadratic model includes quadratic models having both binary and continuous variables,
and in particular mixed-$\{0,1\}$ linear programming.

In addition to the natural mathematical interest in studying
mixed-integer quadratically constrained programming, there is a wealth of
applications that have motivated the development of practical approaches; for example:
Trimloss problems (see \cite{insitu}, for example),
portfolio optimization (see \cite{BonLej}, for example),
Max-Cut and other binary quadratic models (see \cite{RRW,rendl-rinaldi-wiegele-2008}
and the references therein).

In the remainder of this section, we describe some
recent work on practical
computational approaches to nonconvex quadratic optimization models.
Rather than attempt a detailed survey, our goal is
to present a few recent and promising techniques.
One could regard these techniques as belonging more to the
field of  global optimization, but in
Section \ref{s:global} we present
material on global optimization aimed at more
general unstructured nonlinear integer programming problems.

\subsubsection{Disjunctive programming}
It is not surprising that integer variables
in a mixed-integer \emph{nonlinear} program can be treated with
disjunctive programming \cite{Balas74J,Balas98J}.  A corresponding branch-and-cut method was first
described in \cite{stubbs-mehrotra:1999} in the context of 0/1 mixed convex
programming.

Here we describe an intriguing result from \cite{SBL1,SBL2,SBL3},
which shows that one can also make useful disjunctions from nonconvex quadratic
functions in a mixed-integer quadratically-constrained programming problem or even
in a purely continuous quadratically-constrained programming problem.
The starting point for this approach is that we can take a quadratic form $\mathbf x^\top A_i \mathbf x$
in $\mathbf x \in \R^n$,
and rewrite it via an extended formulation as the linear form $\langle A_i,\mathbf X\rangle$,
using the matrix variable $\mathbf X\in\R^{n\times n}$,
and the nonlinear equation $\mathbf X=\mathbf x \mathbf x^\top$. The standard approach
is to relax $\mathbf X=\mathbf x \mathbf x^\top$ to the convex inequality $\mathbf X \succeq \mathbf x \mathbf x^\top$.
But the approach of \cite{SBL1,SBL2,SBL3} involves working with the
nonconvex inequality $\mathbf X \preceq \mathbf x \mathbf x^\top$. This basic idea is as follows.
Let $\mathbf v\in\R^n$ be arbitrary (for now).
We have the equation
\[
\langle \mathbf v \mathbf v^\top , \mathbf X \rangle =
\langle \mathbf v \mathbf v^\top , \mathbf x \mathbf x^\top \rangle
= (\mathbf v^\top \mathbf x)^2 ~,
\]
which we relax as the \emph{concave}
inequality
\[
\tag*{$(\Omega)$}(\mathbf v^\top \mathbf x)^2 \ge \langle \mathbf v \mathbf v^\top , \mathbf X \rangle.
\]
If we have a point $(\hat{\mathbf x},\hat{\mathbf X})$ that
satisfies the convex inequality $\mathbf X \succeq \mathbf x \mathbf x^\top$, but for which
$\hat{\mathbf X}\not=\hat{\mathbf x}\hat{\mathbf x}^\top$~, then it is the case that
$\hat{\mathbf X}-\hat{\mathbf x}\hat{\mathbf x}^\top$
has a positive eigenvalue $\lambda$~. Let $\mathbf v$ denote a unit-length
eigenvector belonging to $\lambda$~. Then
\begin{eqnarray*}
\lambda & = & \lambda \|\mathbf v\|^2_2 \\
& = &
\langle \mathbf v \mathbf v^\top , \hat{\mathbf X}-\hat{\mathbf x}\hat{\mathbf x}^\top \rangle .
\end{eqnarray*}
So, $\lambda>0$ if and only if $(\mathbf v^\top \hat{\mathbf x})^2 <
\langle \mathbf v \mathbf v^\top ,
\hat{\mathbf X}
\rangle$~.
That is, every positive eigenvalue of $\hat{\mathbf X}-\hat{\mathbf x}\hat{\mathbf x}^\top$ yields an inequality of
the form
$(\Omega)$
that is violated by $(\hat{\mathbf x},\hat{\mathbf X})$~. Next, we make a disjunction on this
violated nonconvex inequality $(\Omega)$. First, we choose
a suitable polyhedral relaxation ${\mathcal P}$ of the feasible region, and
we let $[\eta_L,\eta_U]$ be the range of $\mathbf v^\top \mathbf x$ as $(\mathbf x,\mathbf X)$
varies over the relaxation  ${\mathcal P}$.
Next, we choose a value $\theta\in
(\eta_L,\eta_U)$ (e.g., the midpoint), and we get the polyhedral disjunction:
\[
\left\{ (\mathbf x,\mathbf X)\in {\mathcal P} ~:~ \begin{array}{l}
\eta_L(\mathbf v)\le \mathbf v^\top \mathbf x \le \theta \\
(\mathbf v^\top \mathbf x)(\eta_L(\mathbf v)+\theta) ~-~ \theta \eta_L(\mathbf v) \ge \langle \mathbf v \mathbf v^\top , \mathbf X \rangle
\end{array}
\right\}
\]

\[
\mbox{or}
\]

\[
\left\{ (\mathbf x,\mathbf X)\in {\mathcal P} ~:~ \begin{array}{l}
\theta \le \mathbf v^\top \mathbf x \le \eta_U(\mathbf v)\\
(\mathbf v^\top \mathbf x)(\eta_U(\mathbf v)+\theta) ~-~ \theta \eta_U(\mathbf v) \ge \langle \mathbf v \mathbf v^\top , \mathbf X \rangle.
\end{array}
\right\}~.
\]
Notice that the second part of the first (resp., second)
half of the disjunction
corresponds to a secant inequality over the interval
between the point $\theta$ and the lower (resp., upper) bound
for $\mathbf v^\top \mathbf x$. Finally, we use the linear-programming technology of
ordinary disjunctive programming to separate, via a linear inequality,
the point
$(\hat{\mathbf x},\hat{\mathbf X})$ from the convex closure of the two halves of the disjunction.
Details and extensions of this idea appear in  \cite{SBL1,SBL2,SBL3}.

\subsubsection{Branch and cut}

A branch-and-cut scheme for optimization of a nonconvex quadratic form over a box
was recently developed by Vandenbussche and Nemhauser \cite{1057787,1057788}.
They use a formulation of Balas
via linear programming
with complementarity conditions, based on the necessary optimality conditions of
continuous quadratic programming (see  \cite{BalasQuadPolar}). Specifically,
they consider the problem
\begin{equation}
\tag*{$(\mbox{BoxQP}[Q,\mathbf c])$}
\begin{array}{rl}
\min & \frac{1}{2} \mathbf x^\top Q \mathbf  x + \mathbf c^\top \mathbf x \\
\mbox{ s.t. } & \mathbf x\in [0,1]^n,
\end{array}
\end{equation}
where $Q$ is an $n\times n$ symmetric, non positive semi-definite matrix, and
$\mathbf c\in\R^n$. The KKT necessary optimality conditions for $(\mbox{BoxQP}[Q,\mathbf c])$ are
\begin{eqnarray}
&&\mathbf y-Q\mathbf x - \mathbf z = \mathbf c,\label{KKTBoxFirst}\\
&&\mathbf y^\top (\mathbf 1 - \mathbf x) = 0 \\
&&\mathbf z^\top \mathbf x =0 \\
&&\mathbf x\in [0,1]^n\\
&&\mathbf y,~ \mathbf z\in \R^n_+~.\label{KKTBoxLast}
\end{eqnarray}
Vandenbussche and Nemhauser, appealing to a result of Balas,
define ${\cal P}(Q,\mathbf c)$ as the polyhedron defined
as the \emph{convex hull} of
solutions to (\ref{KKTBoxFirst}-\ref{KKTBoxLast}),
and they
work with the reformulation of (BoxQP) as the linear program
\begin{equation}
\tag*{$(\mbox{Balas}[Q,\mathbf c])$}
\begin{array}{rl}
\min & \frac{1}{2} \mathbf c^\top \mathbf  x + \frac{1}{2} \mathbf 1^\top \mathbf  y\\
\mbox{ s.t. } & (\mathbf x,\mathbf y,\mathbf z)\in {\cal P}(Q,\mathbf c).
\end{array}
\end{equation}
The main tactic of Vandenbussche and Nemhauser is to develop
cutting planes for ${\cal P}(Q,c)$.

Burer and Vandenbussche
pursue a similar direction, but they allow general polyhedral constraints
and employ semi-definite relaxations \cite{BurerVandenbussche:2008}.

\subsubsection{Branch and bound}

Linderoth also looks  at quadratically-constrained
programs that are not convex \cite{Lin2005}.
He develops a novel method for repeatedly
partitioning the continuous feasible region
into the Cartesian product of triangles and rectangles.
What is particularly interesting is that to
do this effectively, Linderoth develops convex envelopes of
bilinear functions over rectangles and triangles (also
see Anstreicher and Burer's paper \cite{anstreicher.burer:07}),
and then he demonstrates that these envelopes involve
hyperbolic constraints which
can be reformulated as the second-order cone constraints.
It is interesting to compare this with the similar use
of second-order cone constraints for convex quadratics
(see Section \ref{PracSOCP}).

One can view the technique of Linderoth as being a
specialized  ``Spatial Branch-and-Bound Algorithm.''
In  Section \ref{s:global}
we will describe the Spatial Branch-and-Bound Algorithm
for global optimization
in its full generality.

\section{Global optimization}
\label{s:global}

In the present section we take up the subject of global optimization
of rather general nonlinear functions. This is an enormous subject, and
so we will point to just a couple of directions that we view as promising.
On the practical side, in Section \ref{s:sBB} we describe the
Spatial Branch-and-Bound Algorithm which is one of the most successful
computational approaches in this area. In Section \ref{s:univariate-opt},
from the viewpoint of complexity theory,
with a goal of trying to elucidate the boundary between tractable and
intractable, we describe some very recent work on global optimization
of a very general class of nonlinear functions over an independence
system.

\subsection{Spatial Branch-and-Bound}\label{s:sBB}
In this section we address methods for global optimization of rather
general mixed-integer nonlinear programs having non-convex relaxations. Again, to have any hope at all,
we assume that the variables are bounded. There is
a very large body of work on solution techniques in this space. We will not attempt to
make any kind of detailed survey. Rather we will describe
one very successful methodology, namely the
\emph{Spatial Branch-and-Bound Algorithm.} We refer to \cite{SP1999,TawS2002}
and the references therein.

The Spatial Branch-and-Bound Algorithm for mixed-integer nonlinear programming
 has many similarities to the ordinary branch-and-bound employed for
 the solution of mixed-integer linear programs, but there are many additional
 wrinkles. Moreover, the techniques can be integrated.
 In what follows, we will concentrate on how continuous nonlinearities
 are handled. We leave it to the reader to see how these techniques
 would be integrated with the associated techniques for
 mixed-integer linear programs.

 One main difference with the mixed-integer linear case is that all
 nonlinear functions in a  problem instance are symbolically and
 recursively decomposed via simple operators, until we arrive at
 simple functions. The simple operators
 should be in a limited library. For example:  sum, product, quotient, exponentiation, power,
 logarithm, sine, cosine, absolute value. Such a decomposition is
 usually represented via a collection of rooted directed acyclic graphs. At each root is
 a nonlinear function occurring in the problem formulation.
 Leaves are constants, affine functions and atomic variables.
 Each non-leaf node is thought of as an auxiliary variable
 and also as  representing a simple operator, and its children are
 the arguments of that operator.

 An inequality constraint in the problem
 formulation can be thought of as a bounding interval on a root.
 In addition, the objective function is associated with a root,
 and so lower and upper bounds on the optimal objective value
 can also be  thought of as a bounding interval on a root.
 Simple bounds on a variable in the problem
 formulation can be thought of as a bounding interval on a leaf.
 In this way, we have an extended-variable reformulation of the
 given problem.

 Bounds are propagated up and down each such rooted directed acyclic graph
 via interval arithmetic and a library of convex envelopes or at least
 linear convex relaxations of the graphs of simple nonlinear operators acting
 on one or two variables on simple domains (intervals for univariate
 operators and simple polygons for bivariate operators).
 So, in this way, we have a now tractable convex or even
 linear relaxation of the extended-variable reformulation, and
 this is used to get a lower bound on the minimum
 objective value.

 The deficiency in our relaxation is localized to
 the graphs of simple functions that are only approximated
 by convex sets. We can seek to improve bounds by branching
 on the interval for a variable and reconvexifying on the subproblems.
 For example, we may have a variable $v$ that is a convex function $f$ in a variable
 $w$ on an interval $[l,u]$. Then the convex envelope of the
 graph $G[l,u]:=\{(v,w) ~:~ v=f(w)\}$ is precisely
 \[
\tilde{G}[l,u] := \left \{(v,w) ~:~  f(w) \le v\le f(l) + \left( \frac{f(u)-f(l)}{u-l} \right)(w-l)  \right\}.
 \]
 We may find that at a solution of the relaxation, the values of
 the variables $(u,v)$, say $(\hat{v},\hat{w})\in\tilde G[l,u]$,
 are far from $G[l,u]$. In such a case, we can \emph{branch}
 by choosing a point $b\in[l,u]$ (perhaps at or near  $\hat{v}$),
 and forming two subproblems in which the bounding interval for
 $v$ is amended as  $[l,b]$ in one and $[b,u]$ in the other.
 The value in branching on such a continuous variable is that
 we now reconvexify on the subintervals, effectively replacing $\tilde{G}[l,u]$
 with the smaller set $\tilde{G}[l,b]\cup \tilde{G}[b,u]$. In particular,
 if we did choose $b=\hat{v}$, then $(\hat{v},\hat{w})\notin\tilde{G}[l,b]\cup \tilde{G}[l,b]$,
 and so the algorithm makes some progress.
 We note that a lot of work has gone into good branching strategies
 (see \cite{BelLeeLibMarWac2008} for example).

 Finally, a good Spatial Branch-and-Bound procedure should have an effective
 strategy for finding good feasible solutions, so as to improve the
 objective upper bound (for minimization problems). A good rudimentary
 strategy is to take the solution of a relaxation as a starting point
 for a continuous nonlinear-programming solver aimed at finding a locally-optimal
 solution of the continuous relaxation (of either the original or
 extended-variable formulation). Then if a feasible solution to this relaxation
 is obtained and if it happens to have integer values for the
 appropriate variables, then we have an opportunity to update the
 objective value upper bound. Alternatively, one can use a
 solver aimed mainly at mixed-integer nonlinear programs having
 \emph{convex} relaxation as a heuristic also from such a starting point.
 In fact, the Branch-and-Bound Algorithm in {\tt Bonmin} has options
 aimed at giving good solutions from such a starting point,
 even for non-convex problems.

The Spatial Branch-and-Bound Algorithm relies on
the rapid and tight convexification of simple functions
on simple domains. Therefore, considerable work has gone into
developing closed-form expressions for such envelopes.
This type of work has paralleled some
research in mixed-integer linear programming that has focused on determining
convex hulls for simple constraints.
Useful results include: univariate functions \cite{adjiman,smith,ceria-soares:1999},
univariate monomials of odd degree \cite{convenvbook,l_and_costas},
bilinear functions \cite{alkhayyal,mccormick}, trilinear functions \cite{meyer1},
so-called $(n-1)$-convex functions \cite{JMW}, and fractional terms \cite{tawarmalani3}.
Further relevant work includes algorithms exploiting
variable transformations and appropriate convex envelopes and relaxations.
For example, for the case of ``signomials'' (i.e., terms of
the form $a_0 x_1^{a_1} x_2^{a_2}\cdots x_n^{a_n}$, with
$a_i\in\R$), see \cite{PBW2008}.

We do not make any attempt to exhaustively review available software for global optimization.
Rather we just mention that state-of-the-art codes implementing
a Spatial Branch-and-Bound Algorithm include {\tt Baron} \cite{Sahin1996,TawS2002,TawS2004} and
the new open-source code {\tt Couenne} \cite{BelLeeLibMarWac2008}.

\subsection{Boundary cases of complexity}
\label{s:univariate-opt}

Now, we shift our attention back to the viewpoint of complexity theory.
Our goal is to sample a bit of the recent
work that is aimed at revealing the boundary between tractable and
intractable instances of nonlinear discrete optimization problems.
We describe some very recent work on global optimization
of a very general class of nonlinear functions over an independence
system (see \cite{Lee+Onn+Weismantel:08b}). Other work in this vein includes
\cite{Berstein+Lee+Maruri-Aguilar+Onn+Riccomagno+Weismantel+Wynn:08,Berstein+Lee+Onn+Weismantel}.

Specifically, we consider the problem of optimizing a nonlinear objective function
over a weighted independence system presented by a linear-optimization
oracle. While this problem is generally intractable, we are able to
provide a polynomial-time algorithm that determines an ``$r$-best'' solution
for nonlinear functions of the total weight of an independent set,
where $r$ is a constant that depends on certain Frobenius numbers
of the individual weights and is independent of the size of the ground
set.

An \emph{independence system} is a nonempty set  of vectors
$S\subseteq\{0,1\}^n$ with the property that $\mathbf x\in\{0,1\}^n$, $\mathbf x\leq
\mathbf y\in S$ implies $\mathbf x\in S$.
The general nonlinear optimization problem
over a multiply-weighted independence system is as follows.
Given an independence system $S\subseteq\{0,1\}^n$, weight
vectors $\mathbf w_1,\dots,\mathbf w_d\in\Z^n$, and a function $f:\Z^d\rightarrow\R$,
find $\mathbf x\in S$ minimizing the objective
$
f(\mathbf w_1^\top \mathbf x,\dots,\mathbf w_d^\top \mathbf x)
$.

The representation of the objective in the above composite form has several
advantages. First, for $d>1$, it can naturally be interpreted as
\emph{multi-criteria optimization}: the $d$ given weight vectors
$\mathbf w_1,\dots,\mathbf w_d$ represent $d$ different criteria, where the value of
$\mathbf x\in S$ under criterion $i$ is its $i$-th total weight
$\mathbf w_i^\top \mathbf x$ and the objective is to
minimize the ``balancing'' $f(\mathbf w_1^\top \mathbf x,\dots,\mathbf w_d^\top
\mathbf x)$ of the $d$ given criteria by the given function $f$. Second, it
allows us to classify nonlinear optimization problems into a hierarchy
of increasing generality and complexity: at the bottom lies standard
linear optimization, recovered with $d=1$ and $f$ the identity on
$\Z$; and at the top lies the problem of maximizing an arbitrary
function, which is typically intractable, arising with $d=n$ and
$\mathbf w_i=\mathbf 1_i$ the $i$-th standard unit vector in $\Z^n$ for all $i$.

The computational complexity of the problem depends on the number $d$
of weight vectors, on the weights $w_{i,j}$, on the type of function $f$
and its presentation, and on the type of independence system $S$ and
its presentation. For example, when $S$ is a \emph{matroid}, the
problem can be solved in polynomial time for any fixed $d$, any
$\{0,1,\dots,p\}$-valued weights $w_{i,j}$ with $p$ fixed, and any
function $f$ presented by a \emph{comparison oracle}, even when $S$ is
presented by a mere \emph{membership oracle}, see
\cite{Berstein+Lee+Maruri-Aguilar+Onn+Riccomagno+Weismantel+Wynn:08}.
Also, for example, when $S$ consists of the \emph{matchings} in a
given bipartite graph $G$, the problem can be solved in polynomial
time for any fixed $d$, any weights $w_{i,j}$ presented in unary, and
any \emph{concave} function $f$, see \cite{Berstein+Onn:08}; but on the other
hand, for \emph{convex} $f$, already with fixed $d=2$ and
$\{0,1\}$-valued weights $w_{i,j}$, the problem includes as a special
case the \emph{exact matching problem} whose complexity is
long open \cite{Mulmuley+Vazirani+Vazirani:87,Papadimitriou+Yanakakis:82}.

In view of the difficulty of the problem already for $d=2$, we take a
first step and concentrate on \emph{nonlinear optimization over a
(singly) weighted independence system}, that is, with $d=1$, single
weight vector $\mathbf w=(w_1,\dots,w_n)$, and univariate function
$f:\Z\rightarrow\R$. The function $f$ can be arbitrary and is
presented by a \emph{comparison oracle} that, queried on $\mathbf x,\mathbf y\in\Z$,
asserts whether or not $f(\mathbf x)\leq f(\mathbf y)$. The weights $w_j$ take on
values in a $p$-tuple $\mathbf a=(a_1,\dots,a_p)$ of positive integers.
Without loss of generality we assume that $\mathbf a=(a_1,\dots,a_p)$
is \emph{primitive}, by which we mean that the $a_i$ are distinct
positive integers whose greatest common divisor
$\gcd(\mathbf a):=\gcd(a_1,\dots,a_p)$ is $1$. The independence system $S$ is
presented by a \emph{linear-optimization oracle} that, queried on
vector $\mathbf c\in\Z^n$, returns an element $\mathbf x\in S$ that maximizes the
linear function $\mathbf c^\top \mathbf x=\sum_{j=1}^n c_jx_j$. It turns out that this
problem is already quite intriguing, and so we settle for an
approximative solution in the following sense, that is interesting
in its own right. For a nonnegative integer $r$, we say that $\mathbf x^*\in
S$ is an \emph{$r$-best solution} to the optimization problem over $S$
if there are at most $r$ better objective values attained by feasible
solutions. In particular, a $0$-best solution is optimal. Recall that
the \emph{Frobenius number} of a primitive $\mathbf a$ is the largest integer
$\mathrm{F}(\mathbf a)$ that is not expressible as a nonnegative integer combination
of the $a_i$. We prove the following theorem.

\begin{theorem}
\label{Theorem: Main theorem for univarite polynomial optimization}
For every primitive $p$-tuple $\mathbf a=(a_1,\dots,a_p)$, there is a constant $r(\mathbf a)$
and an algorithm that, given any independence system
$S\subseteq\{0,1\}^n$ presented by a linear-optimization oracle,
weight vector $\mathbf w\in\{a_1,\dots,a_p\}^n$, and function
$f:\Z\rightarrow\R$ presented by a comparison oracle, provides an
$r(\mathbf a)$-best solution to the nonlinear problem $\min\{f(\mathbf w^\top\mathbf x) ~:~ \mathbf x\in
S\}$, in time polynomial in $n$. Moreover:
\begin{enumerate}
\item
If $a_i$ divides $a_{i+1}$ for $i=1,\dots, p-1$, then the algorithm
provides an optimal solution.
\item
For $p=2$, that is, for $\mathbf a=(a_1,a_2)$, the algorithm provides an
$\mathrm{F}(\mathbf a)$-best solution.
\end{enumerate}
\end{theorem}

In fact, we give an explicit upper bound on $r(\mathbf a)$ in terms of the
Frobenius numbers of certain subtuples derived from $\mathbf a$. An interesting
special case is that of $\mathbf a=(2,3)$. Because $\mathrm{F}(2,3)=1$, the solution provided
by our algorithm in that case is either optimal or second best.

The proof of Theorem \ref{Theorem: Main theorem for univarite polynomial
optimization} is pretty technical, so we only outline the main
ideas. Below we present a na\"{\i}ve solution strategy that
does not directly lead to a good approximation. However, this
na\"{\i}ve approach is used as a basic building block. One partitions the
independence system into suitable pieces, to each of which a suitable
refinement of the na\"{\i}ve strategy is applied
separately. Considering the monoid generated by $\{a_1,\dots,a_p\}$
allows one to show that the refined na\"{\i}ve strategy
applied to each piece gives a good approximation within that piece. In
this way, the approximation quality $r(\mathbf a)$ can be bounded as follows,
establishing a proof to Theorem \ref{Theorem: Main theorem for
univarite polynomial optimization}.

\begin{lemma}\label{Bound-Corollary}
Let $\mathbf a=(a_1,\dots,a_p)$ be any primitive $p$-tuple. Then the following hold:
\begin{enumerate}
\item
 An upper bound on $r(\mathbf a)$ is given by $r(\mathbf a)\leq \left(2\max(\mathbf a)\right)^p$.
\item
For divisible $\mathbf a$, we have $r(\mathbf a)=0$.
\item
For $p=2$, that is, for $\mathbf a=(a_1,a_2)$, we have $r(\mathbf a)=\mathrm{F}(\mathbf a)$.
\end{enumerate}
\end{lemma}

Before we continue, let us fix some notation. The \emph{indicator} of
a subset $J\subseteq N$ is the vector $\mathbf 1_J:=\sum_{j\in J}\mathbf
1_j\in\{0,1\}^n$, so that $\mathop{\mathrm{supp}}(\mathbf 1_J)=J$. Unless otherwise
specified, $\mathbf x$ denotes an element of $\{0,1\}^n$ and
$\bm{\lambda}, \bm{\tau},\bm{\nu}$ denote elements of $\Z^p_+$. Throughout,
$\mathbf a=(a_1,\dots,a_p)$ is a \emph{primitive} $p$-tuple. We will be
working with weights taking values in $\mathbf a$, that is, vectors
$\mathbf w\in\{a_1,\dots,a_p\}^n$. With such a weight vector $w$ being clear
from the context, we let $N_i:=\{j\in N ~:~ w_j=a_i\}$ for
$i=1,\dots,p$, so that $N=\biguplus_{i=1}^p N_i$. For
$\mathbf x\in\{0,1\}^n$ we let $\lambda_i(\mathbf x):=\mathopen|\mathop{\mathrm{supp}}(\mathbf x)\cap N_i\mathclose|$ for
$i=1,\dots,p$, and $\bm{\lambda}(\mathbf x):=(\lambda_1(\mathbf x),\dots,\lambda_p(\mathbf x))$,
so that $\mathbf w^\top \mathbf x=\bm{\lambda}(\mathbf x)^\top \mathbf a$. For integers
$z,s\in\Z$ and a set of integers $Z\subseteq \Z$, we define
$z+sZ:=\{z+sx ~:~ x\in Z\}$.

Let us now present the na\"{\i}ve strategy to solve the univariate
nonlinear problem $\min\{f(\mathbf w^\top \mathbf x) ~:~ \mathbf x\in S\}$. Consider a set
$S\subseteq\{0,1\}^n$, weight vector  $\mathbf w\in\{a_1,\dots,a_p\}^n$, and
function $f:\Z\rightarrow\R$ presented by a comparison oracle.
Define the \emph{image} of $S$ under $\mathbf w$ to be the set of values $\mathbf w^\top \mathbf x$
taken by elements of $S$; we denote it by $\mathbf w\cdot S$.

We point out the following simple observation.

\begin{proposition}\label{image}
A necessary condition for any algorithm to find an $r$-best solution to
the problem $\min\{f(\mathbf w^\top \mathbf x) ~:~ x\in S\}$,
where the function~$f$ is presented by a comparison oracle only,
is that it computes all
but at most $r$ values of the image $\mathbf w\cdot S$ of $S$ under $\mathbf w$.
\end{proposition}
Note that this necessary condition is also sufficient for computing the
\emph{objective value} $f(\mathbf w^\top \mathbf x^*)$ of an $r$-best solution, but not
for computing an actual $r$-best solution $\mathbf x^*\in S$, which may be harder.
Any point ${\bar{\mathbf x}}$ attaining $\max\{\mathbf w^\top \mathbf x ~:~ \mathbf x\in S\}$ provides
an approximation of the image given by
\begin{equation}\label{approximation}
\{\mathbf w^\top \mathbf x ~:~ \mathbf x\leq {\bar {\mathbf x}}\}\subseteq \mathbf w\cdot
S\subseteq\{0,1,\dots,\mathbf w^\top {\bar {\mathbf x}}\}\ .
\end{equation}
This suggests the following natural na\"{\i}ve strategy for finding an
approximative solution to the optimization problem over an independence
system $S$ that is presented by a linear-optimization oracle.

\begin{algorithm}
\label{naive-strategy}

{\bf (Na\"{\i}ve Strategy)}

\noindent
{\bf input} Independence system $S\subseteq\{0,1\}^n$ presented by a
linear-optimization oracle,
$f:\Z\rightarrow\R$ presented by a comparison oracle,
and  $\mathbf w\in\{a_1,\ldots,a_p\}^n$\;

\noindent
{\bf obtain }${\bar {\mathbf x}}$ attaining $\max\{\mathbf w^\top \mathbf x ~:~ \mathbf x\in S\}$
using the linear-optimization oracle for $S$\;

\noindent
{\bf  output} $\mathbf x^*$ as one attaining $\min\{f(\mathbf w^\top \mathbf x) ~:~  \mathbf x\leq
{\bar { \mathbf x}}\}$ using the algorithm of Lemma \ref{strategy-lemma} below.

\end{algorithm}

Unfortunately, as the next example shows, the number of values of the
image that are missing from the approximating set on the left-hand
side of equation (\ref{approximation}) cannot generally be bounded by
any constant. So by Proposition \ref{image}, this strategy cannot be
used \emph{as is} to obtain a provably good approximation.

\begin{example}\label{example}
Let $ \mathbf a:=(1,2)$, $n:=4m$, $ \mathbf  y:=\sum_{i=1}^{2m}  \mathbf 1_i$,
$ \mathbf z:=\sum_{i=2m+1}^{4m}  \mathbf 1_i$, and $ \mathbf w:= \mathbf y+2 \mathbf z$, that is,
$$ \mathbf y \ =\ (1,\dots,1,0,\dots,0)\,,\quad  \mathbf z \ =\ (0,\dots,0,1,\dots,1)\,,\quad
 \mathbf w \ =\ (1,\dots,1,2,\dots,2)\,,$$
define $f$ on $\Z$ by
\[
f(k):=
\left\{
  \begin{array}{ll}
    k, & \hbox{$k$ odd;} \\
    2m, & \hbox{$k$ even,}
  \end{array}
\right.
\]
and let $S$ be the independence system
\[
S\ :=\ \{ \mathbf x\in\{0,1\}^n ~:~  \mathbf x\leq  \mathbf y\}\ \cup\ \{ \mathbf x\in\{0,1\}^n ~:~  \mathbf x\leq  \mathbf z\}.
\]
Then the unique optimal solution of the linear-objective
problem $\max\{ \mathbf w^\top  \mathbf x ~:~  \mathbf x\in S\}$ is ${\bar  {\mathbf x}}:= \mathbf z$,
with $ \mathbf w^\top {\bar  {\mathbf x}}=4m$, and therefore
\begin{eqnarray*}
&&\{ \mathbf w^\top  \mathbf x ~:~  \mathbf  x\leq {\bar  {\mathbf x}}\}\ =\ \{2i ~:~ i=0,1,\dots,2m\}
\mbox{ and}\\
&& \mathbf w\cdot S \ =\ \{i ~:~ i=0,1,\dots,2m\}\ \cup\ \{2i ~:~ i=0,1,\dots,2m\}.
\end{eqnarray*}
So all $m$ odd values (i.e., $1,3,\ldots,2m-1$)
in the image $ \mathbf w\cdot S$ are missing from the
approximating set $\{ \mathbf w^\top  \mathbf x ~:~  \mathbf x\leq {\bar  {\mathbf x}}\}$ on the left-hand side
of (\ref{approximation}), and $ \mathbf x^*$ attaining $\min\{f( \mathbf w^\top  \mathbf x)
~:~  \mathbf x\leq{\bar  {\mathbf x}}\}$ output by the above strategy has
objective value $f( \mathbf w^\top  \mathbf x^*)=2m$, while there are
$m={n\over 4}$ better objective values (i.e., $1,3,\ldots,2m-1$)
attainable by feasible points (e.g., $\sum_{i=1}^k  \mathbf  1_i$,
for $k=1,3,\ldots,2m-1$).
\end{example}

Nonetheless, a more sophisticated refinement of the na\"{\i}ve
strategy, applied repeatedly to several suitably chosen subsets of $S$
rather than $S$ itself, will lead to a good approximation.
Note that the na\"{\i}ve strategy can be efficiently implemented as follows.

\begin{lemma}\label{strategy-lemma}
For every fixed $p$-tuple $ \mathbf a$, there is a polynomial-time algorithm that,
given univariate function $f:\Z\rightarrow\R$ presented by a comparison oracle,
weight vector $ \mathbf w\in\{a_1,\dots,a_p\}^n$, and $\bar  {\mathbf x}\in\{0,1\}^n$, solves
$\min\{f( \mathbf w^\top  \mathbf x) ~:~  \mathbf x\leq {\bar  {\mathbf x}}\}$.
\end{lemma}
\begin{proof}
Consider the following algorithm:

\begin{algorithm}
\label{strategy-algorithm}

{\bf input} function $f:\Z\rightarrow\R$ presented by a comparison oracle,
$ \mathbf w\in\{a_1,\ldots,a_p\}^n$ and $\bar {\mathbf x}\in\{0,1\}^n$\;

{\bf let} $N_i:=\{j : w_j=a_i\}$ and
$\tau_i:=\lambda_i({\bar  {\mathbf x}})=\mathopen|\mathop{\mathrm{supp}}({\bar  {\mathbf x}})\cap N_i\mathclose|,\,$ $i=1,\dots,p$\;

For {\emph{every choice of} $ \bm{\nu}=(\nu_1,\dots,\nu_p)\leq(\tau_1,\dots,\tau_p)= \bm{\tau}$}{

{\bf  determine} some $ \mathbf x_{\bm{\nu}}\leq{\bar  {\mathbf x}}$
with $\lambda_i( \mathbf x_{\bm{\nu}})=\mathopen|\mathop{\mathrm{supp}}( \mathbf x_{\bm{\nu}})\cap N_i\mathclose|=\nu_i,\,$ $i=1,\dots,p$\;}


{\bf  output} $ \mathbf x^*$ as one minimizing $f( \mathbf w^\top  \mathbf x)$ among the
$ \mathbf x_{\bm{\nu}}$ by using the comparison oracle of $f$.


\end{algorithm}

As the value $ \mathbf w^\top  \mathbf x$ depends only on the cardinalities
$\mathopen|\mathop{\mathrm{supp}}( \mathbf x)\cap N_i\mathclose|,\,$ $i=1,\dots,p$, it is clear that
\[
\{ \mathbf w^\top  \mathbf x ~:~  \mathbf x\leq {\bar  {\mathbf x}}\} = \{ \mathbf w^\top  \mathbf x_{\bm{\nu}} ~:~  \bm{\nu}\leq \bm{\tau}\}.
\]
Clearly, for each choice $ {\bm{\nu}}\leq {\bm{\tau}}$ it is easy to determine some
$ \mathbf x_{\bm{\nu}}\leq{\bar  {\mathbf x}}$ by zeroing out suitable entries of $\bar{\mathbf x}$. The
number of choices ${\bm{\nu}}\leq {\bm{\tau}}$ and hence of loop iterations and
comparison-oracle queries of $f$ to determine $ \mathbf x^*$ is
\[
\prod_{i=1}^p (\tau_i+1)\ \leq\ (n+1)^p.
\]
\end{proof}

\section{Conclusions}
\label{s:conclusions}

\begin{table}[t]
  \caption{Computational complexity and algorithms for nonlinear integer
    optimization.}
  \label{t:overview-nonlinear-over-nonlinear}
  \begin{center}
    \def\arraystretch{1.4}
    \begin{tabular}{@{}p{.14\linewidth}p{.29\linewidth}p{.29\linewidth}p{.25\linewidth}@{}} \toprule
      & \multicolumn{3}{c}{
        Constraints}\\
      \cmidrule{2-4}
      \def\arraystretch{1.0}\smash{\begin{tabular}[b]{@{}l@{}}Objective\\function
        \end{tabular}}%
      & Linear & Convex Polynomial & Arbitrary Polynomial
      \\[-2.5ex]
      \midrule
      Linear
      & {\RaggedRight
         \textbf{Polynomial-time in fixed dimension:}
         \begin{enumerate}[--]
         \item Lenstra's algorithm \cite{Lenstra83}
         \item Generalized basis
           re\-duc\-tion, Lov\'asz--Scarf \cite{lovasz-scarf:92}
         \item Short rational generating
           functions, Barvinok \cite{Barvinok94}
         \end{enumerate}}
      & \begin{minipage}[t]{1\linewidth}
           \RaggedRight
           \textbf{Polynomial-time in fixed dimension:}

           Lenstra-type algorithms \psref{s:convex-min-fixed-dim}
           \begin{enumerate}[--]
           \item Khachiyan--Porkolab \cite{khachiyan-porkolab:00}
           \item Heinz \cite{heinz-2005:integer-quasiconvex}
           \end{enumerate}
           \par
         \end{minipage}
      &
      {\raggedright
         \textbf{Incomputable:} Hilbert's 10th problem, Matiyasevich
         \cite{matiyasevich-1970}
         \psref{s:overview}, even for:
         \begin{enumerate}[--]
         \item quadratic constraints,
           Jeroslow~\cite{jeroslow-1973:quadratic-ip-uncomputable}
         \item fixed dimension~10,
           Matiyasevich~\cite{jones-1982}
         \end{enumerate}
         \smallskip

       }
      \\[-3ex]
         \RaggedRight Convex max \psref{s:convex-max}
         & \begin{minipage}[t]{1\linewidth}
           \RaggedRight
           \textbf{Polynomial-time in fixed dimension:}
           Cook et al. \cite{cook-hartmann-kannan-mcdiarmid-1992}
           \psref{s:convex-max-fixed-dim}


         \end{minipage}
         & 
         & \textbf{Incomputable} \psref{s:overview}
         \\[7ex]
      \RaggedRight Convex min \psref{s:convex-min}
       & \begin{minipage}[t]{2\linewidth}
           \RaggedRight
           \textbf{Polynomial-time in fixed dimension:}
           Lenstra-type algorithms: Khachiyan--Porkolab \cite{khachiyan-porkolab:00},
           Heinz \cite{heinz-2005:integer-quasiconvex}
           \psref{s:convex-min-fixed-dim}
         \end{minipage}
       &
       & \textbf{Incomputable} \psref{s:overview}
       \\[4ex]
         \RaggedRight Arbitrary Polynomial \psref{s:general-polynomial}
       &
       \begin{minipage}[t]{2\linewidth}
         {\RaggedRight\textbf{NP-hard, inapproximable,}
           even for quadratic forms
           over hypercubes:
           {\textsc{max-cut}}, H\aa stad \cite{Hastad:inapprox97}
           \psref{s:overview}\smallskip\par
           \textbf{NP-hard,} even for fixed dimension~$2$, degree~$4$
           \psref{s:overview}\par}
         \end{minipage}\medskip

       \begin{minipage}[t]{1\linewidth}
         {\RaggedRight\textbf{{\small FPTAS} in
             fixed dimension:}
           Short rational generating functions, De Loera
           et~al.~\cite{deloera-hemmecke-koeppe-weismantel:intpoly-fixeddim}
           \psref{s:fptas}
           \par}

       \end{minipage}
       &
       & \textbf{Incomputable} \psref{s:overview}

       \\[17ex]


      \bottomrule
    \end{tabular}
  \end{center}
\end{table}

In this chapter, we hope to have succeeded in reviewing mixed-integer
nonlinear programming from two important viewpoints.

We have reviewed the computational complexity of several important
classes of mixed-integer nonlinear programs.  Some of the negative complexity
results (incomputability, NP-hardness) that appeared in
Section~\ref{s:overview} have been supplemented by polynomiality or
approximability results in fixed dimension.
Table~\ref{t:overview-nonlinear-over-nonlinear} gives a summary.  In addition
to that, and not shown in the table, we have explored the boundary between
tractable and intractable problems, by highlighting interesting cases in
varying dimension where still polynomiality results can be obtained.

Additionally, we have reviewed a selection of practical algorithms that seem to have
the greatest potential from today's point of view.  Many of these algorithms,
at their core, are aimed at \emph{integer convex minimization}.  Here we have
nonlinear branch-and-bound, outer approximation, the Quesada--Grossman
algorithm, hybrid algorithms, and generalized Benders decomposition. As we
have reported, such
approaches can be specialized and enhanced for problems with SDP constraints,
SOCP constraints, and (convex) quadratics.  For \emph{integer polynomial
  programming} (without convexity assumptions), the toolbox of
Positivstellens\"atze and SOS programming is available.  For the case of
\emph{quadratics} (without convexity assumptions), specialized versions of
disjunctive programming, branch-and-cut, and branch-and-bound have been devised.
Finally, for general \emph{global optimization}, spatial branch-and-bound is
available as a technique, which relies heavily on
convexification methods. 

It is our hope that, by presenting these two viewpoints to the
interested reader, this chapter will help to create a synergy between both
viewpoints in the near future.  May this lead to
a better understanding of the field, and to much better algorithms than what
we have today!

\section*{Acknowledgments}

We would like to thank our coauthors, Jes\'us De Loera
and Shmuel
Onn, for their permission to base the presentation of some of the material in
this chapter on our joint papers
\cite{DeLoera+Hemmecke+Onn+Weismantel:08,Lee+Onn+Weismantel:08,Lee+Onn+Weismantel:08b}.


\begin{thebibliography}{100}
\providecommand{\url}[1]{{#1}}
\providecommand{\urlprefix}{URL }
\expandafter\ifx\csname urlstyle\endcsname\relax
  \providecommand{\doi}[1]{DOI~\discretionary{}{}{}#1}\else
  \providecommand{\doi}{DOI~\discretionary{}{}{}\begingroup
  \urlstyle{rm}\Url}\fi

\bibitem{AbhLL2006}
Abhishek, K., Leyffer, S., Linderoth, J.: Filmint: An outer-approximation-based
  solver for nonlinear mixed integer programs.
\newblock Preprint ANL/MCS-P1374-0906 (2006)

\bibitem{adjiman}
Adjiman, C.: Global optimization techniques for process systems engineering.
\newblock Ph.D. thesis, Princeton University (1998)

\bibitem{Ahuja+Magnanti+Orlin}
Ahuja, R., Magnanti, T., Orlin, J.: Network flows: theory, algorithms, and
  applications.
\newblock Prentice-Hall, Inc., New Jersey (1993)

\bibitem{AAG2009}
Akt\"rk, S., Atamt\"urk, A., G\"urel, S.: A strong conic quadratic
  reformulation for machine-job assignment with controllable processing times.
\newblock Oper. Res. Lett. \textbf{37}(3), 187--191 (2009)

\bibitem{alkhayyal}
Al-Khayyal, F., Falk, J.: Jointly constrained biconvex programming.
\newblock Mathematics of Operations Research \textbf{8}(2), 273--286 (1983)

\bibitem{anstreicher.burer:07}
Anstreicher, K., Burer, S.: Computable representations for convex hulls of
  low-dimensional quadratic forms.
\newblock Tech. rep., Dept. of Management Sciences, University of Iowa (2007)

\bibitem{Balas74J}
Balas, E.: Disjunctive programming: Properties of the convex hull of feasible
  points.
\newblock MSRR No. 330, Carnegie Mellon University (Pittsburgh PA, 1974)

\bibitem{BalasQuadPolar}
Balas, E.: Nonconvex quadratic programming via generalized polars.
\newblock SIAM J. Appl. Math. \textbf{28}, 335--349 (1975)

\bibitem{Balas98J}
Balas{}, E.: Disjunctive programming: Properties of the convex hull of feasible
  points.
\newblock Discrete Applied Mathematics \textbf{89}, 3--44 (1998)

\bibitem{bank-heintz-krick-mandel-solerno-1993}
Bank, B., Heintz, J., Krick, T., Mandel, R., Solern{\'o}, P.: Une borne
  optimale pour la programmation enti{\'e}re quasi-convexe.
\newblock Bull. Soc. math. France \textbf{121}, 299--314 (1993)

\bibitem{bank-krick-mandel-dolerno-1991}
Bank, B., Krick, T., Mandel, R., Solern{\'o}, P.: A geometrical bound for
  integer programming with polynomial constraints.
\newblock In: Fundamentals of Computation Theory, \emph{Lecture Notes In
  Computer Science}, vol. 529, pp. 121--125. Springer-Verlag (1991)

\bibitem{Barvinok94}
Barvinok, A.I.: A polynomial time algorithm for counting integral points in
  polyhedra when the dimension is fixed.
\newblock Mathematics of Operations Research \textbf{19}, 769--779 (1994)

\bibitem{BarviPom}
Barvinok, A.I., Pommersheim, J.E.: An algorithmic theory of lattice points in
  polyhedra.
\newblock In: L.J. Billera, A.~Bj\"orner, C.~Greene, R.E. Simion, R.P. Stanley
  (eds.) New Perspectives in Algebraic Combinatorics, \emph{Math. Sci. Res.
  Inst. Publ.}, vol.~38, pp. 91--147. Cambridge Univ. Press, Cambridge (1999)

\bibitem{BelLeeLibMarWac2008}
Belotti, P., Lee, J., Liberti, L., Margot, F., W\"achter, A.: Branching and
  bounds tightening techniques for non-convex {MINLP} (2008).
\newblock IBM Research Report RC24620

\bibitem{Ben-Tal-Nemirovski:2001}
Ben-Tal, A., Nemirovski, A.: Lectures on Modern Convex Optimization: Analysis,
  Algorithms, and Engineering Applications.
\newblock MPS-SIAM Series on Optimization. SIAM, Philadelphia, USA (2001)

\bibitem{Berstein+Lee+Maruri-Aguilar+Onn+Riccomagno+Weismantel+Wynn:08}
Berstein, Y., Lee, J., Maruri-Aguilar, H., Onn, S., Riccomagno, E., Weismantel,
  R., Wynn, H.: Nonlinear matroid optimization and experimental design.
\newblock SIAM Journal on Discrete Mathematics \textbf{22}(3), 901--919 (2008)

\bibitem{Berstein+Lee+Onn+Weismantel}
Berstein, Y., Lee, J., Onn, S., Weismantel, R.: Nonlinear optimization for
  matroid intersection and extensions.
\newblock IBM Research Report RC24610  (2008)

\bibitem{Berstein+Onn:08}
Berstein, Y., Onn, S.: Nonlinear bipartite matching.
\newblock Discrete Optimization \textbf{5}, 53--65 (2008)

\bibitem{BertsimasWeismantel05}
Bertsimas, D., Weismantel, R.: Optimization over Integers.
\newblock Dynamic Ideas, Belmont, Ma (2005)

\bibitem{BiqMac}
{Biq-Mac Solver - Binary quadratic and Max cut Solver}: {\tt
  biqmac.uni-klu.ac.at} (2006)

\bibitem{bonmin}
Bonami, P., Biegler, L., Conn, A., Cornu{\'e}jols, G., Grossmann, I., Laird,
  C., Lee, J., Lodi, A., Margot, F., Sawaya, N., W{\"a}chter, A.: An
  algorithmic framework for convex mixed integer nonlinear programs.
\newblock Discrete Optimization \textbf{5}, 186--204 (2008)

\bibitem{efpump}
Bonami, P., Cornu{\'e}jols, G., Lodi, A., Margot, F.: A feasibility pump for
  mixed integer nonlinear programs (2008)

\bibitem{bonminOptima}
Bonami, P., Forrest, J., Lee, J., W\"achter, A.: Rapid development of an
  {MINLP} solver with {COIN-OR}.
\newblock Optima \textbf{75}, 1--5 (2007)

\bibitem{bonminHeur}
Bonami, P., Gon\c{c}alves, J.P.M.: Primal heuristics for mixed integer
  nonlinear programs (2008).
\newblock IBM Research Report RC24639

\bibitem{BMUM}
Bonami, P., Lee, J.: Bonmin users' manual.
\newblock Tech. rep. (June 2006)

\bibitem{BonLej}
Bonami, P., Lejeune, M.: An exact solution approach for integer constrained
  portfolio optimization problems under stochastic constraints.
\newblock {\rm To appear in} Operations Research

\bibitem{bonminNEOS}
Bonmin: {\tt neos.mcs.anl.gov/neos/solvers/minco:Bonmin/AMPL.html}

\bibitem{bonmincode}
Bonmin: {\tt projects.coin-or.org/Bonmin} (v. 0.99)

\bibitem{boros-hammer:2002}
Boros, E., Hammer, P.L.: Pseudo-{B}oolean optimization.
\newblock Discrete Applied Mathematics \textbf{123}, 155--225 (2002)

\bibitem{Brion88}
Brion, M.: Points entiers dans les poly{\'e}dres convexes.
\newblock Ann. Sci. {\'E}cole Norm. Sup. \textbf{21}(4), 653--663 (1988)

\bibitem{Brion1997residue}
Brion, M., Vergne, M.: Residue formulae, vector partition functions and lattice
  points in rational polytopes.
\newblock J. Amer. Math. Soc. \textbf{10}, 797--833 (1997)

\bibitem{buchheim-rinaldi:2007}
Buchheim, C., Rinaldi, G.: Efficient reduction of polynomial zero-one
  optimization to the quadratic case.
\newblock SIAM Journal on Optimization \textbf{18}, 1398--1413 (2007)

\bibitem{BurerVandenbussche:2008}
Burer, S., Vandenbussche, D.: A finite branch-and-bound algorithm for nonconvex
  quadratic programming via semidefinite relaxations.
\newblock Mathematical Programming \textbf{Ser. A, 113}, 259--282 (2008)

\bibitem{Burkard+Cela+Pardalos+Pitsoulis:1998}
Burkard, R.E., \c{C}ela, E., Pitsoulis, L.: The quadratic assignment problem.
\newblock In: Handbook of Combinatorial Optimization, Computer-aided chemical
  engineering, pp. 241--339. Kluwer Academic Publishers, Dordrecht (1998)

\bibitem{ceria-soares:1999}
Ceria, S., Soares, J.: Convex programming for disjunctive convex optimization.
\newblock Mathematical Programming \textbf{86}(3), 595--614 (1999).
\newblock \doi{10.1007/s101070050106}

\bibitem{cezik-iyengar:2005}
{\c C}ezik, M.T., Iyengar, G.: Cuts for mixed 0-1 conic programming.
\newblock Mathematical Programming \textbf{104}(1), 179--202 (2005).
\newblock \doi{10.1007/s10107-005-0578-3}

\bibitem{ChoiLamReznick:1995}
Choi, M.D., Lam, T.Y., Reznick, B.: Sums of squares of real polynomials.
\newblock Proceedings of symposia in pure mathematics \textbf{58}(2), 103--126
  (1995)

\bibitem{cook-hartmann-kannan-mcdiarmid-1992}
Cook, W.J., Hartmann, M.E., Kannan, R., McDiarmid, C.: On integer points in
  polyhedra.
\newblock Combinatorica \textbf{12}(1), 27--37 (1992)

\bibitem{gamsdicopt}
Corp., G.D.: {DICOPT}.
\newblock {\tt www.gams.com/dd/docs/solvers/dicopt.pdf}

\bibitem{DeloeraHemmeckeKoeppeWeismantel06}
De~Loera, J.A., Hemmecke, R., K{\"o}ppe, M., Weismantel, R.: {FPTAS} for
  mixed-integer polynomial optimization with a fixed number of variables.
\newblock In: 17th ACM-SIAM Symposium on Discrete Algorithms, pp. 743--748
  (2006)

\bibitem{deloera-hemmecke-koeppe-weismantel:intpoly-fixeddim}
De~Loera, J.A., Hemmecke, R., K{\"o}ppe, M., Weismantel, R.: Integer polynomial
  optimization in fixed dimension.
\newblock Mathematics of Operations Research \textbf{31}(1), 147--153 (2006)

\bibitem{deloera-hemmecke-koeppe-weismantel:mixedintpoly-fixeddim-fullpaper}
De~Loera, J.A., Hemmecke, R., K{\"o}ppe, M., Weismantel, R.: {FPTAS} for
  optimizing polynomials over the mixed-integer points of polytopes in fixed
  dimension.
\newblock Mathematical Programming, Series~A \textbf{118}, 273--290 (2008).
\newblock \doi{10.1007/s10107-007-0175-8}

\bibitem{DeLoera+Hemmecke+Onn+Weismantel:08}
De~Loera, J.A., Hemmecke, R., Onn, S., Weismantel, R.: N-fold integer
  programming.
\newblock Disc. Optim., to appear  (2008)

\bibitem{DeLoera+Onn:2004}
De~Loera, J.A., Onn, S.: The complexity of three-way statistical tables.
\newblock SIAM Journal of Computing \textbf{33}, 819--836 (2004)

\bibitem{DeLoera+Onn:2006a}
De~Loera, J.A., Onn, S.: All linear and integer programs are slim 3-way
  transportation programs.
\newblock SIAM Journal of Optimization \textbf{17}, 806--821 (2006)

\bibitem{DeLoera+Onn:2006b}
De~Loera, J.A., Onn, S.: Markov bases of three-way tables are arbitrarily
  complicated.
\newblock Journal of Symbolic Computation \textbf{41}, 173--181 (2006)

\bibitem{Drewes}
Drewes, S., Ulbrich, S.: Mixed integer second order cone programming.
\newblock IMA Hot Topics Workshop, Mixed-Integer Nonlinear Optimization:
  Algorithmic Advances and Applications, November 17-21, 2008

\bibitem{MR866413}
Duran, M.A., Grossmann, I.E.: An outer-approximation algorithm for a class of
  mixed-integer nonlinear programs.
\newblock Math. Programming \textbf{36}(3), 307--339 (1986)

\bibitem{MR918874}
Duran, M.A., Grossmann, I.E.: Erratum: ``{A}n outer-approximation algorithm for
  a class of mixed-integer nonlinear programs'' [{M}ath.\ {P}rogramming {\bf
  36} (1986), no.\ 3, 307--339].
\newblock Math. Programming \textbf{39}(3), 337 (1987)

\bibitem{filmintNeos}
FilMINT: {\tt www-neos.mcs.anl.gov/neos/solvers/minco:FilMINT/AMPL.html}

\bibitem{FrangGent1}
Frangioni, A., Gentile, C.: Perspective cuts for a class of convex 0-1 mixed
  integer programs.
\newblock Math. Program. \textbf{106}(2, Ser. A), 225--236 (2006)

\bibitem{FrangGent3}
Frangioni, A., Gentile, C.: A computational comparison of reformulations of the
  perspective relaxation: {SOCP} vs. cutting planes.
\newblock Oper. Res. Lett. \textbf{37}(3), 206--210 (2009)

\bibitem{GareyJohnson79}
Garey, M.R., Johnson, D.S.: Computers and Intractability: A Guide to the Theory
  of {NP}-completeness.
\newblock W. H. Freeman and Company, New York, NY (1979)

\bibitem{MR0327310}
Geoffrion, A.M.: Generalized {B}enders decomposition.
\newblock J. Optimization Theory Appl. \textbf{10}, 237--260 (1972)

\bibitem{GoemansWilliamson95b}
Goemans, M.X., Williamson, D.P.: Improved approximation algorithms for maximum
  cut and satisfiability problems using semidefinite programming.
\newblock Journal of the ACM \textbf{42}, 1115--1145 (1995)

\bibitem{GoldLiuWang}
Goldfarb, D., Liu, S.C., Wang, S.Y.: A logarithmic barrier function algorithm
  for quadratically constrained convex quadratic programming.
\newblock SIAM J. Optim. \textbf{1}(2), 252--267 (1991)

\bibitem{Gom58b}
Gomory, R.E.: An algorithm for integer solutions to linear programs.
\newblock Princeton IBM Mathematics Research Project, Technical Report No. 1,
  Princeton University (November 17, 1958)

\bibitem{Gom58}
Gomory, R.E.: Outline of an algorithm for integer solutions to linear programs.
\newblock Bulletin of the American Mathematical Society \textbf{64}, 275--278
  (1958)

\bibitem{Graver:75}
Graver, J.E.: On the foundations of linear and integer linear programming {I}.
\newblock Mathematical Programming \textbf{8}, 207--226 (1975)

\bibitem{GroetschelLovaszSchrijver88}
Gr{\"o}tschel, M., Lov{\'a}sz, L., Schrijver, A.: Geometric Algorithms and
  Combinatorial Optimization.
\newblock Springer, Berlin, Germany (1988)

\bibitem{GunlukLind}
G{\"u}nl{\"u}k, O., Linderoth, J.: Perspective relaxation of mixed integer
  nonlinear programs with indicator variables.
\newblock In: A.~Lodi, A.~Panconesi, G.~Rinaldi (eds.) Integer Programming and
  Combinatorial Optimization 2008 -- Bertinoro, Italy, \emph{Lecture Notes in
  Computer Science}, vol. 5035, pp. 1--16. Springer Berlin / Heidelberg (2008)

\bibitem{GunLind2}
G{\"u}nl{\"u}k\phantom{}, O., Linderoth, J.: Perspective reformulations of
  mixed integer nonlinear programs with indicator variables.
\newblock Optimization Technical Report, ISyE Department, University of
  Wisconsin-Madison (June 20, 2008)

\bibitem{MR878885}
Gupta, O.K., Ravindran, A.: Branch and bound experiments in convex nonlinear
  integer programming.
\newblock Management Sci. \textbf{31}(12), 1533--1546 (1985)

\bibitem{hartmann-1989-thesis}
Hartmann, M.E.: Cutting planes and the complexity of the integer hull.
\newblock Phd thesis, Cornell University, Department of Operations Research and
  Industrial Engineering, Ithaca, NY (1989)

\bibitem{Hastad:inapprox97}
H{\aa}stad, J.: Some optimal inapproximability results.
\newblock In: Proceedings of the 29th Symposium on the Theory of Computing
  (STOC), pp. 1--10. ACM (1997)

\bibitem{heinz-2005:integer-quasiconvex}
Heinz, S.: Complexity of integer quasiconvex polynomial optimization.
\newblock Journal of Complexity \textbf{21}, 543--556 (2005)

\bibitem{Hemmecke:2003b}
Hemmecke, R.: On the positive sum property and the computation of {G}raver test
  sets.
\newblock Math. Programming, {S}eries {B} \textbf{96}, 247--269 (2003)

\bibitem{Hemmecke+Koeppe+Weismantel:08}
Hemmecke, R., K\"{o}ppe, M., Weismantel, R.: Oracle-polynomial time convex
  mixed-integer minimization.
\newblock Manuscript  (2008)

\bibitem{Hemmecke+Onn+Weismantel:08}
Hemmecke, R., Onn, S., Weismantel, R.: A polynomial oracle-time algorithm for
  convex integer minimization.
\newblock Manuscript  (2008)

\bibitem{Hiriart-Urruty-Lemarechal:93:part-two}
Hiriart-Urruty, J.B., Lemaréchal, C.: Convex analysis and minimization
  algorithms II: Advanced theory and bundle methods.
\newblock Grundlehren der Mathematischen Wissenschaften. 306. Berlin: Springer-
  Verlag (1993)

\bibitem{Hosten+Sullivant:07}
Ho{\c{s}}ten, S., Sullivant, S.: A finiteness theorem for {M}arkov bases of
  hierarchical models.
\newblock J. Combin. Theory Ser. A \textbf{114}(2), 311--321 (2007)

\bibitem{cplex}
{Ilog-Cplex}: {{\tt www.ilog.com/products/cplex}} ({v. 10.1})

\bibitem{JMW}
Jach, M., Michaels, D., Weismantel, R.: The convex envelope of (n-1)-convex
  functions.
\newblock SIAM Journal on Optimization (to appear)  (2008)

\bibitem{JacobiPrestel01}
Jacobi, T., Prestel, A.: Distinguished representations of strictly positive
  polynomials.
\newblock J. Reine Angew. Math. \textbf{532}, 223--235 (2001)

\bibitem{jeroslow-1973:quadratic-ip-uncomputable}
Jeroslow, R.G.: There cannot be any algorithm for integer programming with
  quadratic constraints.
\newblock Operations Research \textbf{21}(1), 221--224 (1973)

\bibitem{jones-1982}
Jones, J.P.: Universal diophantine equation.
\newblock Journal of Symbolic Logic \textbf{47}(3), 403--410 (1982)

\bibitem{kelley1960}
Kel{ley, Jr.}, J.E.: The cutting-plane method for solving convex programs.
\newblock Journal of the Society for Industrial and Applied Mathematics
  \textbf{8}(4), 703--712 (1960)

\bibitem{khachiyan:1983:polynomial-programming}
Khachiyan, L.G.: Convexity and complexity in polynomial programming.
\newblock In: Z.~Ciesielski, C.~Olech (eds.) Proceedings of the International
  Congress of Mathematicians, August 16--24, 1983, Warszawa, pp. 1569--1577.
  North-Holland, New York (1984)

\bibitem{khachiyan-porkolab:00}
Khachiyan, L.G., Porkolab, L.: {Integer optimization on convex semialgebraic
  sets.}
\newblock Discrete and Computational Geometry \textbf{23}(2), 207--224 (2000)

\bibitem{koeppe:irrational-barvinok}
K{\"o}ppe, M.: A primal {B}arvinok algorithm based on irrational
  decompositions.
\newblock SIAM Journal on Discrete Mathematics \textbf{21}(1), 220--236 (2007).
\newblock \doi{10.1137/060664768}

\bibitem{koeppe-verdoolaege:parametric}
K\"oppe, M., Verdoolaege, S.: Computing parametric rational generating
  functions with a primal {B}arvinok algorithm.
\newblock The Electronic Journal of Combinatorics \textbf{15}, 1--19 (2008).
\newblock \#R16

\bibitem{Lasserre01}
Lasserre, J.B.: Global optimization with polynomials and the problem of
  moments.
\newblock SIAM Journal on Optimization \textbf{11}, 796--817 (2001)

\bibitem{Laurent01}
Laurent, M.: A comparison of the {S}herali--{A}dams, {L}ov{\'a}sz--{S}chrijver
  and {L}asserre relaxations for 0-1 programming.
\newblock Mathematics of Operations Research \textbf{28}(3), 470--496 (2003)

\bibitem{insitu}
Lee, J.: In situ column generation for a cutting-stock problem.
\newblock Computers \& Operations Research \textbf{34}(8), 2345--2358 (2007)

\bibitem{Lee+Onn+Weismantel:08b}
Lee, J., Onn, S., Weismantel, R.: Nonlinear optimization over a weighted
  independence system.
\newblock IBM Research Report RC24513  (2008)

\bibitem{Lee+Onn+Weismantel:08}
Lee, J., Onn, S., Weismantel, R.: On test sets for nonlinear integer
  maximization.
\newblock Operations Research Letters \textbf{36}, 439–--443 (2008)

\bibitem{Lenstra83}
Lenstra, H.W.: Integer programming with a fixed number of variables.
\newblock Mathematics of Operations Research \textbf{8}, 538--548 (1983)

\bibitem{minlpbb}
Leyffer, S.: User manual for {MINLP\_BB}.
\newblock Tech. rep., University of Dundee, UK (1999)

\bibitem{convenvbook}
Liberti, L.: Comparison of convex relaxations for monomials of odd degree.
\newblock In: I.~Tseveendorj, P.~Pardalos, R.~Enkhbat (eds.) Optimization and
  Optimal Control. World Scientific (2003)

\bibitem{l_and_costas}
Liberti, L., Pantelides, C.: Convex envelopes of monomials of odd degree.
\newblock Journal of Global Optimization \textbf{25}, 157--168 (2003)

\bibitem{Lin2005}
Linderoth, J.: A simplicial branch-and-bound algorithm for solving
  quadratically constrained quadratic programs.
\newblock Mathematical Programming \textbf{103}(2), 251--282 (2005)

\bibitem{LVBL}
Lobo, M.S., Vandenberghe, L., Boyd, S., Lebret, H.: Applications of
  second-order cone programming.
\newblock Linear Algebra Appl. \textbf{284}(1-3), 193--228 (1998).
\newblock ILAS Symposium on Fast Algorithms for Control, Signals and Image
  Processing (Winnipeg, MB, 1997)

\bibitem{loqo}
{LOQO}: {{\tt www.princeton.edu/$\sim$rvdb}} ({v. 4.05})

\bibitem{lovasz-scarf:92}
Lov{\'a}sz, L., Scarf, H.E.: The generalized basis reduction algorithm.
\newblock Mathematics of Operations Research \textbf{17}(3), 751--764 (1992)

\bibitem{matiyasevich-1970}
Matiyasevich, {\relax{Yu}}.V.: Enumerable sets are diophantine.
\newblock Doklady Akademii Nauk SSSR \textbf{191}, 279--282 (1970).
\newblock (Russian); English translation, Soviet Mathematics Doklady, vol. 11
  (1970), pp. 354--357

\bibitem{matiyasevich-1993}
Matiyasevich, {\relax{Yu}}.V.: Hilbert's tenth problem.
\newblock The MIT Press, Cambridge, MA, USA (1993)

\bibitem{mccormick}
McCormick, G.: Computability of global solutions to factorable nonconvex
  programs: Part i --- convex underestimating problems.
\newblock Mathematical Programming \textbf{10}, 146--175 (1976)

\bibitem{meyer1}
Meyer, C., Floudas, C.: Trilinear monomials with mixed sign domains: Facets of
  the convex and concave envelopes.
\newblock Journal of Global Optimization \textbf{29}, 125--155 (2004)

\bibitem{mosek}
{MOSEK}: {{\tt www.mosek.com}} ({v. 5.0})

\bibitem{Mulmuley+Vazirani+Vazirani:87}
Mulmuley, K., Vazirani, U., Vazirani, V.: Matching is as easy as matrix
  inversion.
\newblock Combinatorica \textbf{7}, 105--113 (1987)

\bibitem{Onn+Rothblum:04}
Onn, S., Rothblum, U.G.: Convex combinatorial optimization.
\newblock Disc. Comp. Geom. \textbf{32}, 549--566 (2004)

\bibitem{Papadimitriou+Yanakakis:82}
Papadimitriou, C., Yanakakis, M.: The complexity of restricted spanning tree
  problems.
\newblock Journal of the Association for Computing Machinery \textbf{29},
  285--309 (1982)

\bibitem{Pardalos+Rendl+Wolkowicz:1994}
Pardalos, P.M., Rendl, F., Wolkowicz, H.: The quadratic assignment problem: A
  survey and recent developments.
\newblock In: P.M. Pardalos, H.~Wolkowicz (eds.) Quadratic Assignment and
  Related Problems, \emph{{DIMACS} Series in Discrete Mathematics and
  Theoretical Computer Science}, vol.~16, pp. 1--42. American Mathematical
  Society (1994)

\bibitem{Parrilo03}
Parrilo, P.A.: Semidefinite programming relaxations for semialgebraic problems.
\newblock Mathematical Programming \textbf{96}, 293--320 (2003)

\bibitem{PBW2008}
P\"orn, R., Bj\"ork, K.M., Westerlund, T.: Global solution of optimization
  problems with signomial parts.
\newblock Discrete Optimization \textbf{5}, 108--120 (2008)

\bibitem{Putinar93}
Putinar, M.: Positive polynomials on compact semi-algebraic sets.
\newblock Indiana University Mathematics Journal \textbf{42}, 969--984 (1993)

\bibitem{QG1992}
Quesada, I., Grossmann, I.: An {LP/NLP} based branch and bound algorithm for
  convex {MINLP} optimization problems.
\newblock Computers Chem. Engng. \textbf{16}, 937--947 (1992)

\bibitem{RRW}
Rendl, F., Rinaldi, G., Wiegele, A.: A branch and bound algorithm for max-cut
  based on combining semidefinite and polyhedral relaxations.
\newblock In: M.~Fischetti, D.~Williamson (eds.) Integer Programming and
  Combinatorial Optimization 2007 -- Ithaca, New York, \emph{Lecture Notes in
  Computer Science}, vol. 4513, pp. 295--309. Springer Berlin / Heidelberg
  (2007)

\bibitem{rendl-rinaldi-wiegele-2008}
Rendl, F., Rinaldi, G., Wiegele, A.: Solving max-cut to optimality by
  intersecting semidefinite and polyhedral relaxations.
\newblock Tech. rep., Alpen-Adria-Universit\"at Klagenfurt, Inst. f. Mathematik
  (2008)

\bibitem{Renegar:1992:CCGc}
Renegar, J.: On the computational complexity and geometry of the first-order
  theory of the reals. part {III}: Quantifier elimination.
\newblock Journal of Symbolic Computation \textbf{13}(3), 329--352 (1992)

\bibitem{Renegar:1992:Approximating}
Renegar, J.: On the computational complexity of approximating solutions for
  real algebraic formulae.
\newblock SIAM Journal on Computing \textbf{21}(6), 1008--1025 (1992)

\bibitem{Sahin1996}
Sahinidis, N.: {BARON}: {A} general purpose global optimization software
  package.
\newblock Journal of Global Optimization \textbf{8}, 201--205 (1996)

\bibitem{Santos+Sturmfels:03}
Santos, F., Sturmfels, B.: Higher {L}awrence configurations.
\newblock J. Combin. Theory Ser. A \textbf{103}, 151--164 (2003)

\bibitem{SBL1}
Saxena, A., Bonami, P., Lee{}, J.: Disjunctive cuts for non-convex mixed
  integer quadratically constrained programs.
\newblock In: A.~Lodi, A.~Panconesi, G.~Rinaldi (eds.) Integer Programming and
  Combinatorial Optimization 2008 -- Bertinoro, Italy, \emph{Lecture Notes in
  Computer Science}, vol. 5035, pp. 17–--33. Springer Berlin / Heidelberg
  (2008)

\bibitem{SBL2}
Saxena\phantom{}, A., Bonami, P., Lee, J.: Convex relaxations of non-convex
  mixed integer quadratically constrained programs: Extended formulations
  (2008).
\newblock IBM Research Report RC24621

\bibitem{SBL3}
Saxena\phantom{\strut}, A., Bonami, P., Lee, J.: Convex relaxations of
  non-convex mixed integer quadratically constrained programs: Projected
  formulations (2008).
\newblock IBM Research Report RC24695

\bibitem{sdpt3}
{SDPT3}: {{\tt www.math.nus.edu.sg/$\sim$mattohkc/sdpt3.html}} ({v. 4.0
  (beta)})

\bibitem{Seboe:90}
Seb{\"o}, A.: Hilbert bases, {C}aratheodory's {T}heorem and combinatorial
  optimization.
\newblock In: Proc. of the IPCO conference, Waterloo, Canada, pp. 431--455
  (1990)

\bibitem{sedumiJ}
{SeDuMi}: {{\tt sedumi.mcmaster.ca}} ({v. 1.1})

\bibitem{Shor87}
Shor, N.Z.: An approach to obtaining global extremums in polynomial
  mathematical programming.
\newblock Kibernetika \textbf{52}, 102--106 (1987)

\bibitem{smith}
Smith, E.: On the optimal design of continuous processes.
\newblock Ph.D. thesis, Imperial College of Science, Technology and Medicine,
  University of London (1996)

\bibitem{SP1999}
Smith, E., Pantelides, C.: A symbolic reformulation/spatial branch-and-bound
  algorithm for the global optimisation of nonconvex {MINLPs}.
\newblock Computers \& Chemical Engineering \textbf{23}, 457--478 (1999)

\bibitem{stubbs-mehrotra:1999}
Stubbs, R.A., Mehrotra, S.: A branch-and-cut method for 0-1 mixed convex
  programming.
\newblock Mathematical Programming \textbf{86}(3), 515--532 (1999).
\newblock \doi{10.1007/s101070050103}

\bibitem{Sturmfels96}
Sturmfels, B.: {G}r{\"o}bner Bases and Convex Polytopes.
\newblock American Mathematical Society, Providence, RI (1996)

\bibitem{tarasov-khachiyan-1980}
Tarasov, S.P., Khachiyan, L.G.: Bounds of solutions and algorithmic complexity
  of systems of convex diophantine inequalities.
\newblock Soviet Math. Doklady \textbf{22}(3), 700--704 (1980)

\bibitem{tawarmalani3}
Tawarmalani, M., Sahinidis, N.: Convex extensions and envelopes of
  semi-continuous functions.
\newblock Mathematical Programming \textbf{93}(2), 247--263 (2002)

\bibitem{TawS2002}
Tawarmalani, M., Sahinidis, N.: Convexification and global optimization in
  continuous and mixed-integer nonlinear programming: Theory, algorithms,
  software and applications, \emph{Nonconvex Optimization and Its
  Applications}, vol.~65.
\newblock Kluwer Academic Publishers, Dordrecht (2002)

\bibitem{TawS2004}
Tawarmalani, M., Sahinidis, N.: Global optimization of mixed-integer nonlinear
  programs: A theoretical and computational study.
\newblock Mathematical Programming \textbf{99}(3), 563--591 (2004)

\bibitem{tawarmalani-sahinidis:01}
Tawarmalani, M., Sahinidis, N.V.: {S}emidefinite {R}elaxations of {F}ractional
  {P}rograms via {N}ovel {C}onvexification {T}echniques.
\newblock Journal of Global Optimization \textbf{20}, 137--158 (2001)

\bibitem{4ti2}
4ti2 team: 4ti2 -- a software package for algebraic, geometric and
  combinatorial problems on linear spaces.
\newblock Available at {\url{http://www.4ti2.de}}

\bibitem{Thomas:01}
Thomas, R.R.: Algebraic methods in integer programming.
\newblock In: C.~Floudas, P.~Pardalos (eds.) Encyclopedia of Optimization.
  Kluwer Academic Publishers, Dordrecht (2001)

\bibitem{1057788}
Vandenbussche, D., Nemhauser, G.L.: A branch-and-cut algorithm for nonconvex
  quadratic programs with box constraints.
\newblock Mathematical Programming \textbf{102}(3), 559--575 (2005)

\bibitem{1057787}
Vandenbussche, D., Nemhauser, G.L.: A polyhedral study of nonconvex quadratic
  programs with box constraints.
\newblock Mathematical Programming \textbf{102}(3), 531--557 (2005)

\bibitem{WesLun2005}
Westerlund, T., Lundqvist, K.: {Alpha-ECP}, version 5.101: An interactive
  {MINLP}-solver based on the extended cutting plane method.
\newblock Tech. Rep. 01-178-A, Process Design Laboratory at Abo Akademi
  University (Updated version of 2005-10-21)

\bibitem{WesPet1995}
Westerlund, T., Pettersson, F.: An extended cutting plane method for solving
  convex {MINLP} problems.
\newblock Computers and Chemical Engineering \textbf{19(Suppl.)}, S131--S136
  (1995)

\bibitem{WesPorn2002}
Westerlund, T., P{\"o}rn, R.: Solving pseudo-convex mixed integer optimization
  problems by cutting plane techniques.
\newblock Optimization and Engineering \textbf{3}, 253--280 (2002)

\end{thebibliography}

\end{document}